\documentclass{easychair}

\title{Positive Modal Logic Beyond Distributivity}
\author{Nick Bezhanishvili\inst{1}$^{\text{,\Letter}}$, Anna Dmitrieva\inst{2}, Jim de Groot\inst{3}, Tommaso Moraschini\inst{4}}
\institute{
University of Amsterdam, \texttt{n.bezhanishvili@uva.nl}
\and
University of East Anglia, \texttt{a.dmitrieva@uea.ac.uk}
\and
The Australian National University, Canberra, Ngunnawal Country, Australia \texttt{jim.degroot@anu.edu.au}
\and
Departament de Filosofia, Facultat de Filosofia, Universitat de Barcelona (UB), Carrer Montalegre, 6, 08001 Barcelona, Spain, \texttt{tommaso.moraschini@ub.edu}}
\date{}

\authorrunning{}
\titlerunning{}

\usepackage{amsmath}
\usepackage{amssymb}
\usepackage{amsthm}
  \theoremstyle{definition}
  \swapnumbers
    \newtheorem{para}{}[section]
    \newtheorem{definition}[para]{Definition}
    \newtheorem{examplex}[para]{Example}
    \newtheorem{remarkx}[para]{Remark}
  \theoremstyle{theorem}
    \newtheorem{lemma}[para]{Lemma}
    \newtheorem{corollary}[para]{Corollary}
    \newtheorem{theorem}[para]{Theorem}
    \newtheorem{proposition}[para]{Proposition}

\newenvironment{example}
  {\pushQED{\qed}\examplex}
  {\popQED\endexamplex}
\newenvironment{remark}
  {\pushQED{\qed}\remarkx}
  {\popQED\endremarkx}

\usepackage{stmaryrd}
\usepackage{tikz-cd}
  \usetikzlibrary{arrows,arrows.meta}
  \tikzcdset{arrow style=tikz, diagrams={>=stealth'}}
  \usetikzlibrary{snakes}
  \tikzset{decoration={snake, amplitude=.2mm,segment length=2mm}}
\usepackage{color}
\usepackage{booktabs}
\usepackage{hyperref}
\usepackage{tcolorbox}
\usepackage{multicol}
\usepackage{enumitem}
\usepackage{float}
\usepackage{marvosym} % for the \Letter symbol

\usepackage[cal=dutchcal,scr=boondoxupr,scrscaled=1.050]{mathalfa}

\newcommand{\ms}[1]{\mathscr{#1}}
\newcommand{\topo}[1]{\mathbb{#1}}  % for topological spaces/Priestley spaces
\newcommand{\cat}[1]{\mathsf{#1}}   % for categories
\newcommand{\fun}[1]{\mathcal{#1}}  % for functors
\newcommand{\mo}[1]{\mathfrak{#1}}  % for models
\newcommand{\amo}[1]{\mathscr{#1}}  % for algebraic models
\newcommand{\lan}[1]{\mathbf{#1}}   % for languages
\renewcommand{\log}[1]{\boldsymbol{\mathcal{#1}}}   % for logics

\renewcommand{\iff}{\quad\text{iff}\quad}

\renewcommand{\phi}{\varphi}

\usepackage{stackengine}
\usepackage{scalerel}

\mathchardef\hyphen="2D
\DeclareMathOperator{\op}{op}

\DeclareMathOperator{\Prop}{Prop}
\DeclareMathOperator{\st}{st}
\DeclareMathOperator{\so}{so}
\renewcommand{\th}{\operatorname{th}}

\newcommand{\cp}{\trianglelefteqslant}

\newcommand{\isfil}{\mathsf{isfil}}

\newcommand{\abovemeet}{\mathsf{abovemeet}}
\newcommand{\ISFIL}{\operatorname{ISFIL}}

\newcommand{\AT}{\operatorname{AT}}
\newcommand{\REL}{\operatorname{REL}}

\newcommand{\BOXAT}{\operatorname{BOX\hyphen AT}}

\newcommand{\fil}{\fun{F}}
\newcommand{\bfil}{\fun{F}_b}
\newcommand{\cfil}{\fun{F}_{clp}}
\newcommand{\tfil}{\fun{F}_{top}}
\newcommand{\kfil}{\fun{F}_k}
\newcommand{\pfil}{\fun{F}_p}

\newcommand{\sfil}{\fun{F}_{sat}}

\newcommand{\minf}{\langle\mathsf{inf}\rangle}
\newcommand{\msup}{\langle\mathsf{sup}\rangle}

\newcommand{\tVdash}{\Vdash_{\!t}}

\newcommand{\FOL}{\lan{FOL}}
\newcommand{\SOL}{\lan{SOL}}

\newcommand{\llb}{\llbracket}
\newcommand{\rrb}{\rrbracket}
\newcommand{\llp}{\llparenthesis\kern.2ex}
\newcommand{\rrp}{\kern.2ex\rrparenthesis}

\newcommand{\cdash}[1][]{\mathrel{\kern.1ex\text{%
  \tikz[baseline=-.81ex, line width=.1ex, line cap=round, scale=1.1]
    {\draw (0ex,-.75ex) -- (0ex,.75ex);
     \draw (0ex,0ex) -- (1ex,0ex);
     \draw (-.4ex,0ex) arc(180:235:.91ex);
     \draw (-.4ex,0ex) arc(180:125:.91ex);}$_{#1}$}}\kern.1ex}

\newcommand{\clwidth}{.12ex}
\newcommand{\goodbox}{\hspace{.2ex}\text{%
  \tikz[baseline=-.6ex, rounded corners=.01ex, line width=\clwidth]
    {\draw (-.6ex,-.6ex) rectangle (.6ex,.6ex);}}\kern.2ex}
\newcommand{\gooddiamond}{\hspace{.2ex}\text{%
  \tikz[baseline=-.6ex, rounded corners=.01ex, rotate=45, line width=\clwidth]
    {\draw (-.5ex,-.5ex) rectangle (.5ex,.5ex);}}\kern.2ex}
\renewcommand{\Box}{\goodbox}
\renewcommand{\Diamond}{\gooddiamond}
\newcommand{\bbox}{\hspace{.2ex}\text{%
  \tikz[baseline=-.6ex, rounded corners=.01ex, line width=\clwidth]
    {\draw[fill=black] (-.6ex,-.6ex) rectangle (.6ex,.6ex);}}\kern.2ex}
\newcommand{\bdiamond}{\hspace{.2ex}\text{%
  \tikz[baseline=-.6ex, rounded corners=.01ex, rotate=45, line width=\clwidth]
    {\draw[fill=black] (-.5ex,-.5ex) rectangle (.5ex,.5ex);}}\kern.2ex}
\newcommand{\dbox}{\hspace{.2ex}\text{%
  \tikz[baseline=-.6ex, rounded corners=.01ex, line width=\clwidth]
    {\draw (-.6ex,-.6ex) rectangle (.6ex,.6ex);
     \draw[fill=black] (0,0) circle(.12ex);}}\kern.2ex}
\newcommand{\ddiamond}{\hspace{.2ex}\text{%
  \tikz[baseline=-.6ex, rounded corners=.01ex, rotate=45, line width=\clwidth]
    {\draw (-.5ex,-.5ex) rectangle (.5ex,.5ex);
     \draw[fill=black] (0,0) circle(.12ex);}}\kern.2ex}
\newcommand{\lbox}{\hspace{.2ex}\text{%
  \tikz[baseline=-.6ex, rounded corners=.01ex, line width=\clwidth]
    {\draw (-.6ex,-.6ex) rectangle (.6ex,.6ex);
     \draw (0,-.6ex) -- (0,.6ex);}}\kern.2ex}
\newcommand{\ldiamond}{\hspace{.2ex}\text{%
  \tikz[baseline=-.6ex, rounded corners=.01ex, rotate=45, line width=\clwidth]
    {\draw (-.5ex,-.5ex) rectangle (.5ex,.5ex);
     \draw (-.5ex,-.5ex) -- (.5ex,.5ex);}}\kern.2ex}

\newcommand{\hlbox}{\hspace{.2ex}\text{%
  \tikz[baseline=-.6ex, rounded corners=.01ex, line width=\clwidth]
    {\draw (-.6ex,-.6ex) rectangle (.6ex,.6ex);
     \draw (-.6ex,0) -- (.6ex,0);}}\kern.2ex}
\newcommand{\hldiamond}{\hspace{.2ex}\text{%
  \tikz[baseline=-.6ex, rounded corners=.01ex, rotate=45, line width=\clwidth]
    {\draw (-.5ex,-.5ex) rectangle (.5ex,.5ex);
     \draw (.5ex,-.5ex) -- (-.5ex,.5ex);}}\kern.2ex}

\newcommand{\customtri}{\hspace{.2ex}\text{%
  \tikz[baseline=-.6ex, rounded corners=.01ex, line width=\clwidth]
    {\draw (0,-.6ex) -- (-.65ex,-.6ex) -- (0,.65ex) -- (.65ex,-.6ex) -- (0,-.6ex);}}\kern.2ex}
\newcommand{\lmon}{\hspace{.2ex}\text{%
  \tikz[baseline=-.6ex, rounded corners=.01ex, line width=\clwidth]
    {\draw (0,-.6ex) -- (-.65ex,-.6ex) -- (0,.65ex) -- (.65ex,-.6ex) -- (0,-.6ex);
     \draw (0,0) -- (0,.65ex);}}\kern.2ex}

\usepackage{fontawesome}
\newcommand{\Star}{\scalebox{.75}{\text{\faStarO}}}

\newcommand{\uparr}{{\uparrow}}

\newcommand{\MSL}{\cat{MSL}}
\newcommand{\HMS}{\cat{HMS}}
\newcommand{\Lat}{\cat{Lat}}
\newcommand{\LSpace}{\cat{LSpace}}
\newcommand{\LFrm}{\cat{LFrm}}
\newcommand{\DL}{\cat{DL}}
\newcommand{\PUP}{\cat{PUP}}

\DeclareMathOperator{\FiltComp}{\fun{fe}}
\DeclareMathOperator{\IdComp}{\fun{ie}}
\DeclareMathOperator{\PiComp}{\Pi_1}
\DeclareMathOperator{\Hom}{Hom}

\newcommand{\rbox}[1]{[#1]}
\newcommand{\rdiamond}[1]{\langle #1 \rangle}
\newcommand{\genFil}{\mathbin{\triangledown}}
\newcommand{\bigGenFil}{\operatorname{\bigtriangledown}}
\newcommand{\ftop}{1}
\newcommand{\fleq}{\preccurlyeq}
\newcommand{\fmeet}{\curlywedge}
\newcommand{\bigfmeet}{\bigcurlywedge}
\newcommand{\fjoin}{\curlyvee}

\newcommand{\negphantom}[1]{\ifmmode\settowidth{\dimen0}{$#1$}\else\settowidth{\dimen0}{#1}\fi\hspace*{-\dimen0}}

\begin{document}

\maketitle

\begin{abstract}
  \noindent
  We develop a duality for (modal) lattices that need not be distributive,
  and use it to study positive (modal) logic beyond distributivity,
  which we call weak positive (modal) logic.
  This duality builds on the Hofmann, Mislove and Stralka duality for
  meet-semilattices. 
  We introduce the notion of $\Pi_1$-persistence and show that every
  weak positive modal logic is $\Pi_1$-persistent.
  This approach leads to a new relational semantics for weak positive modal logic,
  for which we 
  prove an analogue of Sahlqvist correspondence result.%
  \footnote{This paper is partially based on the Master's thesis
  \cite{Dmi21}.}
\end{abstract}



%\tableofcontents

%\clearpage
%================================================================================
\section{Introduction}

% 1. Dualities in Logic are useful and cool
Dualities between modal algebras and modal spaces on the one hand and Heyting algebras  and Esakia spaces on the other have been central to the study of modal and intermediate logics~\cite{BRV01, CZ97}. Indeed,  many important results such as Sahlqvist canonicity and correspondence~\cite{Sah75} can be understood through the lenses of duality techniques~\cite{SamVac89}. The  duality between modal algebras and modal spaces has been extended to a duality between modal distributive lattices
and modal Priestley spaces in \cite{Gol89,CelJan99}.
 This led to a Sahlqvist theory for the positive distributive modal logic introduced in   \cite{Dun95, GNV05, CelJan99}. 

% 2. There are many different dualities for lattices, but they have downsides
When the algebraic side of a duality consists of distributive lattice expansions, in the spatial side of the duality one often  works with the 
Priestley space \cite{Pri70, DP02} of all the prime filters of a given lattice. This is no longer the case when the base lattice is non-distributive. 
There are many extensions of dualities for Boolean algebras and distributive lattices to the setting of all lattices, e.g.~by Urquhart \cite{Urq78}, Hartonas  \cite{Hartonas19, HO19}, Gehrke and Van Gool~\cite{GvG14}, 
 Goldblatt~\cite{Gol20}, and Hartung~\cite{Har92}.
Each of these uses either a ternary relation, or two-sorted frames.
While these approaches  have proven fruitful and interesting, they are quite different from known dualities for propositional logics such as  
Stone and Priestley dualities. This makes it difficult to adjust existing tools and techniques
%\color{blue}, as well a logician's intuition, \color{black}
from distributive logics to non-distributive ones. 
%for these dualities for lattices.
%In all these cases one works with ternary relations ( or with two sets etc make more precise {\bf add later when we recall what they do}),

% 3. Here we give yet another duality, which is closely aligned to existing approaches
Hofmann, Mislove and Stralka (HMS)~\cite{HofMisStr74} developed a duality for 
(not necessarily distributive) meet-semilattices along the same lines of the  van Kampen-Pontryagin duality for locally compact abelian groups given in \cite{Roe74}. This was later restricted  to a duality for lattices by Jipsen and Moshier \cite{MosJip09}.
In this approach, the dual space of a lattice  is based not on the prime filters, but on all the (proper) filters. This %perspective
is closely related to  Holliday's  possibility semantics of modal logic \cite{Hol21} (see also \cite{HM22}) and to the choice-free duality for Boolean algebras in \cite{BH20}, which are also based on  spaces of all proper filters. A similar   approach was also developed for ortholattices by Goldblatt \cite{Gol75} and extended later by Bimbó \cite{Bimbo07}. 
%\color{red} I thought Bimbo was spelled as Bimbó (see also the bibliography etc.).\color{black}

% 4. Aim: study non-distributive positive logic. This motivates our choice of duality.
%    Usefulness is witnessed by Sahlqvist results.
The aim of this paper is to investigate positive modal logic that
is not necessarily distributive. We refer to this as
\emph{weak positive modal logic}.
It is a logic with the same language as positive modal logic,
i.e.~the negation- and implication-free fragment of classical modal logic,
which does not necessarily satisfy the distributivity axiom.
%modal lattice logic,
%which has the same language as positive modal logic (i.e.~the negation- and implication-free
%\color{blue}The aim of this paper is to investigate \emph{positive modal logics}, that is, modal logics in the positive language
%\[
%p \mid \top \mid \bot \mid \phi \wedge \phi \mid \phi \vee \phi
%              \mid \Box\phi \mid \Diamond\phi
%\]
% \color{blue} without necessarily assuming the distributivity axiom\color{black}.
%We introduce these logics and call them simply positive (modal) logics.

We study these logics via a duality that builds on HMS duality.
%It is similar to the duality given by Jipsen and Moshier, although we will 
%with  a slightly different setting than Jipsen and Moshier.
%Our choice of duality is based on its similarity to the  dualities used in distributive cases. We also demonstrated this by showing in this paper 
%that classic results such as Sahlqvist correspondence can be adapted from the distributive setting to a non-distributive one, for an appropriately 
%defined relational semantics.
%
% 5. More details about our approach.
 We recall that a Priestley space is a partially ordered compact space satisfying the Priestley separation axiom 
\begin{center}
$x\nleq y$ implies that there is a clopen upset $U$ such that $x\in U$ and $y\notin U$. 
\end{center}
These spaces provide  a  duality for bounded   distributive lattices,  which associates every Priestley space with the lattice of its clopen upsets. 
In the HMS duality,  one works with similar structures,  but the role of partially ordered compact spaces is played by
meet-semilattices with a compact topology and that of clopen upsets by clopen filters. \color{black} Then the HMS analogue of the Priestley separation axiom is 
\begin{center}
$x\nleq y$ implies that there is a clopen filter $U$ such that $x\in U$ and $y\notin U$. 
\end{center}
These spaces provide  a  duality for bounded meet-semilatices. 

  Our approach is analogous to the one of  Esakia duality for Heyting algebras.
  Recall that an Esakia space is a Priestley space where for every pair of
  clopen upsets $U$ and $V$ the Heyting implication $U \to V$ is also a clopen
  upset (this is sometimes formulated as the equivalent condition that 
  ${\downarrow}U$ is clopen for every clopen $U$)~\cite{Esa74, Esakia-book85}.
  In analogy with this, an HMS space is said to be a \emph{lattice space} if
  the join in the lattice of filters of every pair $U$ and $V$ of clopen filters
  (i.e.~$\{x \mid x\geq a\wedge b$ for some $a\in U$ and $b\in V\}$) is also a
  clopen filter. 

We extend this duality to modal  lattices in the signature with two unary modalities, $\Box$ and $\Diamond$. More precisely, by a modal lattice we understand a  lattice with a top element $\top$ and two modalities related via Dunn's axiom $\Diamond x \wedge \Box y
      \leq \Diamond(x \wedge y)$  \cite{Dun95} and satisfying the equations $\Box \top \thickapprox \Diamond \top \thickapprox \top$.  Furthermore, while $\Box$ will be assumed to distribute over finite meets, we require $\Diamond$ to be merely monotone. A similar phenomenon in the context of modal intuitionistic logic has been investigated in \cite{Koj12}. 
Despite the asymmetry between $\Box$ and $\Diamond$,   on the dual side these modalities
are interpreted via a binary relation in the standard way. 

%The two modalities are related via one of Dunn's duality axioms for distributive positive modal logic~\cite{Dun95}. 
%The non-standard interpretation of joins makes Dunn's other duality axiom unsuitable in our context. We restrict our attention  to the serial case, thus adding the seriality axiom,  which allows us to 
%obtain a duality between relational and algebraic semantics. This is similar to the distributive case and, in presence of seriality, uses only one of Dunn's axioms.

  This duality allows us to define a new relational %(Kripke-like)
  semantics for weak positive modal logics in which the analogue of a Kripke frame is
  a meet-semilattice with an extra relation.
  The meet gives rise to a partial order, so these frames can be viewed as
  bi-relational frames where the relations satisfy certain conditions.
  Formulae are interpreted as filters, disjunction is interpreted as the
  least filter generated by the interpretation of each disjunct,
  and modalities are interpreted in the standard way.
  This new semantics can be seen as a generalisation of the team semantics
  of~\cite{Hodges97} and of the modal information semantics of~\cite{vB22,Knu22}.

% 6. About canonical extensions?
Kripke semantics for intuitionistic and modal logics is tightly related to the theory of canonical extensions~\cite{JonTar51,GehHar01,GehJon94}. This is largely due to the fact that a formula is valid in the Kripke frame associated with a Heyting or modal algebra $A$ if and only if it is valid in the canonical extension of $A$. In our case, the role of canonical extensions is played by Gehrke and Priestley's $\Pi_1$-completions \cite{GehPri08}. This is because a formula is valid in the Kripke frame associated with a modal lattice $A$ by our duality iff it is valid in the $\Pi_1$-completion of $A$. Notably, the $\Pi_1$-completion of $A$ can be described concretely as the modal lattice of all filters of the lattice of filters of $A$ (or, equivalently, as the composition of the filter and the ideal completions of $A$).

 Our main results are  Sahlqvist-style preservation and correspondence results for weak positive modal logic with respect to this new semantics. 
 Using a duality technique  similar to that of  Sambin and Vaccaro \cite{SamVac89}, we show that every sequent is preserved by $\Pi_1$-completion. 
 Note that in the propositional setting this corresponds to the fact that every variety of lattices is closed under ideal completions and filter completions \cite{BirSack, BakHal74}.

% \blue{JG: currently we don't talk about ideals in the paper...}
We also prove an analogue of  the Sahlqvist correspondence result. In particular, we introduce Sahlqvist sequents in our language and show that very Sahlqvist sequent has a first-order correspondent. We also introduce the notion of $\Pi_1$-persistence for weak positive modal logics, which is a logical analogue for the corresponding class of algebra to be closed under $\Pi_1$-completions and show that every weak positive modal logic is $\Pi_1$-persistent. As a result every weak positive modal logic is complete with respect to our   relational semantics.  We point out that an alternative approach to Sahlqvist correspondence and canonicity for non-distributive logics has been undertaken in \cite{CP19}, although this perspective is based on canonical extensions and is, therefore, orthogonal to the one developed in this paper.

% 7. Modal extension and corresponding results

% 8. The bigger picture
With this paper we hope to lay a groundwork for a theory of weak positive modal logics. As discussed in the conclusion, there are many interesting directions for future research.
These include the study of logics that lie between non-distributive and distributive positive (modal) logic, deriving more results for the weak modal logic presented in this paper,
as well as extending weak positive logic with different types of modalities.

%================================================================================
\section{Preliminaries}\label{sec:lattices}

  We briefly recall a Stone-type duality for the category of meet-semilattices
  with top due to Hofmann, Mislove and Stralka~\cite{HofMisStr74}.
  We then restrict this to a duality for lattices,
  and show how it relates to various completions of lattices.

%--------------------------------------------------------------------------------
\subsection{Dual Adjunctions}

\begin{definition}
  By a \emph{semilattice} we mean a meet-semilattice with top.
  Every semilattice $(X, \top, \wedge)$ has an underlying partial
  order $\leq$ given by $x \leq y$ iff $x \wedge y = x$.
  A \emph{(semilattice) homomorphism} from $(X, \top, \wedge)$ to
  $(X', \top', \wedge')$ is a function $f : X \to X'$ such that
  $f(\top) = \top'$ and $f(x \wedge y) = f(x) \wedge' f(y)$ for all
  $x, y \in X$.
  We write $\MSL$ for the category of semilattices and homomorphisms.
  
  Similarly, by a \emph{lattice} we mean a bounded lattice,
  and lattice homomorphisms are assumed to preserve these bounds.
  We write $\Lat$ for the category of lattices and lattice homomorphisms.
  %\todo{Add lattices here? $\Lat$. All lattices are bounded in this paper.}
\end{definition}

%\paragraph{Upsets}
  If $(X, \leq)$ is a partial order (possibly coming from a semilattice
  $(X, \top, \wedge)$) and $a \subseteq X$ then we define the
  \emph{upward closure} of $a$ by\
  $\uparr a := \{ y \in X \mid x \leq y \text{ for some } x \in a \}$.
  The set $a$ is called \emph{upward closed} or an \emph{upset} if
  $\uparr a = a$.
  If $a = \{ x \}$ then we write $\uparr x$ instead of $\uparr \{ x \}$.
  The \emph{downward closure} and \emph{downsets} are defined similarly.
  
\begin{definition}
  A \emph{filter} $p$ of a semilattice $(X, \top, \wedge)$ is a nonempty
  upset $p \subseteq X$ that is closed under meets.
  It is called \emph{principal} if
  $p = \uparr x$ for some $x \in X$.
\end{definition}

  Filters of $(X, \top, \wedge)$ correspond bijectively
  to homomorphisms to the two-element semilattice $2 = \{ \top, * \}$:
  every filter $p$ yields a characteristic map
  $\chi_p : X \to 2$ given by $\chi_p(x) = \top$ iff $x \in p$,
  and conversely for every homomorphism $f : X \to 2$, $f^{-1}(\top)$
  is a filter.
  For every semilattice $(X, \top, \wedge)$,
  the collection $\fil(X, \top, \wedge)$ of filters forms a complete
  semilattice ordered by subset inclusion. It is then easy to see
  that the filter $X$ is the largest element in $\fil(X, \top, \wedge)$
  and that the greatest lower bound of a collection of filters is given
  by their intersection.
  Therefore it is also a (complete) lattice.
  The top and bottom element of $\fil(X, \top, \wedge)$
  are given by $X$ and $\{ \top \}$.
  Binary joins are given by
  $$
    p \genFil q = \uparr \{ x \wedge y \mid x \in p, y \in q \},
  $$
  and the join of any set $F \subseteq \fil(X, \top, \wedge)$
  can be given by
  $
    \bigGenFil F
      = \bigcup \{ p_1 \genFil \cdots \genFil p_n \mid p_1, \ldots, p_n \in F \}.
  $

  This assignment $\fil$ extends to a contravariant functor $\fil : \MSL \to \MSL$
  if we define
  its action on a homomorphism $f : (X, \top, \wedge) \to (X', \top', \wedge')$
  by $\fil f = f^{-1} : \fil(X', \top', \wedge') \to \fil(X, \top, \wedge)$.
  For a semilattice $(X, \top, \wedge)$, let
  $$
    \eta_{(X, \top, \wedge)} : (X, \top, \wedge) \to \fil\fil(X, \top, \wedge) : x \mapsto \{ p \in \fil(X, \top, \wedge) \mid x \in p \}.
  $$
  This yield a natural transformation
  $\eta : \fun{id}_{\MSL} \to \fil\fil$ that satisfies
  $\fil\eta \circ \eta_{\fil} = \fun{if}_{\fil}$.
  Therefore:

\begin{proposition}\label{prop:dual-adj-MSL}
  The functor $\fil$ then establishes a dual adjunction between
  $\MSL$ and $\MSL$ with both units given by $\eta$.
\end{proposition}

  In preparation for using semilattices as interpretation for weak positive logic (Section~\ref{sec:logic}),
  and for the duality for lattices that we derive in Theorem~\ref{thm:duality-lat-lspace},
  we restrict the functor $\fil : \MSL \to \MSL$ to
  functors $\Lat \to \LFrm$ and $\LFrm \to \Lat$.
%  While on objects these functors stay the same, we modify the morphisms
%  on both sides.
%
  One occurrence of $\MSL$ is restricted to $\Lat$,
  while the other occurrence is restricted to the category of semilattices
  and so-called L-morphisms.
  The situation is analogous to that for intuitionistic logic, where
  the functors establishing the dual adjunction between distributive lattices
  and posets restrict to contravariant functors between the categories of
  Heyting algebras and intuitionistic Kripke frames
  (see Figure~\ref{fig:int}).
  \begin{figure}[h]
  $$
    \begin{tikzcd}[row sep=large,
                   column sep=large,
                   every matrix/.append style={draw, inner ysep=11pt, inner xsep=6pt, rounded corners}]
      \MSL
            \arrow[r, shift left=0pt, bend left=10, "\fil"]
        & \MSL
            \arrow[l, shift left=0pt, bend left=10, "\bfil"] \\ [-.3em]
      \LFrm
            \arrow[u, >->]
            \arrow[r, shift left=0pt, bend left=10, "\fil"]
        & \cat{Lat}
            \arrow[u, >->]
            \arrow[l, shift left=0pt, bend left=10, "\bfil"]
    \end{tikzcd}
    \qquad
    \begin{tikzcd}[row sep=large,
                   column sep=large,
                   every matrix/.append style={draw, inner ysep=11pt, inner xsep=6pt, rounded corners}]
      \cat{Pos}
            \arrow[r, shift left=0pt, bend left=10, "\fun{up}"]
        & \cat{DL}
            \arrow[l, shift left=0pt, bend left=10, "\fun{pf}"] \\ [-.3em]
      \cat{IntKrip}
            \arrow[u, >->]
            \arrow[r, shift left=0pt, bend left=10, "\fun{up}"]
        & \cat{HA}
            \arrow[u, >->]
            \arrow[l, shift left=0pt, bend left=10, "\fun{pf}"]
    \end{tikzcd}
  $$
  \caption{Functors between categories of (semi)lattices.
           The upper rows are dual adjunctions.
           %The functor 
           $\fun{up}$ takes a poset to its lattice of
           upsets, and $\fun{pf}$ takes a lattice to its ordered
           set of prime filters.}
  \label{fig:int}
  \end{figure}
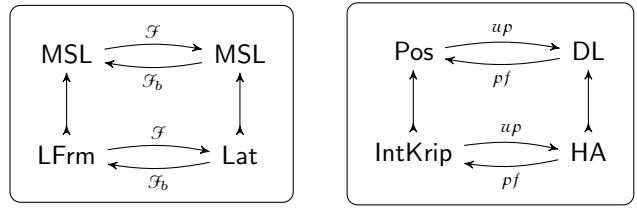

\begin{definition}\label{def:L-mor}
  An \emph{L-morphism} between semilattices $(X, \wedge)$ and $(X', \wedge')$
  is a semilattice homomorphism $f : (X, \wedge) \to (X', \wedge')$ 
  that satisfies for all $x \in X$ and $y', z' \in X'$:
  \begin{itemize}
    \item If $f(x) = \top'$ then $x = \top$;
    \item If $y' \wedge z' \leq f(x)$ then $\exists y, z \in X$ s.t.~%
          $y' \leq' f(y)$ and $z' \leq' f(z)$ and $y \wedge z \leq x$.
          In a picture:
  $$
    
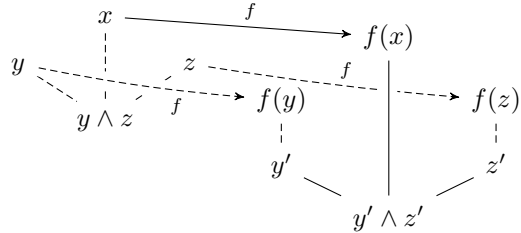
\begin{tikzcd}[scale=.95]
        & [-1.3em]
          x \arrow[rrrd, "f"]
            \arrow[dddd, dashed, -]
        & [-1.4em]
        & [-1em]
        & [-1.5em]
        & [-1.5em] \\ [-2.5em]
        &
        &
        &
        & f(x)
        & \\ [-2.2em]
      y     \arrow[ddr, dashed, -]
        &
        & z \arrow[ddl, dashed, -]
            \arrow[rrrd, bend left=-2, dashed, "f"]
        &&& \\ [-2.1em]
        &
        &
        & f(y) \arrow[dd, dashed, -]
               \arrow[lllu, bend left=3, dashed, <-, crossing over, "f" pos=.25]
        &
        & f(z) \arrow[dd, dashed, -] \\ [-2.6em]
        & y \wedge z
        &&&& \\ [-1.7em]
        &
        &
        & y' \arrow[dr, -]
        &
        & z' \arrow[dl, -] \\ [-1.5em]
        &
        &
        &
        & y' \wedge z'  \arrow[uuuuu, crossing over, -]
        &
    \end{tikzcd}
  $$
  \end{itemize}
  We write $\LFrm$ for the category of semilattices and L-morphisms.
\end{definition}

  The category $\cat{LFrm}$ will be used in Section~\ref{sec:logic} as
  frame semantics for weak positive logic.

\begin{proposition}\label{prop:Lmor-latmor}
  If $f : (X, \top, \wedge) \to (X', \top', \wedge')$ is an L-morphism,
  then $f^{-1} : \fil(X', \top', \wedge') \to \fil(X, \top, \wedge)$ is a lattice
  homomorphism.
\end{proposition}
\begin{proof}
  We know that $f^{-1}$ is a semilattice homomorphism.
  The map $f^{-1}$ preserves the bottom element because
  $f^{-1}(\{ \top' \}) = \{ \top \}$.
  For preservation of joins, we need to show that
  \begin{equation}
    f^{-1}(a' \genFil b') = f^{-1}(a') \genFil f^{-1}(b')
  \end{equation}
  for $a', b' \in \fil(X', \top', \wedge')$.
  The inclusion $\supseteq$ follows from the fact that
  $f^{-1}(a' \genFil b')$ is a filter that contains both $f^{-1}(a')$ and
  $f^{-1}(b')$. Conversely, if $x \in f^{-1}(a' \genFil b')$ then
  $f(x) \in a' \genFil b'$ so there exist $y' \in a'$ and $z' \in b'$
  such that $y' \wedge z' \leq f(x)$.
  Since $f$ is an L-frame morphism, we can find $y, z \in X$ such that
  $y' \leq f(y)$ and $z' \leq f(z)$ and $y \wedge z \leq x$.
  This means that $y \in f^{-1}(a')$ and $z \in f^{-1}(b')$, and hence
  $x \in f^{-1}(a') \genFil f^{-1}(b')$.
\end{proof}

  It follows that $\fil$ restricts to a contravariant functor
  $\fil : \LFrm \to \cat{Lat}$.
%  Interestingly, the converse holds as well;
%  we can restrict $\bfil$ to a contravariant functor
%  $\cat{Lat} \to \cat{LFrm}$.

\begin{proposition}\label{prop:latmor-Lmor}
  Let $h : L \to L'$ be a lattice homomorphism.
  Then $h^{-1} : \fil L' \to \fil L$ is an L-morphism.
\end{proposition}
\begin{proof}
  We know that $h^{-1}$ is a semilattice homomorphism,
  so we only have to show that it satisfies the additional conditions
  from Definition~\ref{def:L-mor}.
  For the first one, suppose $h^{-1}(p') = L$ ($L$ is the top element of $\fil L$).
  Then $\bot \in h^{-1}(p')$, so $\bot' = h(\bot) \in p'$, and therefore
  $p' = L'$.
  
  Next, let $p' \in \fil L'$ and $q, r \in \fil L$ and suppose
  $q \cap r \subseteq h^{-1}(p')$.
  Let $q' := {\uparrow}h[q]$ and $r' := {\uparrow} h[r]$.
  Then it is easy to verify that $q'$ and $r'$ are filters (because $q$ and $r$ are),
  and by construction $q \subseteq h^{-1}(q')$ and $r \subseteq h^{-1}(r')$.
  It remains to show that $q' \cap r' \subseteq p'$.
  Let $a' \in L'$ be such that $a' \in q' \cap r'$.
  Since $a' \in q'$ there exists $a \in q$ such that $h(a) \leq a'$.
  Since $a' \in r'$ there exists $b \in r$ such that $h(b) \leq a'$.
  But then $a \vee b \in q \cap r$, so by assumption $h(a \vee b) \in p'$.
  This implies $a' \in p'$, because
  $h(a \vee b) = h(a) \vee h(b) \leq a'$ and $p'$ is a filter (hence up-closed).
\end{proof}

%--------------------------------------------------------------------------------
\subsection{Dual Equivalences}

  HMS duality is obtained from the dual adjunction in
  Proposition~\ref{prop:dual-adj-MSL} by equipping one side with a
  Priestley topology. 

\begin{definition}
  An \emph{HMS-space} is a tuple $\topo{X} = (X, \top, \wedge, \tau)$
  such that
  $(X, \top, \wedge)$ is a semilattice,
  $(X, \tau)$ is a compact topological space,
  and $\topo{X}$ satisfies the \emph{HMS separation axiom}:
  \begin{center}
    for all $x, y \in X$, if $x \not\leq y$ then there exists a clopen
    filter $a$ such that $x \in a$ and $y \notin a$;
  \end{center}
  An \emph{HMS-morphism} is a continuous semilattice homomorphism.
  We write $\HMS$ for the category of HMS-spaces and HMS-morphisms.
\end{definition}

  The HMS separation axiom is a variation of the Priestley separation axiom.
  It immediately implies that any HMS-space is Hausdorff.
  Furthermore, it can be shown that every HMS-space is zero-dimensional,
  i.e.~every open neighbourhood of a point $x$ contains a clopen
  neighbourhood of $x$.
%  in the
%  same way as for Priestley spaces, so $(X, \tau)$ is a Stone space.
%  
  To see this, suppose $b$ is an open neighbourhood of a point $x$ in an HMS-space
  $\topo{X} = (X, \top, \wedge, \tau)$.
  Then for each $y \in X \setminus b$ either $x \not\leq y$ or $y \not\leq x$.
  By the HMS separation axiom, there exist a clopen filter or a complement of
  a clopen filter (which is clopen as well) containing $x$ but not $y$.
  By construction, the intersection of these clopen neighbourhoods of $x$
  is contained in $b$.
  Since $b$ is open and $\topo{X}$ is compact, there exists a finite number
  of such neighbourhoods whose intersection is contained in $b$.
  This finite intersection is the desired clopen neighbourhood of $x$.
  Thus, $(X, \tau)$ is a Stone space.
  
  For future reference, we prove some properties of closed sets and filters of an HMS-space.

\begin{lemma}
  Let $\topo{X} = (X, \top, \wedge, \tau)$ be an HMS-space and
  $c \subseteq X$ a filter. Then
  (i) $c$ is a closed iff
  (ii) $c$ is principal iff
  (iii) $c$ is the intersection of clopen filters.
%  \begin{enumerate}
%    \item $c$ is a closed filter
%    \item $c$ is a principal filter
%    \item $c$ is the intersection of clopen filters.
%  \end{enumerate}
\end{lemma}
\begin{proof}
  The implication (ii) $\Rightarrow$ (iii) follows from the HMS separation
  axiom and (iii) $\Rightarrow$ (i) is obvious.
  For (i) $\Rightarrow$ (ii), suppose $c$ is not principal.
  Then for each $x \in c$ there exists a $y \in c$ strictly below $x$.
  So for each $x \in c$, using the HMS separation axiom, we can find
  a clopen filter $a_x$ such that $c \not\subseteq a_x$.
  Then $\{ a_x \mid x \in c \}$ is an open cover of $c$ without
  finite subcover.
  (Indeed, for every finite collection $a_{x_1}, \ldots, a_{x_n}$
  we can find $y_1, \ldots, y_n$ such that $y_i \in c$ but $y_i \notin a_i$.
  Then $y_1 \wedge \cdots \wedge y_n$ is in $c$ but not in any of the $a_{x_i}$.)
  So $c$ is not compact, hence not closed.
\end{proof}

\begin{lemma}\label{lem:closed-upward-closure}
  Let $c$ be a closed subset of an M-space $\topo{X} = (X, \wedge, \tau)$.
  Then ${\uparrow}c$ is closed as well.
\end{lemma}
\begin{proof}
  If $y \notin {\uparrow}c$ then for each $x \in c$ we have
  $x \not\leq y$, hence a clopen filter
  $a_x$ containing $x$ but not $y$. Then $c \subseteq \bigcup_{x \in c} a_x$,
  so by compactness we find a finite subcover, say,
  $c \subseteq a_1 \cup \cdots \cup a_n$.
  Since all the $a_i$ are upward closed, we have
  ${\uparrow}c \subseteq a_1 \cup \cdots \cup a_n$.
  By construction, none of the $a_i$ contain $y$, so
  $X \setminus (a_1 \cup \cdots \cup a_n)$ is an open neighbourhood of
  $y$ disjoint from ${\uparrow}c$.
\end{proof}
  
  The clopen filters of an HMS-space form a semilattice with the whole
  space as top element and intersection as meet.
  This gives rise to a contravariant functor
  $$
    \cfil : \HMS \to \MSL,
  $$
  which sends HMS-morphisms to their inverse.
  In the converse direction, for every semilattice $(X, \top, \wedge)$
  we can equip $\fil(X, \top, \wedge)$ with a topology to
  obtain an HMS space, as follows.

\begin{definition}
  Let $A$ be a bounded semilattice.
  Define $\tfil A = (\fil A, A, \cap, \tau_A)$,
  where $\tau_A$ is the topology generated by
  \begin{equation}\label{eq:sl-mspace-subbase}
    \{ \theta_A(a) \mid a \in A \} \cup \{ \theta_A(a)^c \mid a \in A \},
  \end{equation}
  with $\theta_A(a) = \{ p \in \fil A \mid a \in p \}$
  and $\theta_A(a)^c = \fil A \setminus \theta_A(a)$.
  Defining $\tfil h = h^{-1}$ for a semilattice homomorphism $h$,
  we obtain a contravariant functor
  $
    \tfil : \MSL \to \HMS.
  $
\end{definition}

  We now obtain (a reformulation of) the duality 
  Hofmann, Mislove and Stralka~\cite{HofMisStr74,DavWer83,ClaDav98}.

\begin{theorem}\label{thm:mspace-bsl}
  The functors $\cfil$ and $\tfil$ establish a dual equivalence
  $\HMS \equiv^{\op} \MSL$.
\end{theorem}

\begin{remark}
  The following alternative proof for HMS duality was pointed out by
  the reviewer:
  %Alternatively, we can also obtain HMS duality through via Priestley duality.
  The forgetful functor $\mathcal{U}: \DL \to \MSL$ from distributive lattices
  to meet-semilattices has a left adjoint $\mathcal{Fr}: \MSL \to \DL$.
  %which constructs the free distributive lattice over a bounded semilattice.
  It follows that, for any semilattice $L$, the hom sets $\Hom_{\MSL}(L, 2)$ and
  $\Hom_{\DL}(\fun{Fr}L, 2)$ are naturally isomorphic.
  Using this, it is easy to see that the Priestley space dual to $\fun{Fr}L$
  coincides with the HMS space dual to $L$.
  We can then derive HMS duality for semilattices by observing that the
  Priestley spaces dual to the free distributive lattice over a semilattice
  are precisely those whose underlying poset forms a semilattice.
\end{remark}

  We wish to restrict this to a duality for lattices.
  To this end, we restrict the category $\HMS$ to L-spaces
  and suitable morphisms.

\begin{definition}
  A \emph{lattice space} or \emph{L-space} is an HMS-space
  $\topo{X} = (X, \top, \wedge, \tau)$ such that $a \genFil b$
  is clopen whenever $a$ and $b$ are clopen filters.
  An L-space morphism is a continuous L-morphism.
  We write $\LSpace$ for the category of L-spaces and their morphisms.
\end{definition}

%\begin{lemma}
%  If $c, c'$ are closed filters then so is $c \genFil c'$.
%  If $c$ is a closed subset of $X$ then so are $\uparr c$ and $\downarr c$.
%\end{lemma}
%\begin{proof}
%  \todo{To do.}
%\end{proof}
%
%\begin{lemma}
%  A filter of an L-space $\topo{X} = (X, \top, \wedge, \tau)$
%  is open iff it is the join of clopen filters.
%\end{lemma}
%\begin{proof}
%  Suppose $b$ is an open filter. Using compactness and
%  \eqref{it:M-sep} it can be shown that for each $x \in b$ we can find a
%  clopen filter $a$ such that $x \in a \subseteq b$.
%  As a consequence $b = \bigcup_{x \in b} a_x$, and since $b$ is the smallest
%  filter containing all of the $a_x$ we have $b = \bigGenFil_{x \in b} a_x$.
%%
%  Conversely, suppose $b = \bigGenFil_{i \in I} a_i$, where each $a_i$ is a
%  clopen filter. Then $b$ is a filter by definition.
%  As a consequence of \eqref{eq:join-union} we have
%  $$
%    b = \bigcup \{ a_{i_1} \genFil \cdots \genFil a_{i_n} \mid n \in \omega \text{ and } i_1, \ldots, i_n \in I \}.
%  $$
%  Since $\topo{X}$ is an L-space this is the union of clopen sets,
%  hence it is open.
%\end{proof}

\begin{theorem}\label{thm:duality-lat-lspace}
  The duality for bounded semilattices from Theorem~\ref{thm:mspace-bsl}
  restricts to a duality
  $$
    \cat{LSpace} \equiv^{\op} \cat{Lat}.
  $$
\end{theorem}
\begin{proof}
  We only have to verify that the restriction of $\cfil$ to $\cat{LSpace}$ lands in
  $\cat{Lat}$, and the restriction of $\tfil$ to $\cat{Lat}$ lands in $\cat{LSpace}$.
%
%  Since the clopen filters of an L-space $\topo{X}$ are closed under
%  $\genFil$, $\cfil \topo{X}$ is a lattice.
%  If $f : \topo{X} \to \topo{X}'$ is a 
  The former follows from the fact that the clopen filters of an L-space are
  closed under $\genFil$, together with Proposition~\ref{prop:Lmor-latmor}.

  For the latter, suppose that $L$ is a lattice and let $\theta_L(a)$ and
  $\theta_L(b)$ be two arbitrary clopen filters of $\tfil L$.
  %(They are principal by Lemma~\ref{lem:closed-filter-principal}.)
  Writing $x, y, z$ for elements in $\tfil L$, we have
  $$
    \theta_L(a) \genFil \theta_L(b)
      = {\uparrow} \{ x \cap y \mid x \in \theta_L(a), y \in \theta_L(b) \}
      = \theta_L(a \vee b)
  $$
  Let us elaborate on the last equality.
  %If $a = \bot$ or $b = \bot$ then the proof is obvious, so suppose this is not the case.
  If $x \in \theta_L(a)$ and $y \in \theta_L(b)$ then
  $a \in x$ and $b \in y$, so $a \vee b \in x \cap y$.
  So $z \supseteq x \cap y$ implies $a \vee b \in z$, and therefore we have
  ``$\subseteq$''.
  Conversely, if $z \in \theta_L(a \vee b)$ then we need to find
  $x \in \theta_L(a)$ and $y \in \theta_L(b)$ such that $x \cap y \subseteq z$.
  Let $x = {\uparrow}a \in \theta_L(a)$ and $y = {\uparrow}b \in \theta_L(b)$. 
  Then $d \in x \cap y$ implies $a \leq d$ and $b \leq d$, hence
  $a \vee b \leq d$. Since $z \in \theta_L(a \vee b)$ this implies $d \in z$,
  and therefore $x \cap y \subseteq z$. This proves ``$\supseteq$''.
  The restriction on morphisms follows from Proposition~\ref{prop:latmor-Lmor}.
\end{proof}

\begin{remark}
  In~\cite{MosJip09}, Moshier and Jipsen study a spectral analogue of
  Hofmann, Mislove and Stralka's duality for semilattices,
  which they also call HMS duality.
  Their ``HMS spaces'' relate to the original ones in the same way
  spectral spaces relate to Priestley spaces.
  Moshier and Jipsen also restrict their duality to lattices, obtaining what
  they call ``BL spaces''. Likewise, these are equivalent to our L-spaces
  through the same change of topology. Note that, while the join on BL spaces
  is defined via an infinite intersection of open filters
  (see \cite[Section 3]{MosJip09}), it coincides with the usual join of
  filters considered here.

  In \cite{Dmi21} Theorem~\ref{thm:duality-lat-lspace} was proven with
  different terminology: L-spaces and -morphisms are called ``PUP spaces''
  and ``PUP morphisms,'' and the category $\LSpace$ is called $\PUP$.
\end{remark}

%--------------------------------------------------------------------------------
\subsection{Completions of lattices}\label{subsec:prelim-compl}

  We relate several completions of a lattice to collections
  of certain filters of its dual L-space.

\begin{definition}
  A \emph{completion} of a lattice $L$ is a pair $(e, C)$ where $C$ is a
  complete lattice and $e : L \to C$ is a lattice embedding.
  An element in $C$ is called \emph{open} if it is the join of elements in the
  image of $e$, and \emph{closed} if it is the meet of elements in the image of
  $e$.
  
  A completion $(e, C)$ is called \emph{dense}
  if every element of $C$ can be written as the join of meets of elements
  in $L$, and as the meet of joins of elements in $L$.
  It is called \emph{compact} if for any set $A$ of closed elements of $C$
  and $B$ of open elements of $C$, $\bigwedge A \leq \bigvee B$ if and
  only if there are finite subsets $A' \subseteq A$ and $B' \subseteq B$
  such that $\bigwedge A' \leq \bigvee B'$.
\end{definition}

  It is well known that every lattice has a dense and compact completion
  which is unique up to isomorphism,
  see e.g.~\cite[Propositions~2.6 and~2.7]{GehHar01}.

\begin{definition}%[Completions of a lattice]
  Let $L$ be a lattice.
  \begin{enumerate}
    \item The \emph{ideal completion} of $L$ is the collection $\IdComp L$ of
          ideals of $L$ ordered by inclusion, with $i : L \to \IdComp L$ given
          by $a \mapsto {\downarrow}a$.
          Meets in $\IdComp L$ are given by intersection.
          As a consequence, the join of a collection of ideals is the smallest
          ideal containing their union.
    \item The \emph{filter completion} of $L$ is the collection
          $\FiltComp L$ of filters of $L$ ordered by reverse inclusion,
          with $i : L \to \FiltComp L : a \mapsto {\uparrow}a$.
          Then arbitrary joins in $\fun{F}L$ are given by intersections,
          and the meet of a collection of filters in $\FiltComp L$ is the
          smallest filter of $L$ containing their union.
    \item The \emph{canonical extension} of $L$ is the unique dense and
          compact completion of $L$.
    \item The \emph{$\Pi_1$-completion} of a lattice $L$ is given by
          the composition of the ideal and the filter completion.
          That is, it consists of the lattice $\IdComp(\FiltComp L)$
          with inclusion $a \mapsto \{ p \in \FiltComp L \mid a \in p \}$.
  \end{enumerate}
\end{definition}

  The $\Pi_1$-completion was studied in \cite{GehPri08}.
  Note that the ideal and filter completions are closely related.
  If we denote by $L^\circ$ the lattice $L$ with the order reversed,
  then the ideals of $L$ correspond to the filters of $L^\circ$
  and we get $\IdComp L = (\FiltComp L^\circ)^\circ$.
  
  A filter $p$ of an L-space $\topo{X} = (X, \top, \wedge, \tau)$
  is called \emph{saturated} if it equals the intersection of all
  open filters containing $p$.
  The collection of saturated filters of $\topo{X}$ is denoted by
  $\sfil\topo{X}$.

\begin{proposition}\label{prop:lattice-completions}
  Let $L$ be a lattice and $\topo{X}_L$ its dual L-space.
  \begin{enumerate}
    \item The filter completion of $L$ is isomorphic to the
          complete lattice $\kfil(\topo{X}_L)$ of principal filters of
          $\topo{X}_L$ of $L$, with inclusion
          $\theta_L: L \to \kfil(\topo{X}_L) : a \mapsto \theta(a)$.
    \item The canonical extension of $L$ is isomorphic to the
          complete lattice $\sfil(\topo{X}_L)$ of saturated filters of
          $\topo{X}_L$, with inclusion
          $\theta_L : L \to \sfil(\topo{X}_L) : a \mapsto \theta(a)$.
    \item The $\Pi_1$-completion of $L$ is isomorphic to the
          complete lattice $\fil(\topo{X}_L)$ of filters of $\topo{X}_L$,
          with inclusion $\theta_L : L \to \fil(\topo{X}_L) : a \mapsto \theta(a)$.
  \end{enumerate}
\end{proposition}
\begin{proof}
The first item follows from the lattice of principal filters of $L$ being isomorphic to $L^\circ$ and $\fil L = (\FiltComp L)^o$. The second item is similar to~\cite[Theorem~4.1]{MosJip09}. Finally, the third item follows from the mentioned above connection between filter and ideal completions $\IdComp L = (\FiltComp L^\circ)^\circ$ together with $\fil L = (\FiltComp L)^o$.
\end{proof}

%================================================================================
\section{Semilattice semantics for weak positive logic}\label{sec:logic}

  We use the duality and dual adjunction from Section~\ref{sec:lattices}
  to give frame semantics for weak positive logic, i.e.~the logic the
  same signature as positive logic, but with lattices as algebraic semantics.
  Inspired by the fact that the filters of a semilattice form a lattice,
  we use semilattices as frames and (principal) filters as denotations
  of formulae.

  We start this section by giving an axiomatisation of our logic.
  By design the algebraic semantics is simply given by lattices.
  In Section~\ref{subsec:l-frame} we define frames and models,
  give examples, and prove that the frame semantics is sound.
  In Section~\ref{subsec:complete} we use the duality from
  Section~\ref{sec:lattices} to derive completeness for weak positive logics
  with respect to several classes of frames.
  We give the standard translation into a suitable first-order logic
  and prove Sahlqvist correspondence in Section~\ref{subsec:corr},
  where we also work out specific examples of correspondence results.
  
  To distinguish the various notions of entailment each has their own
  notation, which are summarised in Table~\ref{tab:entailment}.
  We denote the interpretation of a formula $\phi$ in a lattice $\amo{A}$
  and in a frame $\mo{M}$
  by $\llp \phi \rrp_{\amo{A}}$ and $\llb \phi \rrb^{\mo{M}}$, respectively.
  Besides, we write $\ftop$ and $\fmeet$ for the top element and meets of a
  semilattice when it is regarded as frame semantics.

  \begin{table}[h]
  \centering
    \begin{tabular}{lll}
      \toprule
        Notation & Purpose & Location \\ \midrule
        $\phi \vdash \psi$ & Syntactic entailment & Def.~\ref{def:logic} \\
        $\phi \cdash \psi$ & Algebraic entailment & Def.~\ref{def:alg-entail} \\
        $\phi \Vdash \psi$ & Semantic entailment & Def.~\ref{def:interpretation} \\
        $\phi \Vdash_{\cat{LSpace}} \psi$ & Topological semantic entailment
          & Def.~\ref{def:inerpretation-lspace} \\
        $\phi \models \psi$ & First-order entailment & Sec.~\ref{subsec:corr} \\
      \bottomrule
    \end{tabular}
    \caption{Different notions of entailment.}
    \label{tab:entailment}
  \end{table}

%--------------------------------------------------------------------------------
\subsection{Logic and Algebraic Semantics}\label{subsec:algsem}

  Let $\lan{L}(\Prop)$ denote the language generated by the grammar
  $$
    \phi ::= p \mid \top \mid \bot \mid \phi \wedge \phi \mid \phi \vee \phi,
  $$
  where $p$ ranges over some arbitrary but fixed set $\Prop$ of proposition
  letters.
  If no confusion arises we omit reference to $\Prop$ and simply
  write $\lan{L}$.
%
  %For lack of a strong enough implication,
  %Since we do not have a strong enough implication available,
  We define logics based on $\lan{L}$
  as a collection of \emph{consequence pairs},
  similar to e.g.~\cite{Dun95}.
  A consequence pair is an expressions of the form $\phi \cp \psi$
  where $\phi$ and $\psi$ are formulae in $\lan{L}$,
  and intuitively means: ``If $\phi$ holds, then
  so does $\psi$.''
  
\begin{definition}\label{def:logic}
  Let $\log{L}$ be the smallest set of consequence pairs
  closed under the following axioms and rules:
  \begin{align*}
    p \cp \top,
      &\qquad \bot \cp p,
      &\text{\emph{top} and \emph{bottom}} \\
    p \cp p,
      &\qquad \frac{p \cp q \quad q \cp r}{p \cp r},
      &\text{\emph{reflexivity} and \emph{transitivity}} \\
    p \wedge q \cp p,
       \qquad p \wedge q \cp q,
      &\qquad \dfrac{r \cp p \quad r \cp q}{r \cp p \wedge q},
      &\text{\emph{conjunction rules}} \\
    p \cp p \vee q,
       \qquad q \cp p \vee q,
      &\qquad \dfrac{p \cp r \quad q \cp r}{p \vee q \cp r}
      &\text{\emph{disjunction rules}}
  \end{align*}
  If $\Gamma$ is a set of consequence pairs then we let
  $\log{L}(\Gamma)$ denote the smallest set of consequence pairs
  closed under uniform substitution, the axioms and rules mentioned above and those in $\Gamma$.
  We write $\phi \vdash_{\Gamma} \psi$ if $\phi \cp \psi \in \log{L}(\Gamma)$
  and $\phi \dashv\vdash_{\Gamma} \psi$ if both $\phi \vdash_{\Gamma} \psi$
  and $\psi \vdash_{\Gamma} \phi$.
%  If $\Gamma = \emptyset$ then we omit it.
%  We abbreviate $\psi \vdash_{\emptyset} \phi$ to $\psi \vdash \phi$,
%  and similar for $\phi \dashv\vdash \psi$.
  If $\Gamma$ is the empty set then we simply write
  $\phi \vdash \psi$ and $\phi \dashv\vdash \psi$.
\end{definition}

  The algebraic semantics of the logic $\lan{L}$ are simply lattices.
  We establish this formally.

\begin{definition}
  Let $A$ be a lattice with operations $\top_A, \bot_A, \wedge_A, \vee_A$,
  and induced order $\leq_A$.
  A \emph{lattice model} is a pair $\amo{A} = (A, \sigma)$ 
  consisting of a lattice $A$ and an assignment $\sigma : \Prop \to A$
  of the proposition letters.
  The assignment $\sigma$ uniquely extends to a map
  $\llp \cdot \rrp_{\amo{A}} : \lan{L} \to A$
  by interpreting connectives with their lattice counterparts.
  
%  We define the interpretation
%  $\llp \phi \rrp_{\amo{A}}$ of an $\lan{L}$-formula $\phi$
%  in $\amo{A}$ recursively via
%  \begin{align*}
%    \llp p \rrp_{\amo{A}} &= \sigma(p)
%      &\llp \top \rrp_{\amo{A}} &= \top_A 
%      &\llp \phi \wedge \psi \rrp_{\amo{A}}
%        &= \llp \phi \rrp_{\amo{A}} \wedge_A \llp \psi \rrp_{\amo{A}} \\
%    %%
%      &
%      &\llp \bot \rrp_{\amo{A}} &= \bot_A
%      &\llp \phi \vee \psi \rrp_{\amo{A}}
%        &= \llp \phi \rrp_{\amo{A}} \vee_A \llp \psi \rrp_{\amo{A}}.
%  \end{align*}
  We say that a lattice $A$ \emph{validates} a consequence pair $\phi \cp \psi$
  if $\llp \phi \rrp_{\amo{A}} \leq_A \llp \psi \rrp_{\amo{A}}$ for
  all lattice models $\amo{A}$ based on $A$, notation: $A \cdash \phi \cp \psi$.
  If $\Gamma$ is a set of consequence pairs then we write
  $\cat{Lat}(\Gamma)$ for the full subcategory of $\cat{Lat}$ whose objects
  validate all consequence pairs in $\Gamma$.
\end{definition}
  
\begin{definition}\label{def:alg-entail}
  Let $\Gamma \cup \{ \phi \cp \psi \}$ be a set of consequence pairs.
  Write $\phi \cdash[\Gamma] \psi$ if 
  $\llp \phi \rrp_{\amo{A}} \leq_A \llp \psi \rrp_{\amo{A}}$ for
  every lattice model $\amo{A} = (A, \sigma)$ with $A \in \cat{Lat}(\Gamma)$.
  We abbreviate $\phi \cdash[\emptyset] \psi$
  to $\phi \cdash \psi$.
\end{definition}

  Observe that $\dashv\vdash_{\Gamma}$ is an equivalence relation on $\lan{L}$.
  Write $L(\Gamma)$ for the set of $\dashv\vdash_{\Gamma}$-equivalence classes
  of $\lan{L}$, and denote by $[\phi]$ the equivalence class of $\phi$ in
  $L(\Gamma)$.
  Then it follows from the rules in Definition~\ref{def:logic} that
  $L(\Gamma)$ carries a lattice structure, where
  $\top_L = [\top]$, $\bot_L = [\bot]$,
  $[\phi] \wedge_L [\psi] = [\phi \wedge \psi]$ and
  $[\phi] \wedge_L [\psi] = [\phi \vee \psi]$.
  Moreover, $L(\Gamma)$ is in $\cat{Lat}(\Gamma)$, and
  setting $\sigma_L : \Prop \to L(\Gamma) : p \mapsto [p]$ yields
  lattice model $\amo{L}_{\Gamma} = (L(\Gamma), \sigma_L)$ which acts as
  the Lindenbaum-Tarski algebra. It follows from induction on the
  structure of $\phi$ that  $\llp \phi \rrp_{\amo{L}_{\Gamma}} = [\phi]$
  for all $\lan{L}$-formulae $\phi$.

\begin{lemma}\label{lem:LT}
  We have $\phi \cdash[\Gamma] \psi$ if and only if
  $\llp \phi \rrp_{\amo{L}_{\Gamma}} \leq_L \llp \psi \rrp_{\amo{L}_{\Gamma}}$.
\end{lemma}
\begin{proof}
  The ``only if'' holds by definition.
  Conversely, if $A \in \cat{Lat}(\Gamma)$ and $\amo{A} = (A, \sigma_A)$
  is a lattice model, then the assignment $[p] \mapsto \sigma_A(p)$
  extends to a lattice homomorphism $i : \amo{L}_{\Gamma} \to \amo{A}$
  such that $[\phi] = \llp \phi \rrp_{\amo{A}}$.
  (This is well defined because $A$ validates all consequence pairs in $\Gamma$.)
  Then $\llp \phi \rrp_{\amo{L}_{\Gamma}} \leq_L \llp \psi \rrp_{\amo{L}_{\Gamma}}$
  implies $[\phi] \leq_L [\psi]$. Monotonicity of $i$ yields
  $\llp \phi \rrp_{\amo{A}} \leq_L \llp \psi \rrp_{\amo{A}}$,
  hence~$\phi \cdash[\Gamma] \psi$.
\end{proof}

\begin{theorem}\label{thm:alg-sem}
  We have $\phi \vdash_{\Gamma} \psi$ if and only if $\phi \cdash[\Gamma] \psi$.
\end{theorem}
\begin{proof}
  By Lemma~\ref{lem:LT} it suffices to show that
  $\phi \vdash_{\Gamma} \psi$ if and only if
  $\llp \phi \rrp_{\amo{L}_{\Gamma}} \leq_L \llp \psi \rrp_{\amo{L}_{\Gamma}}$.
  It follows from the conjunction rules, reflexivity and transitivity that
  $\phi \vdash_{\Gamma} \psi$ if and only if
  $\phi \wedge \psi \dashv\vdash_{\Gamma} \phi$.
  Therefore we have $\phi \vdash_{\Gamma} \psi$ if and only if
  $[\phi \wedge \psi] = [\phi]$ in $\amo{L}_{\Gamma}$,
  and since $[\phi \wedge \psi] = [\phi] \wedge_L [\psi]$ this holds
  if and only if $[\phi] \leq_L [\psi]$.
  Recalling that $[\phi] = \llp \phi \rrp_{\amo{L}_{\Gamma}}$ completes the proof.
\end{proof}

%--------------------------------------------------------------------------------
\subsection{Frame Semantics}\label{subsec:l-frame}
  
  The collection of filters of a semilattice forms a lattice.
  Therefore we can use semilattices as frame semantics of weak positive logic,
  with filters serving as denotations of formulae.
  If moreover the semilattice is a lattice, then we can also use
  principal filters as denotations of formulae.

\begin{definition}\label{def:interpretation}
  A \emph{lattice model} or \emph{L-model} is a semilattice
  $(X, \ftop, \fmeet)$ together with a valuation
  $V : \Prop \to \fil(X, \ftop, \fmeet)$ which assigns to each proposition
  letter a filter of $(X, \ftop, \fmeet)$.
  An L-model $(X, \ftop, \fmeet, V)$ is called \emph{principal} if
  $(X, \ftop, \fmeet)$ has a bottom element $0$ and binary joins denoted by
  $\fjoin$ (so it forms a lattice)
  and $V(p)$ is a principal filter for all $p \in \Prop$.

  The interpretation of an $\lan{L}$-formula $\phi$ at a state $x$ in
  a (principal) L-model $\mo{M} = (X, \ftop, \fmeet, V)$ is defined recursively via
  \begin{align*}
    \mo{M}, x \Vdash \top &\phantom{\iff} \text{always} \\
    \mo{M}, x \Vdash \bot &\iff x = \ftop \\
    \mo{M}, x \Vdash p &\iff x \in V(p) \\
    \mo{M}, x \Vdash \phi \wedge \psi
      &\iff \mo{M}, x \Vdash \phi \text{ and } \mo{M}, x \Vdash \psi \\
    \mo{M}, x \Vdash \phi \vee \psi
      &\iff \exists y, z \in X \text{ s.t. } \mo{M}, y \Vdash \phi \text{ and }
            \mo{M}, z \Vdash \psi \text{ and } y \fmeet z \fleq x
  \end{align*}
  We write $\llb \phi \rrb^{\mo{M}} := \{ x \in X \mid \mo{M}, x \Vdash \phi \}$
  for the \emph{truth set} of $\phi$ in $\mo{M}$.
  If the underlying (semi)lattice is fixed and we want to emphasise the role of the valuation
  in the interpretation, we will write $V(\phi)$ instead of $\llb \phi \rrb^{\mo{M}}$.
  The \emph{theory} of $x$ is denoted by
  $\th_{\mo{M}}(x) := \{ \phi \in \lan{L} \mid \mo{M}, x \Vdash \phi \}$.
\end{definition}

  Note that the (semi)lattice underlying an L-model is uniquely determined
  by its partial order. So we may view L-models as a type of relational
  semantics, where the relation is used to define a non-standard interpretation
  of joins.
  When viewed as frame semantics, we denote the top element and meet of a
  semilattice by $\ftop$ and $\fmeet$ and call the semilattice itself an L-frame.

\begin{definition}
  %\jg{Which notions do we actually need?}
  We write $\mo{M}, x \Vdash \phi \cp \psi$ if $x \in \llb \phi \rrb^{\mo{M}}$
  implies $x \in \llb \psi \rrb^{\mo{M}}$, and
  $\mo{M} \Vdash \phi \cp \psi$ if
  $\llb \phi \rrb^{\mo{M}} \subseteq \llb \psi \rrb^{\mo{M}}$.
  If $\mo{X}$ is an L-frame, we let $\mo{X}, x \Vdash \phi \cp \psi$ if
  $\mo{M}, x \Vdash \phi \cp \psi$ for all L-models $\mo{M}$ based on $\mo{X}$,
%  $x \in V(\phi)$ implies
%  $x \in V(\psi)$ for all valuations $V$ of $\mo{X}$,
  and $\mo{X} \Vdash \phi \cp \psi$ if $\mo{X}, x \Vdash \phi \cp \psi$ for
  all states $x$ of $\mo{X}$.

  We say that $\mo{M}$ or $\mo{X}$ \emph{validates} $\phi \cp \psi$
  if $\mo{M} \Vdash \phi \cp \psi$ or $\mo{X} \Vdash \phi \cp \psi$, respectively.
  If $\Gamma$ is a set of consequence pairs, then we let
  $\cat{LFrm}(\Gamma)$ denote the full subcategory of $\cat{LFrm}$ whose
  objects validate all consequence pairs in $\Gamma$.
  We write $\phi \Vdash_{\Gamma} \psi$ if $\mo{X} \Vdash \phi \cp \psi$
  for all $\mo{X} \in \cat{LFrm}(\Gamma)$.
  If $\Gamma = \emptyset$ then we write $\phi \Vdash \psi$ instead of
  $\phi \Vdash_{\emptyset} \psi$.
\end{definition}
  
  For any L-frame $\mo{F} = (X, \leq)$, the collection
  $\mo{F}^* := \fil(X, \leq)$ forms a lattice, called the \emph{complex algebra}
  of $\mo{F}$. Since valuations of $\mo{F}$ correspond bijectively to
  assignments of $\mo{F}^*$, we can define the \emph{complex algebra}
  of an L-model $\mo{M} = (X, \leq, V)$ by 
  $\mo{M}^* = (\fil(X, \leq), V)$.
  A routine induction on the structure of $\phi$ then proves the following
  lemma.

%  \ad{Alternatively, we could use this as a definition, i.e. say that an L-model is a semilattice whose filters are a lattice model. That could possibly shorten the paper. Then we should probably still mention what the definition of join turns out to be.}
%  \jg{I think a direct definition is nicer.}

\begin{lemma}\label{lem:complex-alg1}
  For every L-model $\mo{M}$ and $\lan{L}$-formula $\phi$ we have
  $\llb \phi \rrb^{\mo{M}} = \llp \phi \rrp_{\mo{M}^*}$.
\end{lemma}

  %As an immediate corollary we obtain the following persistence result.
  The next persistence result is similar to persistence in intuitionistic logic,
  except we require formulae to be interpreted as (principal) filters rather than upsets.

\begin{proposition}[Persistence]
  Let $\mo{M} = (X, \ftop, \fmeet, V)$ be a (principal) L-model.
  Then for each $\phi \in \lan{L}$ the truth set $\llb \phi \rrb^{\mo{M}}$
  of $\phi$ is a (principal) filter of $(X, \ftop, \fmeet)$.
\end{proposition}
\begin{proof}
  The fact that $\llb \phi \rrb^{\mo{M}}$ is a filter for each $\phi$
  follows from Lemma~\ref{lem:complex-alg1}.
  Suppose $\mo{M}$ is principal.
  Then $\llb p \rrb^{\mo{M}}$ is principal by definition,
  as are $\llb \bot \rrb^{\mo{M}} = \uparr\top$
  and $\llb \top \rrb^{\mo{M}} = \uparr 0$.
  If $\phi = \phi_1 \wedge \phi_2$ or $\phi_1 \vee \phi_2$ then
  we proceed by induction.
  We may assume that
  $\llb \phi_1 \rrb^{\mo{M}} = {\uparrow}x_1$ and
  $\llb \phi_2 \rrb^{\mo{M}} = {\uparrow}x_2$, so that
  $\llb \phi_1 \wedge \phi_2 \rrb^{\mo{M}} = {\uparrow}(x_1 \fjoin x_2)$
  and $\llb \phi_1 \vee \phi_2 \rrb^{\mo{M}} = {\uparrow}(x_1 \fmeet x_2)$.
\end{proof}

\begin{theorem}[Soundness]\label{thm:soundness}
  If $\phi \vdash_{\Gamma} \psi$ then $\phi \Vdash_{\Gamma} \psi$.
\end{theorem}
\begin{proof}
  If $\mo{M}$ is an L-model that validates all consequence pairs in $\Gamma$,
  then $\mo{M}^* \in \cat{Lat}(\Gamma)$.
  Since $\phi \vdash_{\Gamma} \psi$, Theorem~\ref{thm:alg-sem} yields
  $\phi \cdash[\Gamma] \psi$, and hence
  $\llp \phi \rrp_{\mo{M}^*} \leq \llp \psi \rrp_{\mo{M}^*}$.
  Lemma~\ref{lem:complex-alg1} now implies
  $\llb \phi \rrb^{\mo{M}} \leq \llb \psi \rrb^{\mo{M}}$,
  so that $\mo{M}$ validates $\phi \cp \psi$.
%  Since $\mo{M}$ was chosen arbitrarily, this proves
%  $\phi \Vdash_{\Gamma} \psi$.
\end{proof}

  We turn the collections of (principal) L-models into a category by
  equipping with truth-preserving morphisms.

\begin{definition}
  An \emph{L-model morphism} from
  $(X, \ftop, \fmeet, V)$ to $(X', \ftop', \fmeet', V')$ is
  an L-morphism (Definition~\ref{def:L-mor})
  $f : (X, \ftop, \fmeet) \to (X', \ftop', \fmeet')$
  that satisfies $V = f^{-1} \circ V'$.
%  We write $\cat{LMod}$ for the category of L-models and L-model morphisms,
%  and $\cat{pLMod}$ for its full subcategory of principal L-models.
\end{definition}

  A routine induction on the structure of $\phi$ shows that
  L-model morphisms preserve and reflect truth of
  $\lan{L}$-formulae.

\begin{proposition}
  Let $f : \mo{M} \to \mo{M}'$ be an L-model morphism.
  Then for all states $x$ of $\mo{M}$ and all $\phi \in \lan{L}$,
  $$
    \mo{M}, x \Vdash \phi \iff \mo{M}', f(x) \Vdash \phi.
  $$
\end{proposition}
%\begin{proof}
%  Routine induction on the structure of $\phi$.
%% \blue{Since $x \in V(\phi) \Leftrightarrow f(x) \in V'(\phi)$.}
%% \jg{I don't understand this: it is just restating what we need to prove.}
%\end{proof}

  The remainder of this subsection is devoted to examples of L-frames and -models.

\begin{example}\label{exm:Lfrm}
%  \begin{enumerate}
%    \item \label{it:exm-Lfrm1}
          Any linearly ordered set with a largest element is an L-frame.
          Filters in such frames are simply upsets.
          For example $\mathbb{N} \cup \{ \infty \}$ with the usual ordering
          is an L-frame which is prinicipal.
          The set $\mathbb{N}$ ordered by $\geq$ is also an L-frame, with top
          element $0$. It is not principal because it lacks a bottom element.  
\end{example}

\begin{example}
  As a special case of Example~\ref{exm:Lfrm},
  consider the collection $\fun{P}_{\omega}X$ of finite subsets of $X$.
  This forms a semilattice with top element $\emptyset$,
  and meet given by the set-theoretic union.
  Filters of $\fun{P}_{\omega}X$ correspond bijectively with subsets of $X$.
%          (This is the free semilattice over $X$, and filters of
%          $\fun{P}_{\omega}X$ correspond bijectively with subsets of $X$.)
%%          \qedhere
%% 
%    \item Another interesting class of examples of principal L-frames is
%          given by rooted trees of finite depth.
%%
%          Filters in such frames are always principal, and a valuation
%          indicates that a property $p$ is true at a node $x$ and henceforth.
%          As usual, a node satisfies $\phi \wedge \psi$ if it satisfies both
%          $\phi$ and $\psi$.
%%
%          In such structures, a node $x$ satisfies $\phi \vee \psi$ if it has
%          two cousins $y$ and $z$ whose ``youngest'' common ancestor is also
%          an ancestor of $x$, such that $y$ satisfies $\phi$ and
%          $z$ satisfies $\psi$.
%          \qedhere    
%
%    \ad{Doesn't this semilattice have a bottom instead of a top? Should we add a top to this structure?}
\end{example}

\begin{example}[Propositional team semantics]
  We briefly recall a simplified version of team semantics for propositional
  logics, underlying versions of modal dependence and independence
  logics such as the ones studied in \cite{HelEA14,KinEA14,Yang17,YV17}.
  Let $\lan{T}(\Prop)$ the language be given by the grammar
  $\phi ::= p \mid \neg p \mid \phi \wedge \phi \mid \phi \vee \phi$.
  Then $\lan{T}(\Prop)$-formulae can be interpreted in models
  consisting of a set $X$ and a valuation
  $\Pi : \Prop \to \fun{P}X$ of the proposition letters.
  However, rather than assigning truth of formulae to elements of $X$,
  truth is defined for \emph{subsets} of $X$ (the teams).
  Let $\mo{M} = (X, \Pi)$ be such a model and $T \subseteq X$ a team,
  then we let
  \begin{align*}
    \mo{M}, T \tVdash p &\iff T \subseteq V(p) \\
    \mo{M}, T \tVdash \neg p &\iff T \cap V(p) = \emptyset \\
    \mo{M}, T \tVdash \phi \wedge \psi &\iff \mo{M}, T \tVdash \phi \text{ and } \mo{M}, T \tVdash \psi \\
    \mo{M}, T \tVdash \phi \vee \psi
      &\iff \exists\, T_1, T_2 \subseteq T \text{ s.t. } T_1 \cup T_2 = T
            \text{ and } \mo{M}, T_1 \tVdash \phi
            \text{ and } \mo{M}, T_2 \tVdash \psi
  \end{align*}
  We can add $\top$, which is true for every team,
  and $\bot$ satisfying $\mo{M}, T \Vdash_t \bot$ iff $T = \emptyset$.
  
  This interpretation resembles
  Definition~\ref{def:interpretation}. Let us make this precise.
  For a set $\Prop$ of proposition letters, let
  ${\neg{\Prop}} = \{ \neg p \mid p \in \Prop \}$.
  Then, given a team model $\mo{M} = (X, \Pi)$, we can define a principal L-model
  $\mo{M}' = (\fun{P}X, \emptyset, \cup, V)$, with
  $V(p) = \{ a \in \fun{P}X \mid a \subseteq \Pi(p) \}$
  and $V(\neg p) = \{ a \in \fun{P}X \mid a \cap \Pi(p) = \emptyset \}$.
  Then for each team model $\mo{M}$,
  team $T$, and formula $\phi \in \lan{T}(\Prop)$ we have
  $$
    \mo{M}, T \tVdash \phi \iff \mo{M}', T \Vdash \phi. 
  $$
  This can be proven by induction on the structure of $\phi$.
  The only non-trivial step is for joins:
  \begin{align*}
  \mo{M}, t \Vdash_t \phi \vee \psi
      &\iff \exists T_1, T_2 \in \fun{P}X \text{ s.t. } T_1 \cup T_2 = T
            \text{ and } \mo{M}, T_1 \Vdash_t \phi \text{ and } \mo{M}, T_2 \Vdash_t \psi \\
      &\iff \exists T_1, T_2 \in \fun{P}X \text{ s.t. } T_1 \cup T_2 = T
            \text{ and } \mo{M}, T_1 \Vdash \phi \text{ and } \mo{M}, T_2 \Vdash \psi \\
      &\iff \exists T_1', T_2' \in \fun{P}X \text{ s.t. } T_1' \cup T_2' \supseteq T
            \text{ and } \mo{M}, T_1' \Vdash \phi \text{ and } \mo{M}, T_2' \Vdash \psi \\
      &\iff \mo{M}, T \Vdash \phi \vee \psi
  \end{align*}
  The first ``iff'' is the definition of $\Vdash_t$, the second follows from
  the induction hypothesis. The third ``iff'' follows from persistence and the
  fact that the frame is ordered by reverse inclusion, and the last ``iff''
  hold by the definition of $\Vdash$. \qedhere
%  
%  \ad{I think this is not true, since our join condition has an inequality instead of equality. Should we change this to an implication?}
%  \blue{Jim: I added a proof.}
\end{example}

\begin{example}[Modal information logic]
  Modal information logic~\cite{vB22} is the extension of propositional classical
  logic with two binary modal operators $\minf$ and $\msup$.
  These are interpreted in Kripke models $\mo{M} = (X, R, V)$ where
  $R$ is a pre-order on $X$ as follows:
  \begin{align*}
    \mo{M}, x \Vdash \minf (\phi, \psi) &\iff \exists y, z \in X
              \text{ s.t. } x = \inf(y, z)
              \text{ and } \mo{M}, y \Vdash \phi
              \text{ and } \mo{M}, z \Vdash \psi \\
    \mo{M}, x \Vdash \msup (\phi, \psi) &\iff \exists y, z \in X
              \text{ s.t. } x = \sup(y, z)
              \text{ and } \mo{M}, y \Vdash \phi
              \text{ and } \mo{M}, z \Vdash \psi
  \end{align*}
  Note that we need not require that every pair of states has an infimum
  and a supremum, nor that it is unique.
  The definition simply uses the fact that they might
  exist. Observe that we can recover the usual modal and temporal diamonds via
  $\Diamond\phi = \minf (\phi, \top)$ and
  %Moreover, we can define a temporal diamond $\bdiamond$ as
  $\bdiamond\phi = \msup (\phi, \top)$.
  
  Clearly, every L-model is a model for modal information logic.
  Interestingly, the interpretation of $\minf$ is closely aligned
  to our interpretation of joins; the only difference is that the
  infimum is allowed to be \emph{below} the state under consideration.
  Taking this into account, our interpretation of joins in an L-model
  $\mo{M} = (X, \ftop, \fmeet, V)$ coincides with
  $$
    \phi \vee \psi = \bdiamond(\minf(\phi, \psi)),
  $$
  where $\vee$ is the non-classical join of weak positive logic.
\end{example}

%--------------------------------------------------------------------------------
\subsection{Descriptive Frames and Completeness}\label{subsec:complete}

  We have already seen a duality for lattices by means of L-spaces.
  Since every L-space is based on a complete semilattice,
  L-spaces can be viewed as topologised (principal) L-frames.
  In this subsection we define clopen valuations for L-spaces and
  show how this gives rise to completeness results.
  We denote L-spaces and L-spaces with a valuation by $\topo{X}$ and $\topo{M}$.
  If $\topo{X}$ is an L-space, then we write $\kappa\topo{X}$ for its
  underlying (principal) L-frame.

\begin{definition}\label{def:inerpretation-lspace}
  A \emph{clopen valuation} for an L-space $\topo{X}$ is an assignment
  $V : \Prop \to \cfil\topo{X}$, which assigns to each proposition letter
  a clopen filter of $\topo{X}$. We call a pair $\topo{M} = (\topo{X}, V)$
  of an L-space and a clopen valuation an \emph{L-space model}.
  The interpretation $\llb \phi \rrb^{\topo{M}}$ of an $\lan{L}$-formula
  $\phi$ in an L-space model $\topo{M} = (\topo{X}, V)$ is defined as in
  the underlying L-model $(\kappa\topo{X}, V)$.

  An L-space model $\topo{M}$ validates a consequence pair $\phi \cp \psi$
  if $\llb \phi \rrb^{\topo{M}} \subseteq \llb \psi \rrb^{\topo{M}}$,
  notation: $\topo{M} \Vdash \phi \cp \psi$.
  %\jg{$\leftarrow$ do we need this?}
  We say that an L-space $\topo{X}$ validates $\phi \cp \psi$ if
  every L-space model based on it validates $\phi \cp \psi$.
  Finally, we write $\phi \Vdash_{\cat{LSpace}} \psi$ if every L-space
  validates $\phi \cp \psi$.
\end{definition}

\begin{lemma}\label{lem:Lspace-lattice-val}
  Let $\topo{X}$ be an L-space, $A$ its dual lattice, and $\phi, \psi \in \lan{L}$.
  Then
  $$
    \topo{X} \Vdash \phi \text{ iff } A \cdash \phi
    \quad\text{and}\quad
    \topo{X} \Vdash \phi \cp \psi \text{ iff } A \cdash \phi \cp \psi.
  $$
\end{lemma}
\begin{proof}
  The first ``iff'' follows from the fact that clopen valuations of $\topo{X}$
  correspond bijectively to assignments of the proposition letters for $A$,
  together with a routine induction on the structure of $\phi$.
  The second ``iff'' follows immediately from the first.
\end{proof}

\begin{remark}\label{rem:descriptive}
  We can alternatively describe L-spaces as \emph{descriptive L-frames}.
  This is similar to the perspective of Esakia spaces as
  descriptive intuitionistic Kripke frames, see~\cite[Chapter~3]{Esakia-book85} and \cite[Section~8]{CZ97}.
  %see also Figure~\ref{fig:int2}.
  We briefly sketch this alternative perspective.
%  That is, as L-frames with extra structure, similar to descriptive
%  intuitionistic Kripke frames. Since we prefer to work with L-spaces,
%  we only briefly sketch this perspective.
  
  A \emph{general L-frame} is a tuple $(X, \ftop, \fmeet, A)$ such that
  $(X, \ftop, \fmeet)$ is an L-frame and $A$ is a collection of filters of
  $(X, \ftop, \fmeet)$ containing $X$ and $\emptyset$, and closed under
  $\cap$ and $\genFil$.
  %The sets in $A$ are called \emph{admissible filters}.
  Let $-A = \{ X \setminus a \mid a \in A \}$.
  A \emph{descriptive L-frame} is a general L-frame $(X, \ftop, \fmeet, A)$ that is
  \begin{itemize}
    \item \emph{refined}: for all $x, y \in X$ such that $x \not\fleq y$
          there exists an $a \in A$ such that $x \in a$ and $y \notin a$;
    \item \emph{compact}: if $C \subseteq A \cup -A$ has the finite
          intersection property then $\bigcap C \neq \emptyset$.
%    \item \emph{descriptive} if it is refined and compact.
  \end{itemize}
  A general L-morphism from $(X, \ftop, \fmeet, A)$ to $(X', \ftop', \fmeet', A')$ is an L-morphism
  $f : (X, \ftop, \fmeet) \to (X', \ftop', \fmeet')$ such that $f^{-1}(a') \in A$ for all $a' \in A'$.
  Write $\cat{D\hyphen LFrm}$ for the category
  of descriptive L-frames and general L-morphisms.
  Then we have
  $
    \cat{D\hyphen LFrm} \cong \cat{LSpace}.
  $
\end{remark}

%%%%%%%%%%%%%%%%%%%%%%%% FIGURE: OVERVIEW OF DUALITIES %%%%%%%%%%%%%%%%%%%%%%%%%%
%  \begin{figure}[h]
%  $$
%    \begin{tikzcd}[row sep=large,
%                   column sep=large,
%                   every matrix/.append style={draw, inner ysep=7pt, inner xsep=6pt, rounded corners}]
%      %%
%        & \cat{HMS}
%            \arrow[r, -, "\equiv^{\op}" {below,pos=.6}, "\text{Thm.~\ref{thm:mspace-bsl}}" above]
%        & \MSL \\
%      \cat{D\hyphen LFrm}
%            \arrow[r, -, "\cong" below, "\text{Rem.~\ref{rem:descriptive}}" above]
%        & \cat{LSpace}
%            \arrow[u, >->]
%            \arrow[r, -, "\equiv^{\op}" {below,pos=.6}, "\text{Thm.~\ref{thm:duality-lat-lspace}}" above]
%        & \Lat \arrow[u, >->]
%    \end{tikzcd}
%    \qquad
%    \begin{tikzcd}[row sep=large,
%                   column sep=large,
%                   every matrix/.append style={draw, inner ysep=7pt, inner xsep=6pt, rounded corners}]
%      %%
%        & \cat{Pries}
%            \arrow[r, -, "\text{Priestley}", "\equiv^{\op}" {below,pos=.6}]
%        & \cat{DL} \\
%      \cat{D\hyphen Krip}
%            \arrow[r, -, "\cong" swap]
%        & \cat{Esakia}
%            \arrow[u, >->]
%            \arrow[r, -, "\text{Esakia}", "\equiv^{\op}" {below,pos=.6}]
%        & \cat{HA} \arrow[u, >->]
%    \end{tikzcd}
%  $$
%  \caption{Dualities for weak positive logic and for intuitionistic logic.}
%  \label{fig:int2}
%  \end{figure}
%%%%%%%%%%%%%%%%%%%%%%%%%%%%%%%%%%%%%%%%%%%%%%%%%%%%%%%%%%%%%%%%%%%%%%%%%%%%%%%%%

  Next, we use the notion of $\Pi_1$-preservation to prove a general
  completeness result.

\begin{definition}
  A consequence pair $\phi \cp \psi$ is called \emph{$\Pi_1$-persistent} if
  for every L-space $\topo{X}$,
  $$
    \topo{X} \Vdash \phi \cp \psi
    \quad\text{implies}\quad
    \kappa\topo{X} \Vdash \phi \cp \psi.
  $$
\end{definition}
  
  It is well known that filter and ideal completions preserve all
  (in)equalities (see e.g.~\cite{BirSack}).
  Combining this with Lemmas~\ref{lem:complex-alg1}
  and~\ref{lem:Lspace-lattice-val} we find:
  
\begin{lemma}\label{lem:canonicity}
  Any consequence pair $\psi \cp \chi$ of $\lan{L}$-formulae is
  $\Pi_1$-persistent.
\end{lemma}

\begin{theorem}\label{thm:compl}
  Let $\Gamma$ be a set of consequence pairs.
  Then the logic $\log{L}(\Gamma)$ is sound and complete with respect to the
  following classes of frames:
  \begin{itemize}
    \item $\cat{D\hyphen LFrm}(\Gamma)$ $($descriptive frames validating $\Gamma)$;
    \item $\cat{PLFrm}(\Gamma)$ $($principal L-frames validating $\Gamma)$;
    \item $\cat{LFrm}(\Gamma)$ $($L-frames validating $\Gamma)$.
  \end{itemize}
\end{theorem}
\begin{proof}
  Soundness holds by definition, so we prove completeness.
  Suppose $\phi \not\vdash_{\Gamma} \psi$.
  Then by Theorem~\ref{thm:alg-sem} we can find a lattice $A$
  validating all consequence pairs in $\Gamma$, but not $\phi \cp \psi$.
  As a consequence of Lemma~\ref{lem:Lspace-lattice-val} the L-space
  $\topo{X}$ dual to $A$ validates all consequence pairs in $\Gamma$
  but does not validate $\phi \cp \psi$, thus we find completeness
  with respect to $\cat{D\hyphen LFrm}(\Gamma)$.
  
  Since $\topo{X} \not\Vdash \phi \cp \psi$
  there must exist a clopen valuation $V$ such that
  $(\topo{X}, V) \not\Vdash \phi \cp \psi$.
  Therefore $(\kappa\topo{X}, V) \not\Vdash \phi \cp \psi$.
  Besides, Lemma~\ref{lem:canonicity} implies that
  $\kappa\topo{X}$ validates all consequence pairs in $\Gamma$.
  This implies completeness with respect to $\cat{LFrm}(\Gamma)$.
  Lastly, we note that $\topo{X}$ is a principal L-frame and
  since $V$ is clopen is assigns to each proposition letter a principal filter.
  Thus $\topo{X}$ is a principal L-frame validating $\Gamma$ but not
  $\phi \cp \psi$, proving completeness with respect to
  $\cat{PLFrm}(\Gamma)$.
\end{proof}

\begin{remark}
  Another way to prove Theorem~\ref{thm:compl} is via a Sahlqvist-style
  argument and the correspondence proved in Section~\ref{subsec:corr}.
  This resembles the approach taken by Sambin and Vaccaro~\cite{SamVac89}.
  For the detailed order-topological proof we refer to  \cite[Section 4.2]{Dmi21}. 
  The basic idea is as follows: if a sequent is refuted on some $L$-frame,
  then it is refuted on this frame via a minimal valuation which is closed.
  An analogue of the so-called intersection lemma entails that the value of
  a positive formula on a closed valuation is the intersection of the values
  of this formula on a clopen valuation. This produces a clopen valuation
  refuting the sequent.
\end{remark}

\begin{remark}\label{rem:can}
  The $\Pi_1$-persistence discussed here allows us to move from valuations that interpret proposition letters as clopen filters to valuations that assign to each proposition letter an arbitrary filter. It is analogous to $d$-persistence in intuitionistic and modal
  logics~\cite{BRV01,CZ97}
In the classical setting $d$-persistence allows one to move from clopen valuations to arbitrary valuations, and in the intuitionistic case from valuations into clopen upsets to valuations into 
all upsets. We point out that, while in the distributive setting this corresponds algebraically to canonical extensions, in our setting the corresponding algebraic structure is the $\Pi_1$-completion.
\end{remark}

%  We can now prove completeness of $\lan{L}$ with respect to several classes
%  of frames.
%
%\begin{theorem}\label{thm:completeness}
%  The logic $\log{L}$ is (sound and) complete with respect to the classes of
%  L-spaces, principal L-frames, and all L-frames.
%%  \begin{enumerate}
%%    \item L-spaces;
%%    \item principal L-frames;
%%    \item all L-frames.
%%  \end{enumerate}
%\end{theorem}
%\begin{proof}
%  Soundness follows from Theorem~\ref{thm:soundness}.
%  For completeness, we show that $\phi \not\vdash \psi$
%  implies $\phi \not\Vdash \psi$.
%  So suppose $\phi \not\vdash \psi$, then by Theorem~\ref{thm:alg-sem}
%  we can find a lattice $A$ that does not validate $\phi \cp \psi$.
%  As a consequence of Lemma~\ref{lem:Lspace-lattice-val} the L-space
%  $\topo{X}$ dual to $A$ does not validate $\phi \cp \psi$,
%  so there must exist a clopen valuation $V$ such that
%  $(\topo{X}, V) \not\Vdash \phi \cp \psi$.
%  But this implies that $(\kappa\topo{X}, V) \not\Vdash \phi \cp \psi$.
%  Since $\kappa\topo{X}$ is a (principal) L-frame and every clopen filter of
%  an L-space is principal %(due to Lemma~\ref{lem:closed-filter-principal})
%  we have found a (principal) L-model refuting $\phi \cp \psi$.
%\end{proof}

%--------------------------------------------------------------------------------
\subsection{The First-Order Translation and Sahlqvist Correspondence}\label{subsec:corr}

  In this section we define the standard translation of $\lan{L}$ into a suitable
  first-order logic. We use this to derive a Sahlqvist correspondence
  result. We prove that for every consequence pair $\psi \cp \chi$,
  the collection of L-frames validate $\psi \cp \chi$ are first-order definable.
  Our proof of the correspondence result follows a standard proof from normal
  modal logic, such as found in~\cite[Section 3.6]{BRV01}.
  Thus, it showcases how our duality for lattices allows us to
  transfer classical techniques to the positive non-distributive setting.
  However, it is complicated (or rather, made more interesting) by the
  non-standard interpretation of disjunctions.
  %\blue{Other references? like Celani/Jansana.}

\begin{definition}
  Let $\FOL$ be the single-sorted first-order language which has a
  unary predicate $P_p$ for every proposition letter $p$,
  and a binary relation symbol $R$.
\end{definition}

  Intuitively, the relation symbol of our first-order language accounts
  of the poset structure of L-frames. It is used in the translation of
  disjunctions.
%  The unusual interpretation of the joins
%  in $\lan{L}$ use this order, and this needs to be reflected in the definition
%  of the standard translation.

  If $x, y$ and $z$ are first-order variables, then we can express that
  $x$ is above every lower bound of $y$ and $z$ in the ordering induced by the
  relation symbol $R$ using a first-order sentence.
  In order to streamline notation we abbreviate this as follows:
  $$
    \abovemeet(x; y, z) := \forall w ((wRy \wedge wRz) \to wRx).
  $$
  If $x, y_1, \ldots, y_n$ is a finite set of variables and $n \geq 1$ then we
  define $\abovemeet(x; y_1, \ldots, y_n)$ in the obvious way.
  In particular, $\abovemeet(x; y)$ is simply $x R y$.
  
  We are now ready to define the standard translation.

\begin{definition}\label{def:st}
  Let $x$ be a first-order variable.
  Define the standard translation $\st_x : \lan{L} \to \FOL$ recursively
  via
  \begin{align*}
    \st_x(p) &= P_px \\
    \st_x(\top) &= (x = x) \\
    \st_x(\bot) &= \forall y(y R x) \\
    \st_x(\phi \wedge \psi) &= \st_x(\phi) \wedge \st_x(\psi) \\
    \st_x(\phi \vee \psi)
      &= \exists y \exists z (\abovemeet(x; y, z)
         \wedge \st_y(\phi) \wedge \st_z(\psi))
  \intertext{Furthermore, we define the standard translation of a consequence pair
  $\phi \cp \psi$ as}
    \st_x(\phi \cp \psi) &= \st_x(\phi) \to \st_x(\psi).
  \end{align*}
\end{definition}

  Every L-model $\mo{M} = (X, \leq, V)$ gives rise to a first-order structure
  for $\FOL$: $\leq$ accounts for the interpretation of the
  binary relation symbol, and the interpretation of the unary
  predicates is given via the valuations of the proposition letters.
  We write $\mo{M}^{\circ}$ for the L-model $\mo{M}$ conceived of as
  a first-order structure for $\FOL$.

\begin{proposition}\label{prop:model-corr1}
  For every L-model $\mo{M}$ and state $w$ of $\mo{M}$ we have
  \begin{enumerate}
    \item $\mo{M}, w \Vdash \phi$ iff $\mo{M}^{\circ} \models \st_x(\phi)[w]$;
    %\item $\mo{M} \Vdash \phi$ iff $\mo{M}^{\circ} \models \forall x(\st_x(\phi))$
    \item $\mo{M}, w \Vdash \phi \cp \psi$
          iff $\mo{M}^{\circ} \models \st_x(\phi \cp \psi)[w]$.
%    \item $\mo{M} \Vdash \phi \cp \psi$
%          iff $\forall x(\mo{M}^{\circ} \models \st_x(\phi \cp \psi))$
  \end{enumerate}
\end{proposition}
\begin{proof}
  The first item follows immediately from the definition of the standard
  translation, and the other item follows from the first one.
\end{proof}

  In order to obtain similar results as in Proposition~\ref{prop:model-corr1}
  for frames, we need to quantify the unary predicates in $\FOL$ corresponding
  to the proposition letters. We can do so in a second-order language, say,
  $\SOL$. However, getting a second-order correspondent for a consequence
  pair $\phi \cp \psi$ that is satisfied in a frame if and only if
  $\phi \cp \psi$ is, is not as easy as simply quantifying over all
  possible interpretations of the unary predicates.
  That is, we cannot simply add $\forall P_1 \cdots \forall P_n$ in front
  of $\st_x(\phi \cp \psi)$. Indeed, we wish to only take those interpretations
  into account that arise from a valuation of the proposition letters
  \emph{as filters}.

  Thus we wish to quantify over interpretations of the unary predicates
  corresponding to filters in the underlying frame.
  We can force this by adding conditions that ensure that the $P$'s are
  interpreted as filters in the antecedent of the implication
  $\st_x(\phi) \to \st_x(\psi)$. Then the implication is vacuously true
  for ``illegal'' interpretations of the unary predicates.
  This intuition motivates the following definition of the
  \emph{second-order translation} of a consequence pair.

\begin{definition}\label{def:fil-sot}
  Let $p_1, \ldots, p_n$ be the proposition letters occurring in $\psi$
  and $\chi$, and let $P_1, \ldots, P_n$ denote their corresponding unary
  predicates.
  For each $P_i$, abbreviate
  \begin{equation*}
    \isfil(P_i)
      = \exists w P_iw \wedge \forall x \forall y \forall z(
          (P_iy \wedge P_iz \wedge \abovemeet(x; y, z)) \to P_ix).
  \end{equation*}
  Using this abbreviation, we define the \emph{second order translation}
  of a consequence pair $\psi \cp \chi$ by
  \begin{equation}\label{eq:sot}
    \so(\psi \cp \chi)
      = \forall P_1 \cdots \forall P_n \big(
        (\isfil(P_1) \wedge \cdots \wedge \isfil(P_n) \wedge \st_x(\psi))
        \to \st_x(\chi) \big).
  \end{equation}
\end{definition}

  To disburden notation, we will often abbreviate
  $\isfil(P_1) \wedge \cdots \wedge \isfil(P_n)$ as $\ISFIL$.

  Since all unary predicates in $\so(\psi \cp \chi)$ are in the scope of a quantifier,
  the formula $\so(\psi \cp \chi)$ can be interpreted in a first-order structure with a single relation.
  Therefore, every L-model $\mo{X}$ gives rise to a structure
  $\mo{X}^{\circ}$ for $\SOL$ in which we can interpret second order
  translations.

\begin{lemma}\label{lem:sot}
  For all L-frames $\mo{X} = (X, \ftop, \fmeet)$ and all consequence pairs
  $\psi \cp \chi$ we have
  $$
    \mo{X}, w \Vdash \psi \cp \chi
      \iff \mo{X}^{\circ} \models \so(\psi \cp \chi)[w].
  $$
\end{lemma}
\begin{proof}
  If $\mo{X}, w \Vdash \psi \cp \chi$ then $w \in V(\psi)$
  implies $w \in V(\chi)$
  for every valuation $V$ for $\mo{X}$.
  If any of the $P_i$ is inter\-preted as a subset of $X$ that is not
  a filter, then the implication inside the quantifiers in \eqref{eq:sot}
  is automatically true, because the
  antecedent is false. If all $P_i$ are interpreted as filters,
  then the implication holds because of the assumption.
  The converse is similar.
\end{proof}

  Next, we show how one can use the second-order translation to obtain
  local correspondence results.
  We first define what we mean by local correspondence.

\begin{definition}
  Let $\phi \cp \psi$ be a consequence pair and $\alpha(x)$ a first-order
  formula with free variable $x$. Then we say that $\phi \cp \psi$ and
  $\alpha(x)$ are \emph{local frame correspondents} if for any L-frame
  $\mo{X}$ and any state $w$ we have
  $$
    \mo{X}, w \Vdash \phi \cp \psi
      \iff \mo{X} \models \alpha(x)[w].
  $$
\end{definition}

  Since our language is positive, every formula is upward monotone.
  That is, extending the valuation increases the truth set of formulae.
  
\begin{lemma}\label{lem:increasing-val}
  Let $\mo{X}$ be an L-frame and let $V$ and $V'$ be valuations for $\mo{X}$
  such that $V(p) \subseteq V'(p)$ for all $p \in \Prop$.
  Then for all $\phi \in \lan{L}$ we have $V(\phi) \subseteq V'(\phi)$.
\end{lemma}

  We now prove that every consequence pair has a local correspondent.

\begin{theorem}\label{thm:corr}
  Any consequence pair $\psi \cp \chi$ of $\lan{L}$-formulae
  locally corresponds to a first-order formula with one free variable.
\end{theorem}
\begin{proof}
  We know that $\mo{X}, w \Vdash \psi \cp \chi$ if and only if
  $\mo{X}^{\circ} \models \so(\psi \cp \chi)[w]$.
  Our strategy for obtaining a first-order correspondent is to remove all
  second-order quantifiers from the second-order translation.
  We assume that %this expression has been processed such that
  no two quantifiers bind the same variable.
  
  If $\psi$ is equivalent to $\top$ then as a consequence of Lemma~\ref{lem:increasing-val}
  $\psi \cp \chi$ is equivalent
  to $\top \cp \chi'$, where $\chi'$ is obtained
  from $\chi$ by replacing all proposition letters with $\bot$.
  This, in turn, implies that $\so(\top \cp \chi')$ is a first-order
  correspondent of $\top \cp \chi$, since the lack of proposition
  letters in $\chi'$ implies that there are no second-order
  quantifiers in $\so(\top \cp \chi')$.
  If $\psi$ is equivalent to $\bot$ then $\psi \cp \chi$ is vacuously valid on all L-frames.
  So we may assume that the antecedent does not use $\top$ or $\bot$.
  
  Let $p_1, \ldots, p_n$ be the propositional variables occurring in
  $\psi$, and write $P_1, \ldots, P_n$ for their corresponding unary predicates.
  We assume that every proposition letter that occurs in $\chi$
  also occurs is $\psi$, for otherwise we may replace it by $\bot$ to
  obtain a formula which is equivalent in terms of validity on frames.
  
  \bigskip\noindent
  {\it Step 1.}
  %We start by pre-processing the formula $\so(\psi \cp \chi)$ some more.
%  We use the fact that, after applying the second-order translation,
%  we have classical laws such as distributivity.
%
%  \bigskip\noindent
%  {\it Step 1A.}
  Use equivalences of the form
  $$
    (\exists w(\alpha(w)) \wedge \beta) \leftrightarrow \exists w(\alpha(w) \wedge \beta),
    \qquad
    (\exists w (\alpha(w)) \to \beta) \leftrightarrow\forall w(\alpha(w) \to \beta)
  $$
  to pull out all quantifiers that arise in $\st_x(\psi)$.
  Let $Y := \{ y_1, \ldots, y_m \}$ denote the set of (bound) variables that occur
  in the antecedent of the implication from the second-order translation.
  We end up with a formula of the form
  \begin{equation}\label{eq:post-process11}
    \forall P_1 \cdots \forall P_n \forall y_1 \cdots \forall y_m
    \big( (\ISFIL \wedge \AT \wedge \REL) \to \st_x(\chi) \big).
  \end{equation}
  Here $\ISFIL = \isfil(P_1) \wedge \cdots \wedge \isfil(P_n)$,
  $\AT$ is a conjunction of formulae of the form $P_iz$ and
  $\REL$ is a conjunction of formulae of the form $\abovemeet(z; z', z'')$,
  where $z, z', z'' \in Y \cup \{ x \}$.
  
  \bigskip\noindent
  {\it Step 2.} Next we read off minimal instances of the $P_i$ making the
  antecedent true. Intuitively, these correspond to the smallest valuations
  for the $p_i$ making the antecedent true.
  For each proposition letter $P_i$, let $P_iy_{i_1}, \ldots, P_iy_{i_k}$ be
  the occurrences of $P_i$ in $\AT$ in the antecedent of \eqref{eq:post-process11}.
  We define the valuation of $p_i$ to be the filter generated
  by the (interpretations of) $y_{i_1}, \ldots, y_{i_k}$.
  Formally,
  $$
    \sigma(P_i) := \lambda u. \abovemeet(u; y_{i_1}, \ldots, y_{i_k}).
  $$
  (If $k = 0$, i.e.~there is no variable $y$ with $P_iy$, then we let
  $\sigma(P_i) = \lambda u.(u \neq u)$.)
  Then for each L-model $\mo{M}$ and states $x', y_1', \ldots, y_m'$ in $\mo{M}$
  we have
  $$
    \mo{M}^{\circ} \models \AT \wedge \REL[x, y_1', \ldots, y_m']
    \quad\text{implies}\quad
    \mo{M}^{\circ} \models \forall y (\sigma(P_i)y \to P_iy).
  $$
  
  If we replace each unary predicate $P$ in \eqref{eq:post-process11} with
  $\sigma(P)$, then all conjoints in $\ISFIL$ and $\AT$ become true.
  Writing $[\sigma(P)/P]\st_x(\chi)$ for the formula obtained from $\st_x(\chi)$
  by replacing each instance of a unary predicate $P$ with $\sigma(P)$,
  we arrive at the first-order formula
  \begin{equation}\label{eq:Sahl-2}
    \forall y_1 \cdots \forall y_m \big(
      \REL
      \to [\sigma(P_i)/P_i]\st_x(\chi) \big)
  \end{equation}
  
  \bigskip\noindent
  {\it Step 3.}
  Finally, we claim that for every L-frame $\mo{X}$,
  $\mo{X}^{\circ}$ validates \eqref{eq:post-process11} if and only if
  it validates \eqref{eq:Sahl-2}.
  The implication from left to right is simply an instantiation of
  the quantifiers as filters.
  For the converse, assume that $\mo{M}$ is some model based on $\mo{X}$,
  so that $\mo{M}^{\circ}$ is an extension of $\mo{X}^{\circ}$ giving
  the interpretations of the unary predicates as filters. We may disregard
  the case where any of them is not a filter as that would make the antecedent
  in \eqref{eq:post-process11} false, hence the whole implication true.
  Let $x', y_1', \ldots, y_m'$ be states in $\mo{M}$
  and assume that
  \begin{equation}\label{eq:Sahl-3}
    \mo{M}^{\circ} \models \ISFIL \wedge \AT \wedge \REL[x', y_1', \ldots, y_m'].
  \end{equation}
  We need to show that $\mo{M}^{\circ} \models \st_x(\chi)[x', y_1', \cdots, y_m']$.
  It follows from the assumption that \eqref{eq:Sahl-2} holds that
  $\mo{M}^{\circ} \models [\sigma(P)/P]\st_x(\chi)[x', y_1', \cdots, y_m']$.
  Moreover, as a consequence of \eqref{eq:Sahl-3} we have
  $\mo{M}^{\circ} \models \forall y(\sigma(P)(y) \to Py)$ for all
  $P \in \{ P_1, \ldots, P_n \}$.
  Using Lemma~\ref{lem:increasing-val} it follows that
  $\mo{M}^{\circ} \models \st_x(\chi)[x', y_1', \cdots, y_m']$, as desired.  
\end{proof}

  Let us work through some explicit examples so we can see the proof of
  the theorem in action.
  Recall that $\abovemeet(x; y)$ is simply $xRy$.
  
%\todo{still sloppy notation...
%also: remove the $\forall x$'s!! We go for local frame correspondent, not global!}

\begin{example}
  Consider the formula $p \wedge (q \vee q') \cp (p \wedge q) \vee (p \wedge q')$.
  This corresponds to distributivity; the reverse consequence pair is always valid.
  We temporarily abbreviate $\chi := (p \wedge q) \vee (p \wedge q')$.
  The second-order translation of this formula is
  \begin{align*}
    \so({p \wedge (q \vee q')} \cp \chi)
      &= \forall P \forall Q \forall Q'
        \big( \big[ \isfil(P) \wedge \isfil(Q) \wedge \isfil(Q') \\
        &\wedge Px \wedge \exists y \exists y'(\abovemeet(x; y, y') \wedge Qy \wedge Q'y')\big]
        \to \st_x(\chi) \big)
  \end{align*}
  As per instructions, we rewrite this to
  \begin{equation}\label{eq:exm-distr}
  %\begin{split}
    %\so&(p \wedge (q \vee q') \cp \chi) \\
       \forall P \forall Q \forall Q' \forall y \forall y' \big(
        (\ISFIL \wedge Px \wedge \abovemeet(x; y, y') \wedge Qy \wedge Q'y') 
         \to \st_x(\chi) \big).
  %\end{split}
  \end{equation}
  Next, we obtain
  $\sigma(P) = \lambda u. \abovemeet(u; x)$, which is simply $x R u$.
  Similarly we find $\sigma(Q) = \lambda u. y R u$
  and $\sigma(Q') = \lambda u. y' R u$.
  Plugging these into the antecedent yields
  \begin{equation*}
    [\sigma(P_i)/P_i](\st_x(\chi))
      = \exists z \exists z' \big(\abovemeet(x; z, z')
             \wedge (x R z) \wedge (y R z)
             \wedge (x R z') \wedge (y' R z')\big).
  \end{equation*}
  Thus we find the following first-order correspondent:
  \begin{align*}
    \forall x \forall y \forall y'
      \big( \abovemeet(x; y, y') \to
      \big[& \exists z \exists z'(\abovemeet(x; z, z')
         &&\hspace{-2.5em}\wedge (x R z) \wedge (y R z) \\
         &&&\hspace{-2.5em}\wedge (x R z') \wedge (y' R z'))\big]\big).
  \end{align*}
  Recall that the predicate $R$ is interpreted as the partial order $\fleq$
  underlying an L-frame.
  Furthermore, since $z \wedge z' R x$ and $x R z$ and $x R z'$,
  we find that $z \wedge z' = x$.
  Thus, a state $w$ in an L-frame $(X, \ftop, \fmeet)$ with partial order $\fleq$,
  satisfies distributivity if and only if
  $$
    \forall y \forall y' \big(
      (y \fmeet y' \fleq w) \to
       \exists z \exists z'((w = z \fmeet z') \wedge (y \fleq z) \wedge (y' \fleq z'))\big).
  $$
  Therefore, an L-frame validates distributivity if it is a distributive semilattice.
  Note that here $\fmeet$ and $\fleq$ are L-frame operators while
  the quantifiers and $\to, \wedge$ are connectives from the first-order
  language used to reason about L-frames.
\end{example}

\begin{example}
  Next consider the modularity axiom
  $$
    ((p_1 \wedge p_3) \vee p_2) \wedge p_3 
    \cp (p_1 \wedge p_3) \vee (p_2 \wedge p_3).
  $$
  Writing $\chi$ for the right hand side of the consequence pair,
  after applying Step 1 of the proof of Theorem~\ref{thm:corr} we have
  \begin{equation}\label{eq:exm-mod}
    \forall P_1 \forall P_2 \forall P_3 \forall y \forall y'
      \big((\ISFIL \wedge P_1y \wedge P_3y \wedge P_2y' \wedge P_3x \wedge
      \abovemeet(x; y, y')) \to \st_x(\chi) \big)
  \end{equation}
  We then get $\sigma(P_1) = \lambda u.y R u$,
  $\sigma(P_2) = \lambda u.y' R u$ and
  $\sigma(P_3) = \lambda u.\abovemeet(u; x, y)$.
  Substituting these and leaving out $\ISFIL$ and $\AT$ yields
  \begin{equation*}
  \begin{split}
    &\forall y \forall y' \big[ \big(
      \abovemeet(x; y, y')
             %\wedge (y \leq y)
             \wedge \abovemeet(y; x, y)
             %\wedge (y' \leq y')
             \wedge \abovemeet(x; x, y) \big) \\
      &\to \exists z \exists z' (\abovemeet(x; z, z')
        \wedge (y R z) \wedge \abovemeet(z; x, y)
        \wedge (y' R z') \wedge \abovemeet(z'; x, y)) \big].
  \end{split}
  \end{equation*}
  Leaving out everything that is trivially true, this yields the following
  condition. A world $w$ in an L-frame $(X, \ftop, \fmeet)$ with partial
  order $\fleq$ satisfies the modularity axiom if and only if
  $$
    \forall y \forall y'
      ((y \fmeet y' \fleq w) \to \exists z \exists z'( (z \fmeet z' \fleq w)
        \wedge (y \fleq z)
        \wedge (y' \fleq z')
        \wedge (w \fmeet y \fleq z')).
  $$
  In yet other words, 
  $w$ satisfies modularity if for all $y, y' \in X$ such that
  $y \fmeet y' \fleq w$ we can find
  $z, z'$ above $y, y'$, respectively, such that
  $z \fmeet z' \fleq w$ and $w \fmeet y \fleq z'$.
\end{example}

%================================================================================
\section{Normal Modal Extension}

  We investigate the extension of weak positive logic with two
  modal operators, $\Box$ and $\Diamond$, interpreted via a relation in the
  usual way (see e.g.~\cite[Definition~1.20]{BRV01}).
  As our point of departure we take L-frames with an additional relation.
  We stipulate sufficient conditions on the relation to ensure persistence,
  but we do not enforce any axioms on the modalities.
  It then turns out that $\Box$ preserves finite conjunctions,
  while, as a consequence of the non-standard interpretation of disjunctions,
  $\Diamond$ only satisfies monotonicity.
  This is reminiscent of the modal extension of intuitionistic logic investigated
  by Kojima~\cite{Koj12}.
  The interaction axioms relating $\Box$ and $\Diamond$ are closely aligned
  to Dunn's axioms for positive modal logic \cite{Dun95},
  see Remark~\ref{rem:Dunn}.
  
  After investigating the modal logic from a semantic point of view, we use
  our newly developed intuition to give syntactic definition of the logic
  and its algebraic semantics in Section~\ref{subsec:modal-lattice},
  and a duality in Section~\ref{subsec:modal-l-space}.
  We then extend the definition of the filter- and $\Pi_1$-completion
  to the modal setting and prove completeness for weak positive modal logic
  in Section~\ref{subsec:ext-modal-lat}.

  Finally, in Section~\ref{subsec:modal-corr} we extend the Sahlqvist
  correspondence result for weak positive logic to the modal setting.
  It is no longer the case that any consequence pair is Sahlqvist,
  and we identify as Sahlqvist consequence pairs
  precisely the negation-free Sahlqvist formulae from normal modal logic
  \cite[Definition~3.51]{BRV01}, where the implication is replaced by $\cp$.
  
  One may wonder whether it would be more natural to insist that $\Diamond$
  be normal as well. We do not because the additional conditions required
  to ensure that $\Diamond$ is normal complicate the presentation of the
  semantics and duality. Moreover, in order to make $\Diamond$ normal we
  only need to extend our basic system with the consequence pair
  $$
    \Diamond(p \vee q) \cp \Diamond p \vee \Diamond q,
  $$
  which are Sahlqvist!
  We compute its local correspondent in Example~\ref{exm:normal-diamond}.

%--------------------------------------------------------------------------------
\subsection{Relational Meet-Frames}

  Let $\lan{L}_{\Box\!\Diamond}$ be the language generated by the grammar
  $$
   \phi ::= p \mid \top \mid \bot \mid \phi \wedge \phi \mid \phi \vee \phi
              \mid \Box\phi \mid \Diamond\phi,
  $$
  where $p$ ranges over some set $\Prop$ of proposition letters.
  A \emph{modal consequence pair} is an expression of the form $\phi \cp \psi$,
  where $\phi, \psi \in \lan{L}_{\Box\!\Diamond}$.
  We derive an appropriate notion of modal L-frame, such that the truth set of
  each formula is guaranteed to be a filter.
  
%\begin{definition}%[Egli-Milner order]
%  Let $(X, \leq)$ be a partially ordered set.
%  The \emph{Egli-Milner order} on $\fun{P}X$, the powerset of $X$, is defined
%  by letting $a \sqsubseteq b$ iff
%  \begin{itemize}
%    \item For all $x \in a$ there exists an $y \in b$ such that $x \leq y$; and
%    \item For all $y \in b$ there exists an $x \in a$ such that $x \leq y$.
%  \end{itemize}
%\end{definition}

\begin{definition}\label{def:modal-l-frame}
  A \emph{modal L-frame} is a tuple $(X, \ftop, \fmeet, R)$ where
  $(X, \ftop, \fmeet)$ is an L-frame with underlying partial order $\fleq$,
  and $R$ is a binary relation on $X$ such that:
  \begin{enumerate}
    %\item $R[x]$ is f-convex for each $x$;
    \item \label{it:mlf-1}
          If $x \fleq y$ and $yRz$ then there exists a $w \in X$ such that
          $xRw$ and $w \fleq y$;
    \item \label{it:mlf-2}
          If $x \fleq y$ and $xRw$ then there exists a $z \in X$ such that
          $yRz$ and $w \fleq z$;
    \item \label{it:mlf-3}
          If $(x \fmeet y)Rz$ then there exist $v, w \in X$ such that
          $xRv$ and $yRw$ and $v \fmeet w \fleq z$;
    \item \label{it:mlf-4}
%          If $s \in R[x]$ and $t \in R[y]$ then there exists a
%          $z \in R[x \wedge y]$ such that $s \wedge t \leq z$.
          If $xRv$ and $yRw$ then $(x \fmeet y)R(v \fmeet w)$;
%    \item \label{it:mlf-5}
%          For all $x \in X$ there exists a $y \in X$ such that $xRy$;
%          \blue{Maybe we don't need this!}
    \item \label{it:mlf-6}
          For all $x$, $\ftop Rx$ if and only if $x = \ftop$.
  \end{enumerate}
  A \emph{modal L-model} is a a modal L-frame together with a valuation $V$
  that assigns to each proposition letter a filter of $(X, \ftop, \fmeet)$.
  
  Just like in the propositional case in Section~\ref{subsec:l-frame},
  we can identify a class of frames where formulae can be interpreted
  exclusively as \emph{principal} filters.
  To this end, we define a \emph{principal modal L-frame} as a
  modal L-frame $(X, \ftop, \fmeet, R)$ that additionally satisfies:
  \begin{enumerate}
    \setcounter{enumi}{-1}
    \item \label{it:pmlf-0}
           $(X, \ftop, \fmeet)$ has binary joins and all non-empty meets;
    \setcounter{enumi}{2}
    \renewcommand{\theenumi}{$\arabic{enumi}'$}
    \renewcommand{\labelenumi}{\theenumi.}
    \item \label{it:pmlf-3}
%          If $(x \wedge y)Rz$ then there exist $v, w \in X$ such that
%          $xRv$ and $yRw$ and $v \wedge w \leq z$;
          If $(\bigfmeet x_i)Rz$, where the $i$ range over some index set $I$,
          then there exist $z_i$ such that $x_iRz_i$ for all $i \in I$
          and $\bigfmeet z_i \leq z$;
    \item \label{it:pmlf-4}
          %If $xRv$ and $yRw$ then $(x \wedge y)R(v \wedge w)$.
          If $x_iRy_i$, where $i$ ranges over some index set $I$,
          then $(\bigfmeet x_i)R(\bigfmeet y_i)$;
  \end{enumerate}
  Clearly, items~\eqref{it:pmlf-3} and~\eqref{it:pmlf-4}
  subsume~\eqref{it:mlf-3} and~\eqref{it:mlf-4}.
  A principal modal L-model is a principal modal L-frame with a valuation
  that assigns to each proposition letter a principal filter.
\end{definition}

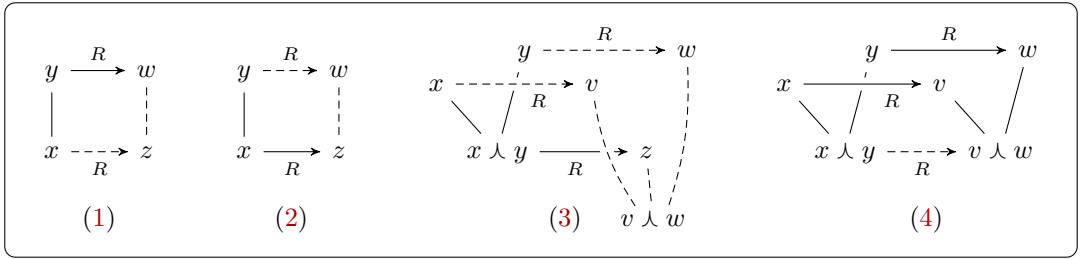
\begin{figure}[h]
  \centering
    \begin{tikzcd}[row sep=large,
                   column sep=large,
                   every matrix/.append style={draw, inner ysep=2pt, 
                     inner xsep=3pt, rounded corners}]
    \text{
    \begin{tikzpicture}[scale=.9, baseline=0]
      \node (x) at (0,0) {$x$};
      \node (y) at (0,1.2) {$y$};
      \node (z) at (1.4,0) {$z$};
      \node (w) at (1.4,1.2) {$w$};
      \draw (x) to (y);
      \draw[densely dashed, -stealth'] (x) to node[below]{\scriptsize{$R$}} (z);
      \draw[-stealth'] (y) to node[above]{\scriptsize{$R$}} (w);
      \draw[densely dashed] (w) to (z);
      \node (label) at (.7,-1) {\eqref{it:mlf-1}};
    \end{tikzpicture}
    \hspace{1.5em}
    \begin{tikzpicture}[scale=.9, baseline=0]
      \node (x) at (0,0) {$x$};
      \node (y) at (0,1.2) {$y$};
      \node (z) at (1.4,0) {$z$};
      \node (w) at (1.4,1.2) {$w$};
      \draw (x) to (y);
      \draw[-stealth'] (x) to node[below]{\scriptsize{$R$}} (z);
      \draw[densely dashed, -stealth'] (y) to node[above]{\scriptsize{$R$}} (w);
      \draw[densely dashed] (w) to (z);
      \node (label) at (.7,-1) {\eqref{it:mlf-2}};
    \end{tikzpicture}
    \hspace{1.5em}
    \begin{tikzpicture}[scale=.9, baseline=0]
      \node (x) at (-.9,1) {$x$};
      \node (y) at (.4,1.5) {$y$};
      \node (xandy) at (0,0) {$x \fmeet y$};
      \node (z) at (2.2,0) {$z$};
      \node (v) at (1.4,1) {$v$};
      \node (w) at (2.8,1.5) {$w$};
      \node (vandw) at (2.3,-1) {$v \fmeet w$};
      \node (vandwleft) at (2.1,-.9) {};
      \node (vandwright) at (2.5,-.9) {};
      \draw[-stealth'] (xandy) to node[below,pos=.4]{\scriptsize{$R$}} (z);
      \draw (x) to (xandy);
      \draw (y) to (xandy);
      \draw[white, line width=5pt, bend right=12] (v) to (vandwleft);
      \draw[densely dashed, bend right=12] (v) to (vandwleft);
      \draw[densely dashed, bend left=10] (w) to (vandwright);
      \draw[white, line width=5pt] (x) to (v);
      \draw[densely dashed, -stealth'] (x) to node[below,pos=.7]{\scriptsize{$R$}} (v);
      \draw[densely dashed, -stealth'] (y) to node[above]{\scriptsize{$R$}} (w);
      \draw[densely dashed] (vandw) to (z);
      \node (label) at (1,-1) {\eqref{it:mlf-3}};
    \end{tikzpicture}
    \hspace{1.5em}
    \begin{tikzpicture}[scale=.9, baseline=0]
      \node (x) at (-.9,1) {$x$};
      \node (y) at (.4,1.5) {$y$};
      \node (xandy) at (0,0) {$x \fmeet y$};
      \node (v) at (1.4,1) {$v$};
      \node (w) at (2.7,1.5) {$w$};
      \node (vandw) at (2.3,0) {$v \fmeet w$};
      \draw (x) to (xandy);
      \draw (y) to (xandy);
      \draw[white, line width=5pt] (x) to (v);
      \draw[-stealth'] (x) to node[below,pos=.75]{\scriptsize{$R$}}(v);
      \draw[-stealth'] (y) to node[above]{\scriptsize{$R$}} (w);
      \draw (v) to (vandw);
      \draw (w) to (vandw);
      \draw[densely dashed, -stealth'] (xandy) to node[below]{\scriptsize{$R$}} (vandw);
      \node (label) at (1.2,-1) {\eqref{it:mlf-4}};
    \end{tikzpicture}
    }
    \end{tikzcd}
  \caption{The four conditions of a modal L-frame.
           Lines denote the poset order, with high nodes being bigger.
           Arrows denote the relation $R$.}
  \label{fig:modal-l-frame}
\end{figure}

\begin{definition}  
  The interpretation of $\lan{L}_{\Box\!\Diamond}$-formulae in a
  (principal) modal L-model $\mo{M}$ is defined via the clauses from
  Definition~\ref{def:interpretation} and
  \begin{align*}
    \mo{M}, x \Vdash \Box\phi
      &\iff \forall y \in X, xRy \text{ implies } \mo{M}, y \Vdash \phi \\
    \mo{M}, x \Vdash \Diamond\phi
      &\iff \exists y \in X \text{ such that } xRy \text{ and } \mo{M}, y \Vdash \phi
  \end{align*}
  Satisfaction and validity of formulae and modal consequence pairs
  are defined as expected.
  In particular, if $\ms{K}$ is a class of modal L-frames and $\phi \cp \psi$
  is a modal consequence pair, then we write
  $\phi \Vdash_{\ms{K}} \psi$ if the consequence pair $\phi \cp \psi$ is valid
  on all frames in $\ms{K}$.
  %\blue{Should we make this precise?}
\end{definition}

  The first four conditions of a modal L-frame are depicted in
  Figure~\ref{fig:modal-l-frame}.
  Observe that~\eqref{it:mlf-1} and~\eqref{it:mlf-6} together imply
  seriality, i.e.~every state has an $R$-successor.
  If $(X, \ftop, \fmeet, R)$ is a modal L-frame, then we define
  for $x \in X$ and filters $a \subseteq X$:
  $$
    R[x] := \{ y \in X \mid xRy \}, \qquad
    [R]a = \{ x \in X \mid R[x] \subseteq a \}, \qquad
    \langle R \rangle a = \{ x \in X \mid R[x] \cap a \neq \emptyset \}.
  $$
  %The first two conditions of 
  Definition~\ref{def:modal-l-frame}\eqref{it:mlf-1} and~\eqref{it:mlf-2}
  together say that $x \leq y$ implies $R[x] \sqsubseteq R[y]$, where
  $\sqsubseteq$ denotes the Egli-Milner order on
  $\fun{P}X$~\cite[Definition~6.2.2]{AbrJun95}.
  Furthermore, if $\mo{M}$ %= (X, \ftop, \fmeet, R, V)$
  is a modal L-model then
  $$
    \llb \Box\phi \rrb^{\mo{M}} = [R] \llb \phi \rrb^{\mo{M}}, \qquad
    \llb \Diamond\phi \rrb^{\mo{M}} = \langle R \rangle \llb \phi \rrb^{\mo{M}}.
  $$
  Next we prove persistence.

\begin{proposition}\label{prop:modal-persistence3} %[Persistence]
  Let $\mo{M} = (X, \ftop, \fmeet, R, V)$ be a (principal) modal L-model.
  Then for each $\phi \in \lan{L}_{\Box\!\Diamond}$ the set
  $\llb \phi \rrb^{\mo{M}}$ is a (principal) filter in $(X, \ftop, \fmeet)$.
\end{proposition}
\begin{proof}
  We assume that $\mo{M}$ is not principal; the case for principal modal
  L-models is similar.
  The proof proceeds by induction on the structure of $\phi$.
  The only non-trivial cases are the modal cases.
  We prove the statement for $\phi = \Diamond\psi$;
  the case $\phi = \Box\psi$ is similar.

%  So suppose $\phi = \Box\psi$,
%  $\mo{M}, x \Vdash \Box\psi$ and $x \leq y$.
%  Then $R[x] \subseteq \llb \psi \rrb^{\mo{M}}$.
%  As a consequence of \eqref{it:mlf-1}, for each $w \in R[y]$ we can find
%  a $z \in R[x]$ such that $z \leq w$. Since $z \in \llb \psi \rrb^{\mo{M}}$,
%  the induction hypothesis implies that $w \in \llb \psi \rrb^{\mo{M}}$.
%  Therefore $R[y] \subseteq \llb \psi \rrb^{\mo{M}}$, so that
%  $\mo{M}, y \Vdash \Box\psi$.
%  
%  Next, suppose that both $x$ and $y$ satisfy $\Box\psi$.
%  We wish to show $\mo{M}, x \wedge y \Vdash \Box\psi$.
%  Let $z$ be an $R$-successor of $x \wedge y$.
%  Then by \eqref{it:mlf-3} we can find $v \in R[x]$ and $w \in R[y]$ such
%  that $v \wedge w \leq z$.
%  By assumption $v$ and $w$ satisfy $\psi$, hence by the induction hypothesis
%  $\mo{M}, v \wedge w \Vdash \psi$ and again by the induction hypothesis
%  $\mo{M}, z \Vdash \psi$. This proves that
%  $R[x \wedge y] \subseteq \llb \psi \rrb^{\mo{M}}$,
%  so that $\mo{M}, x \wedge y \Vdash \Box\psi$ are desired.
  
  Suppose $\phi = \Diamond\psi$,
  $\mo{M}, x \Vdash \Diamond\psi$ and $x \fleq y$.
  Then there exists an $R$-successor $z$ of $x$ satisfying $\psi$,
  and by \eqref{it:mlf-2} we can find an $R$-successor $w$ of $y$
  such that $z \fleq y$. By the induction hypothesis we then find
  $\mo{M}, w \Vdash \psi$ and therefore $\mo{M}, y \Vdash \Diamond\psi$.
  Next, suppose that both $x$ and $y$ satisfy $\Diamond\psi$.
  Then there exist $v \in R[x]$ and $w \in R[y]$ satisfying $\psi$.
  By \eqref{it:mlf-4} we have $(x \fmeet y)R(v \fmeet w)$ and
  by the induction hypothesis $\mo{M}, v \fmeet w \Vdash \psi$.
  Therefore $\mo{M}, x \fmeet y \Vdash \Diamond\psi$.
  Lastly, \eqref{it:mlf-6} implies $\mo{M}, \ftop \Vdash \Diamond\psi$.
  We conclude that $\llb \Diamond\psi \rrb^{\mo{M}}$ is a filter
  in $(X, \ftop, \fmeet)$.
\end{proof}

\begin{remark}
  We could have slightly weakened condition \ref{it:mlf-4} by requiring
  the existence of some $(x \fmeet y)$-successor above $v \fmeet w$.
  We use the current formulation because it aligns more closely
  to the notion of a modal L-space.
  % It was also convenient for Sahlqvist correspondence, but I don't remember where.
%  \blue{We could have used a slightly weaker condition for the last one,
%  but this works out nicely in the correspondence theorem.}
\end{remark}

  Morphisms between modal L-frames and -models are a combination of
  L-morphisms and an adaptation of p-morphisms for positive modal
  logic~\cite{CelJan99}.

\begin{definition}\label{def:mlfm}
  A \emph{bounded L-morphism} from $(X, \ftop, \fmeet, R)$ to
  $(X', \ftop', \fmeet', R')$ is an L-morphism
  $f : (X, \ftop, \fmeet) \to (X', \ftop', \fmeet')$
  such that for all $x, y \in X$ and $z' \in X'$:
  \begin{enumerate}
    \item \label{it:mlfm-1}
          If $xRy$ then $f(x)R'f(y)$;
    \item \label{it:mlfm-2}
          If $f(x)R'z'$ then there exists a $z \in X$ such that
          $xRz$ and $f(z) \fleq z'$;
    \item \label{it:mlfm-3}
          If $f(x)R'z'$ then there exists a $w \in X$ such that
          $xRw$ and $z' \fleq' f(w)$.
  \end{enumerate}
  (See also Figure~\ref{fig:mlfm}.)
  A bounded L-morphism between models is bounded L-morphism between
  the underlying frames that preserves and reflects truth of proposition letters.
\end{definition}

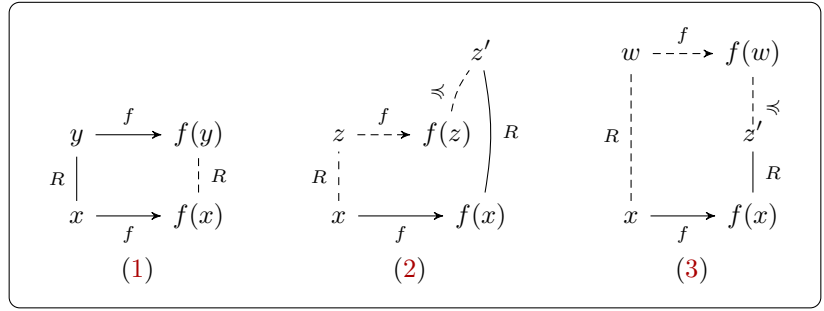
\begin{figure}[h]
  \centering
    \begin{tikzcd}[row sep=large,
                   column sep=large,
                   every matrix/.append style={draw, inner ysep=2pt, 
                     inner xsep=3pt, rounded corners}]
    \text{
    \begin{tikzpicture}[scale=.9, baseline=0]
      \node (x) at (0,0) {$x$};
      \node (y) at (0,1.2) {$y$};
      \node (fx) at (1.8,0) {$f(x)$};
      \node (fy) at (1.8,1.2) {$f(y)$};
      \draw (x) to node[left]{\scriptsize{$R$}} (y);
      \draw[-stealth'] (x) to node[below]{\scriptsize{$f$}} (fx);
      \draw[-stealth'] (y) to node[above]{\scriptsize{$f$}} (fy);
      \draw[densely dashed] (fx) to node[right]{\scriptsize{$R$}} (fy);
      \node (label) at (.9,-.8) {\eqref{it:mlfm-1}};
    \end{tikzpicture}
    \qquad
    \begin{tikzpicture}[scale=.9, baseline=0]
      \node (x) at (0,0) {$x$};
      \node (z) at (0,1.2) {$z$};
      \node (fx) at (2.1,0) {$f(x)$};
      \node (fz) at (1.6,1.2) {$f(z)$};
      \node (zp) at (2.1,2.4) {$z'$};
      \draw[dashed] (x) to node[left]{\scriptsize{$R$}} (z);
      \draw[-stealth'] (x) to node[below]{\scriptsize{$f$}} (fx);
      \draw[densely dashed, -stealth'] (z) to node[above]{\scriptsize{$f$}} (fz);
      \draw[bend right=10] (fx) to node[right]{\scriptsize{$R$}} (zp);
      \draw[densely dashed, bend left=12] (fz) to node[left]{\scriptsize{$\fleq$}} (zp);
      \node (label) at (1.05,-.8) {\eqref{it:mlfm-2}};
    \end{tikzpicture}
    \qquad
    \begin{tikzpicture}[scale=.9, baseline=0]
      \node (x) at (0,0) {$x$};
      \node (w) at (0,2.4) {$w$};
      \node (fx) at (1.8,0) {$f(x)$};
      \node (fw) at (1.8,2.4) {$f(w)$};
      \node (zp) at (1.8,1.2) {$z'$};
      \draw[densely dashed] (x) to node[left]{\scriptsize{$R$}} (w);
      \draw[-stealth'] (x) to node[below]{\scriptsize{$f$}} (fx);
      \draw[densely dashed, -stealth'] (w) to node[above]{\scriptsize{$f$}} (fw);
      \draw (fx) to node[right]{\scriptsize{$R$}} (zp);
      \draw[densely dashed] (fw) to node[right]{\scriptsize{$\fleq$}} (1.8,1.3);
      \node (label) at (.9,-.8) {\eqref{it:mlfm-3}};
    \end{tikzpicture}
    }
    \end{tikzcd}
  \caption{The conditions of a bounded L-morphism.}
  \label{fig:mlfm}
\end{figure}

\begin{proposition}\label{prop:mor-pres-truth-modal}
  Let $\mo{M} = (X, \ftop, \fmeet, R, V)$ and $\mo{M}' = (X', \ftop', \fmeet', R', V')$ be
  two (principal) modal L-models. If $f : \mo{M} \to \mo{M}'$ is a bounded L-morphism,
  $x \in X$ and $\phi \in \lan{L}_{\Box\!\Diamond}$, then
  $$
    \mo{M}, x \Vdash \phi \iff \mo{M}', f(x) \Vdash \phi.
  $$
\end{proposition}
\begin{proof}
  This follows from a routine induction on the structure of $\phi$.
  We showcase the modal cases of the proof.
  Suppose $\phi = \Box\psi$. It follows immediately from
  Definition~\ref{def:mlfm}\eqref{it:mlfm-1} that
  $\mo{M}', f(x) \Vdash \Box\psi$ implies $\mo{M}, x \Vdash \Box\psi$.
  So suppose $\mo{M}, x \Vdash \Box\psi$.
  If $y'$ is an $R'$-successor of $f(x)$, then there exists some
  $z \in X$ such that $xRz$ and $f(z) \fleq y'$.
  This implies $\mo{M}, z \Vdash \psi$ and by induction 
  $\mo{M}', f(z) \Vdash \psi$. Persistence then yields
  $\mo{M}', y' \Vdash \psi$. Therefore $\mo{M}', f(x) \Vdash \Box\psi$.
  
  If $\phi = \Diamond\psi$ Then the preservation from
  left to right follows from Definition~\ref{def:mlfm}\eqref{it:mlfm-1}.
  Conversely, if $\mo{M}', f(x) \Vdash \Diamond\psi$,
  then there exists a $y' \in X'$ such that $f(x)Ry'$ and $\mo{M}', y' \Vdash \psi$.
  By \eqref{it:mlfm-3} we can find some $w \in X$ such that
  $xRw$ and $y' \fleq f(w)$. Persistence implies $\mo{M}', f(w) \Vdash \psi$
  and induction yields $\mo{M}, w \Vdash \psi$.
  Therefore $\mo{M}, x \Vdash \Diamond\psi$.
\end{proof}

  We give a number of modal consequence pairs that are valid in
  every modal L-frame. These motivate the definition of a modal lattice
  in Section~\ref{subsec:modal-lattice}.

\begin{lemma}\label{lem:some-val}
  Let $(X, \ftop, \fmeet, R)$ be a modal L-frame. The following consequence
  pairs are valid:
  \begin{align*}
    \top &\cp \Box\top
      & \top &\cp \Diamond\top
      && %& \Diamond\bot &\cp \bot
      & \text{(modal top)} \\
    \Box(\phi \wedge \psi) &\cp \Box\phi \wedge \Box\psi
      & \Diamond\phi &\cp \Diamond(\phi \vee \psi)
      &&& \text{(monotonicity)} \\
    \Box\phi \wedge \Box\psi &\cp \Box(\phi \wedge \psi)
      &\Diamond \phi \wedge \Box \psi &\cp \Diamond(\phi \wedge \psi)&& 
      & \text{(normality and duality)}
%    \Diamond \phi \wedge \Box \psi &\cp \Diamond(\phi \wedge \psi)
%      &&&& %\Box(\phi \vee \psi) &\cp \Diamond(\phi \vee \psi) \vee \Box\psi
%      & \text{(duality)}
  \end{align*}
\end{lemma}
\begin{proof}
  All of these follow immediately from the definition of the interpretation
  of $\Box$ and $\Diamond$.
  In particular, they do not rely on any of the conditions from 
  Definition~\ref{def:modal-l-frame}.
  %\blue{Should we showcase some of them?}
\end{proof}

  Observe that the consequence pair $\top \cp \Diamond \top$ corresponds
  to seriality, i.e.~the frame condition that every state has an $R$-successor.
  In presence of $\top \cp \Box\top$ and the duality axiom it is
  equivalent to $\Box\phi \cp \Diamond\phi$.

\begin{remark}\label{rem:Dunn}
  The duality axiom in Lemma~\ref{lem:some-val} corresponds to
  one of Dunn's duality axioms for positive modal logic~\cite{Dun95}.
  It seems that the non-standard interpretation of joins makes Dunn's other
  duality axiom, $\Box(\phi \vee \psi) \cp \Box\phi \vee \Diamond\psi$,
  unsuitable in our context.
  On the other hand, we have $\Box\phi \cp \Diamond\phi$,
  %Lemma~\ref{lem:modal-lat-3},
  which is not assumed by Dunn.
  We flag investigation of the connection between the various axioms
  relating $\Box$ and $\Diamond$ as an interesting direction for further
  research.
\end{remark}

%--------------------------------------------------------------------------------
\subsection{Logic and Modal Lattices}\label{subsec:modal-lattice}

  Guided by the validities from Lemma~\ref{lem:some-val}, we define 
  the logic $\log{L}_{\Box\!\Diamond}$ as follows.
  
\begin{definition}\label{def:modal-logic}
  Let $\log{L}_{\Box\!\Diamond}$ be the smallest set of modal consequence pairs
  closed under the axioms and rules from Definition~\ref{def:logic},
  and under:
  $$
  \begin{array}{cccr}
    \phantom{\Box}\top \cp \Box\top
      &&\top \cp \Diamond\top
      %&\Diamond\bot \cp \bot\phantom{\Diamond}
      &\text{(\emph{modal top})} \\ [.4em]
    \dfrac{\phi \cp \psi}{\Box\phi \cp \Box\psi}
      &&\dfrac{\phi \cp \psi}{\Diamond\phi \cp \Diamond\psi}
      &\text{(\emph{Becker's rules})} \\ [1em]
    \Box\phi \wedge \Box\psi \cp \Box(\phi \wedge \psi)
      &&\Diamond\phi \wedge \Box\psi \cp \Diamond(\phi \wedge \psi)
      &\text{(\emph{linearity} and \emph{duality})}
  \end{array}
  $$
  
  If $\Gamma$ is a set of modal consequence pairs then
  $\log{L}_{\Box\!\Diamond}(\Gamma)$ denotes the smallest set of modal consequence pairs
  closed under the axioms and rules above and those in $\Gamma$.
  We write $\phi \vdash_{\Gamma} \psi$ if
  $\phi \cp \psi \in \log{L}_{\Box\!\Diamond}(\Gamma)$
  and $\phi \dashv\vdash_{\Gamma} \psi$ if both $\phi \vdash_{\Gamma} \psi$
  and $\psi \vdash_{\Gamma} \phi$,
  omitting $\Gamma$ if it is empty.
%  If $\Gamma$ is the empty set then we simply write
%  $\phi \vdash \psi$ and $\phi \dashv\vdash \psi$.
\end{definition}

  Observe that Becker's rule together with linearity for
  $\Box$ implies that $\Box$ is a normal modal operator.
  The algebraic semantics of the logic is given by modal lattices.

\begin{definition}\label{def:modal-lat}
  A \emph{modal lattice} is a tuple $(A, \Box, \Diamond)$ consisting of a lattice
  $A$ and two maps $\Box, \Diamond : A \to A$ satisfying for all $a, b \in A$:
  \begin{align*}
    \top
      &= \Box\top
      &  \top
      &= \Diamond\top
      &  %\Diamond\bot
      & \\ %= \bot \\
    \Diamond a
      &\leq \Diamond(a \vee b)
      &\Box(a \wedge b)
      &= \Box a \wedge \Box b
      &\Diamond a \wedge \Box b
      &\leq \Diamond(a \wedge b)
      %& \Box(a \vee b)
      %&\leq \Diamond(a \vee b) \vee \Box b
  \end{align*}
  A \emph{modal lattice homomorphism} from $(A, \Box, \Diamond)$ to
  $(A', \Box', \Diamond')$ is a lattice homomorphism $h : A \to A'$
  such that $h(\Box a) = \Box' h(a)$ and $h(\Diamond a) = \Diamond' h(a)$
  for all $a \in A$.
  We write $\cat{MLat}$ for the category of modal lattice and modal lattice
  homomorphisms.
\end{definition}

\begin{example}
  %\blue{to do: check notation}
  Let $\mo{X} = (X, \ftop, \fmeet, R)$ be a modal L-frame.
  Then $\mo{X}^* := (\fil(X, \ftop, \fmeet), [R], \langle R \rangle)$ is a modal lattice.
  If $\mo{X}$ is principal then
  $\mo{X}^{\ddagger} := (\pfil(X, \ftop, \fmeet), [R], \langle R \rangle)$ is
  a modal lattice.
\end{example}

  Formulae $\phi \in \lan{L}_{\Box\!\Diamond}$ can be interpreted in
  a modal lattice $\amo{A} = (A, \Box, \Diamond)$ with an assignment
  $\sigma : \Prop \to A$. Analogous to Section~\ref{subsec:algsem},
  the interpretation of proposition letters is given by the assignment,
  and the connectives and modalities as interpreted via their counterparts
  in $\amo{A}$.
  This gives rise to validity of formulae and modal consequence pairs in
  a modal lattice $\amo{A}$.
  
  If $\mo{M} = (\mo{X}, V)$ is a modal L-model then $V$ is an assignment
  for $\mo{X}^{*}$ and we write $\mo{M}^{*} = (\mo{X}^{*}, V)$.
  If $\mo{M} = (\mo{X}, V)$ is a principal modal L-model then
  $V$ is an assignment for $\mo{X}^{\ddagger}$ and we let
  $\mo{M}^{\ddagger} = (\mo{X}^{\ddagger}, V)$.
  We obtain the following counterpart of Lemma~\ref{lem:complex-alg1}.
  
\begin{lemma}\label{lem:complex-alg}
  %\blue{check notation}
  Let $\mo{M}$ be a modal L-model, $\mo{N}$ a principal modal L-model,
  and $\phi \in \lan{L}_{\Box\!\Diamond}$.
  Then
  $$
    \llb \phi \rrb^{\mo{M}} = \llp \phi \rrp_{\mo{M}^{*}}
    \quad\text{and}\quad
    \llb \phi \rrb^{\mo{N}} = \llp \phi \rrp_{\mo{N}^{\ddagger}}.
  $$
\end{lemma}
  
  We write $\phi \cdash[\Gamma] \psi$ if any modal lattice that
  validates all consequence pairs in $\Gamma$ also validates
  $\phi \cp \psi$.
  Then we can prove the next theorem in the same way as in
  Section~\ref{subsec:algsem}.
  
\begin{theorem}\label{thm:alg-sem2}
  Let $\Gamma \cup \{ \phi \cp \psi \}$ be a set of modal consequence pairs.
  Then $\phi \vdash_{\Gamma} \psi$ iff $\phi \cdash[\Gamma] \psi$.
\end{theorem}

%--------------------------------------------------------------------------------
\subsection{Modal L-spaces and Duality}\label{subsec:modal-l-space}

  %\todo{Adapt to include top element in frames!}
  
  We define the modal counterpart of L-spaces as follows.

\begin{definition}\label{def:mls}
  A \emph{modal L-space} is a tuple $(X, \ftop, \fmeet, \tau, R)$ such that
  \begin{enumerate}
    \item \label{it:mls-1}
          $(X, \ftop, \fmeet, \tau)$ is an L-space;
    \item \label{it:mls-1b}
          $R$ is a binary relation on $X$ such that
          $\ftop R x$ iff $x = \ftop$ for all $x \in X$;
%          \blue{formerly: for every $x \in X$
%          there exists a $y \in X$ with $xRy$;}
    \item \label{it:mls-2}
          If $a$ is a clopen filter, then so are $[R]a$ and $\langle R \rangle a$;
    \item \label{it:mls-3}
          \label{it:mls-tight}
          For all $x, y \in X$ we have $xRy$ if and only if for all $a \in \cfil\topo{X}$:
          \begin{itemize}
            \item If $x \in [R]a$ then $y \in a$;
            \item If $y \in a$ then $x \in \langle R \rangle a$.
          \end{itemize}
  \end{enumerate}
  Truth and validity in modal L-spaces is defined as usual, using
  clopen valuations.
\end{definition}

  The third item is a condition often seen in the definition of general
  frames. Item \eqref{it:mls-tight} is our counterpart of the tightness
  condition, and has previously been used in~\cite[Section~2]{CelJan99}.
  Next, we prove that each modal L-space has an underlying
  (principal) modal L-frame.

\begin{lemma}\label{lem:mls-succ-closed}
  Let $\topo{X} = (X, \ftop, \fmeet, \tau, R)$ be a modal L-space.
  Then $R[x]$ is closed for all $x \in X$.
\end{lemma}
\begin{proof}
  Suppose $y \notin R[x]$.
  Then there exists a clopen filter $a$ such that either $x \in [R]a$
  and $y \notin a$, or $y \in a$ and $x \notin \langle R \rangle a$.
  In the first case $X \setminus a$ is a clopen neighbourhood
  of $y$ disjoint from $R[x]$. In the second case
  $a$ is a clopen neighbourhood of $y$ disjoint from $R[x]$.
\end{proof}
  
\begin{proposition}\label{prop:mls-mlf}
  Let $\topo{X} = (X, \ftop, \fmeet, \tau, R)$ be a modal L-space.
  Then $(X, \ftop, \fmeet, R)$ is a principal modal L-frame. 
\end{proposition}
\begin{proof}
  We know that L-spaces have all non-empty meets, and hence also binary
  joins, so \eqref{it:pmlf-0} is satisfied.
  Furthermore,~\eqref{it:mlf-6} is satisfied by definition.
  We verify the other conditions from Definition~\ref{def:modal-l-frame},
  starting with \eqref{it:pmlf-4}.

  \medskip\noindent
  {\it Condition \ref{it:pmlf-4}.} \
  Suppose $x_iRy_i$, where $i$ ranges over some index set $I$.
  If $I = \emptyset$ then this condition states $\ftop R \ftop$ which holds by definition,
  so assume that $I \neq \emptyset$.
  By the tightness condition of modal L-spaces, in order to prove
  $(\bigfmeet x_i)R(\bigfmeet y_i)$ it suffices to show that for all
  clopen filters $a$, $\bigfmeet x_i \in [R]a$ implies $\bigfmeet y_i \in a$
  and $\bigfmeet y_i \in a$ implies $\bigfmeet x_i \in \langle R \rangle a$.
  
  First assume $\bigfmeet x_i \in [R]a$. Since $[R]a$ is a clopen filter
  and $\bigfmeet x_i \fleq x_j$ for each $j \in I$ we have
  $x_j \in [R]a$, so $R[x_j] \subseteq a$. By assumption $x_jRy_j$,
  so $y_j \in a$ for all $j \in I$. Since $a$ is a clopen filter it is
  principal, hence $\bigfmeet y_i \in a$.
  Second, suppose $\bigfmeet y_i \in a$.
  Then $y_j \in a$ for all $j \in I$,
  which implies $x_j \in \langle R \rangle a$ for all $j \in I$.
  Since $\langle R \rangle a$ is a clopen filter, hence principal,
  we find $\bigfmeet x_i \in \langle R \rangle a$.
  
  \medskip\noindent
  {\it Condition \ref{it:mlf-1}.} \
  Let $x \fleq y$ and $yRz$.
  Suppose towards a contradiction that there exists no $w \in X$
  such that $xRw$ and $w \fleq z$.
  Let $x' = \bigfmeet R[x]$ be the minimal element in $R[x]$
  (which is an $R$-successor of $x$ by \eqref{it:pmlf-4}).
  Then $x' \not\fleq z$, so we can find a clopen filter $a$ containing
  $x'$ such that $z \notin a$. This implies $R[x] \subseteq a$,
  so that $x \in [R]a$, but $y \notin [R]a$ because
  $yRz$ and $z \notin a$. As $x \fleq y$ this violates the
  fact that $[R]a$ is a filter.

  \medskip\noindent
  {\it Condition \ref{it:mlf-2}.} \
  Let $x \fleq y$ and $xRw$. Suppose towards a contradiction that
  there exists no $z \in X$ such that $yRz$ and $w \fleq z$.
  Then $R[y] \cap {\uparrow}w = \emptyset$. Both $R[y]$ and
  ${\uparrow}z$ are closed, as a consequence of Lemmas~\ref{lem:mls-succ-closed}
  and~\ref{lem:closed-upward-closure} and the fact that singletons in
  a Stone spaces are always closed.
  Therefore we can find a clopen filter $a$ containing ${\uparrow}z$
  which is disjoint from $R[y]$. This implies that
  $x \in \langle R \rangle a$ while $y \notin \langle R \rangle a$.
  Since $x \fleq y$ this contradicts the fact that $\langle R \rangle a$
  is a filter.
  
  \medskip\noindent
  {\it Condition \ref{it:pmlf-3}.} \
  Finally, let $\{ x_i \mid i \in I \}$ be some collection of
  elements of $X$ and suppose $(\bigfmeet x_i)Rz$.
  If $I$ is empty then $\bigfmeet x_i = \ftop$ and
  Definition~\ref{def:mls}\eqref{it:mls-1b} implies $z = \ftop$,
  %(take $a = \{ \ftop \}$),
  and the empty collection witnesses truth of~\eqref{it:pmlf-3}.
  So assume $I \neq \emptyset$.
  Since $\bigfmeet x_i \fleq x_j$ for all $j \in I$,
  condition \eqref{it:mlf-2} implies that $R[x_j] \neq \emptyset$ for all
  $j \in I$. As a consequence of \eqref{it:pmlf-4} there is a smallest
  element $z_j := \bigfmeet R[x_j]$ in each $R[x_j]$.
  We claim that $\bigfmeet z_j \leq z$.
  Suppose not, then there is a clopen filter $a$ containing $\bigfmeet z_j$
  such that $z \notin a$. This implies $x_j \in [R]a$ for all $j \in I$,
  but $\bigfmeet x_i \notin [R]a$ because $(\bigfmeet x_i)Rz$ and $z \notin a$.
  But this contradicts the fact that $[R]a$ is principal filter.
\end{proof}

  This proposition motivates the following definition.

\begin{definition}
  Let $\topo{X} = (X, \ftop, \fmeet, \tau, R)$ be a modal L-space.
  Then we write $\pi\topo{X} = (X, \ftop, \fmeet, R)$ for the underlying
  principal modal L-frame, and $\kappa\topo{X} = (X, \ftop, \fmeet, R)$ for the
  underlying principal modal L-frame regarded as a modal L-frame.
\end{definition}

  While they may appear the same, the difference between $\pi\topo{X}$ and
  $\kappa\topo{X}$ lies in the valuations they allow for. While valuations of
  $\pi\topo{X}$ necessarily interpret proposition letters as principal filters,
  a valuation for $\kappa\topo{X}$ can assign any filter to a proposition letter.
  As a consequence, both frames differ in terms of \emph{validity}.

  For future reference, we prove the following lemma about
  modal L-spaces. It states that the action of $[R]$ and $\langle R \rangle$
  on any filter is determined by their action on clopen filters.

\begin{lemma}\label{lem:Rc}
  Let $\topo{X} = (X, \ftop, \fmeet, \tau, R)$ be a modal L-space.
  The for every closed filter $c \in \kfil(\topo{X})$, %we have
  \begin{equation}\label{eq:Rc1}
    \rbox{R} c
      = \bigcap \big\{ \rbox{R} a \mid a \in \cfil(\topo{X}),
         c \subseteq a \big\} %\label{eq:Rc1}
    \quad\text{and}\quad
    \rdiamond{R} c
      = \bigcap \big\{ \rdiamond{R} a \mid a \in \cfil(\topo{X}), c \subseteq a \big\}.
  \end{equation}
  Furthermore, for every filter $d \in \fil(\topo{X})$ we have
  \begin{equation}\label{eq:Rc3}
    \rbox{R}d 
      = \bigGenFil \big\{ \rbox{R}c \mid c \in \kfil(\topo{X}), c \subseteq d \big\}
    \quad\text{and}\quad
    \rdiamond{R}d 
      = \bigGenFil \big\{ \rdiamond{R}c \mid c \in \kfil(\topo{X}), c \subseteq d \big\}.
  \end{equation}
\end{lemma}
\begin{proof}
  The left-to-right inclusion of the first equality follows from the fact
  that $c \subseteq a$ implies $\rbox{R}c \subseteq \rbox{R}a$.
  For the converse, suppose $x \in \topo{X}$ is such that
  $x \notin \rbox{R}c$. Then $R[x] \not\subseteq c$, so there is some
  $y \in R[x]$ such that $y \notin c$.
  Since $c$ is a closed filter, it is the intersection of all clopen filters
  that contain it. A routine compactness argument yields a clopen
  filter $b$ containing $c$ such that $y \notin b$.
  Therefore $R[x] \not\subseteq b$, so $x \notin \rbox{R}b$
  and hence $x$ is not in the right hand side of the equality.
  
  For the second equality, it suffices to prove the right-to-left inclusion as well.
  Suppose $x \notin \rdiamond{R}c$. Then $R[x] \cap c = \emptyset$.
  Since $R[x]$ is closed, we can use compactness to find some clopen filter $a$
  containing $c$ disjoint from $R[x]$, and the argument proceeds as above.
  
  Next, consider~\eqref{eq:Rc3}.
  We know that $c \subseteq d$ implies $\rbox{R}c \subseteq \rbox{R}d$,
  and since the latter is a filter containing $\rbox{R}c$
  for all closed filters $c \subseteq d$, it contains the filter generated
  by them. This proves $\supseteq$.
  Conversely, suppose $x \in \rbox{R}d$.
  Then $R[x] \subseteq d$.
  Since $d$ is up-closed we also have $\uparr R[x] \subseteq d$.
  The set $\uparr R[x]$ is closed by Lemmas~\ref{lem:mls-succ-closed}
  and~\ref{lem:closed-upward-closure},
  and it is a filter because $R[x]$ is closed under $\fmeet$.
  Thus we have found a closed filter $c := \uparr R[x]$ contained in $d$ such that
  $x \in \rbox{R}c$. This proves~$\subseteq$.
  
  The right-to-left inclusion of the second equality in~\eqref{eq:Rc3}
  is easy again. For $\subseteq$,
  suppose $x \in \rdiamond{R}d$. Then $R[x] \cap d \neq \emptyset$,
  so we can find some $y \in R[x] \cap d$.
  Now $\uparr y$ is a closed filter contained in $d$
  such that $x \in \rdiamond{R}(\uparr y)$, witnessing $\subseteq$.
\end{proof}

  Since we know that each modal L-space has an underlying modal L-frame,
  we can now conveniently define the morphisms between them as follows.

\begin{definition}
  A \emph{modal L-space morphism} from $(X, \ftop, \fmeet, \tau, R)$ to
  $(X', \ftop', \fmeet', \tau', R')$ is a function $f : X \to X'$ such that
  $f : (X, \ftop, \fmeet, \tau) \to (X', \ftop', \fmeet', \tau')$ is an L-space morphism
  and $f : (X, \ftop, \fmeet, R) \to (X', \ftop', \fmeet', R')$ is a bounded L-morphism.
  We denote the resulting category by $\cat{MLSpace}$.
\end{definition}

  We work towards a duality between modal lattices and modal L-spaces.

\begin{proposition}
  If $\topo{X} = (X, \ftop, \fmeet, \tau, R)$ is a modal L-space
  then $(\cfil\topo{X}, [R], \langle R \rangle)$ is a modal lattice.
  Moreover, if $f : \topo{X} \to \topo{X}'$ is a modal L-space morphism,
  then
  $$
    \cfil f = f^{-1}
      : (\cfil\topo{X}', [R'], \langle R' \rangle)
        \to (\cfil\topo{X}, [R], \langle R \rangle)
  $$
  is a modal lattice homomorphism.
\end{proposition}
\begin{proof}
  The maps $[R], \langle R \rangle$ are functions on $\cfil\topo{X}$
  by definition.
  It follows from Proposition~\ref{prop:mls-mlf} and
  Lemma~\ref{lem:some-val} that they satisfy the conditions from
  Definition~\ref{def:modal-lat}.
  
  If $f$ is a modal L-space morphism then in particular it is an
  L-space morphism, so $\cfil f$ is a lattice homomorphism
  from $\cfil\topo{X}'$ to $\cfil\topo{X}$.
  So we only have to show that $f^{-1}$ preserves the modalities.
  This can be proven in the same way as in
  Proposition~\ref{prop:mor-pres-truth-modal}.
\end{proof}

  We now show how to turn a modal lattice into a modal L-space.

\begin{definition}
  Let $\amo{A} = (A, \Box, \Diamond)$ be a modal lattice.
  Then we define the binary relation $R_A$ on $\bfil A$ by
  $$
    p R_A q \iff \Box^{-1}(p) \subseteq q \subseteq \Diamond^{-1}(p).
  $$
\end{definition}

\begin{lemma}\label{lem:modal-lat-2}
  Let $(A, \Box, \Diamond)$ be a modal lattice and $p \in \fil A$.
  %If $\Diamond\top \in p$
  Then $\Box^{-1}(p)$ is a filter in $\fil A$ and $pR_A\Box^{-1}(p)$.
\end{lemma}
\begin{proof}
  The set $\Box^{-1}(p)$ is a filter because $\Box : A \to A$ preserves
  conjunctions and the top element.
  To show $pR_A\Box^{-1}(p)$ we need to prove
  $\Box^{-1}(p) \subseteq \Box^{-1}(p) \subseteq \Diamond^{-1}(p)$.
  The left inclusion is trivial.
  For the right one, suppose $a \in \Box^{-1}(p)$.
  Then $\Box a \in p$. Besides, $\Diamond\top = \top \in p$.
  So the duality axiom implies
  $\Diamond\top \wedge \Box a \leq \Diamond(\top \wedge a) = \Diamond a \in p$,
  hence $a \in \Diamond^{-1}(p)$.
\end{proof}

\begin{lemma}\label{lem:modal-lat-3}
  Let $(A, \Box, \Diamond)$ be a modal lattice. Then for each
  $a \in A$ we have
  $$
    [R_A]\theta_A(a) = \theta_A(\Box a)
    \quad\text{and}\quad
    \langle R_A \rangle\theta_A(a) = \theta_A(\Diamond a).
  $$
\end{lemma}
\begin{proof}
  Suppose $p \in [R_A]\theta_A(a)$.
  By Lemma~\ref{lem:modal-lat-2} we have $pR_A\Box^{-1}(p)$ and
  by assumption $a \in \Box^{-1}(p)$. This implies that $\Box a \in p$
  and therefore $p \in \theta_A(\Box a)$.
  For the reverse inclusion, suppose $p \in \theta_A(\Box a)$.
  Then $a \in \Box^{-1}(p)$, so $pR_Aq$ implies $a \in q$,
  and hence $p \in [R_A]\theta_A(a)$.
  
  Next, suppose $p \in \langle R_A \rangle\theta_A(a)$.
  Then there exists a filter $q$ such that $pR_Aq$ and $a \in q$.
  By definition of $R_A$ this implies $a \in \Diamond^{-1}(p)$ and hence
  $\Diamond a \in p$, so $p \in \theta_A(\Diamond a)$.
  Conversely, suppose $p \in \theta_A(\Diamond a)$.
  Let $q$ be the filter generated by $\Box^{-1}(p)$ and $a$.
  We claim that $c \in q$ implies $\Diamond c \in p$.
  To see this, note that for each $c \in q$ there exists some $d \in \Box^{-1}(p)$
  such that $d \wedge a \leq c$.
  We then have $\Box d \in p$, and $\Diamond a \in p$ by assumption,
  so $\Box d \wedge \Diamond a \in p$.
  Since $\Box d \wedge \Diamond a \leq \Diamond(d \wedge a) \leq \Diamond c$
  we find $\Diamond c \in p$.
  The filter $q$ is nonempty because it contains $\top$.
  Furthermore, we have $\Box^{-1}(p) \subseteq q$ by definition of $q$
  and we just showed that $c \in q$ implies $\Diamond c \in p$,
  so that $q \subseteq \Diamond^{-1}(p)$.
  This proves $pR_Aq$.
  By design $a \in q$ so $q$ witnesses the fact that
  $p \in \langle R_A \rangle \theta_A(a)$.
\end{proof}

\begin{lemma}
  If $\amo{A} = (A, \Box, \Diamond)$ is a modal lattice, then
  $\amo{A}_* := (\fil A, A, \cap, \tau_A, R_A)$ is a modal L-space.
\end{lemma}
\begin{proof}
  We verify the condition from Definition~\ref{def:mls}.
  Item \eqref{it:mls-1} follows from Theorem~\ref{thm:duality-lat-lspace}.
  Item \eqref{it:mls-1b} follows from Lemma~\ref{lem:modal-lat-2}.
  Item \eqref{it:mls-2} follows from Lemma~\ref{lem:modal-lat-3}.
  Item \eqref{it:mls-3} follows from the definition
  of $R_A$ and the fact that each clopen filter is of the form
  $\theta_A(a)$.
\end{proof}

\begin{lemma}
  Let $h : \amo{A} \to \amo{A}'$ be a modal lattice homomorphism.
  Then $\fil h = h^{-1} : \amo{A}'_* \to \amo{A}_*$ is a bounded
  L-space morphism.
\end{lemma}
\begin{proof}
  It follows from the duality between lattices and L-spaces that
  $h^{-1}$ is an L-space morphism, so we only have to verify the
  three conditions from Definition~\ref{def:mlfm}.
  We write $R'$ and $R$ for the relations from $\amo{A}'_*$ and
  $\amo{A}_*$.
  
  \begin{enumerate}
    \item Let $p'$ and $q'$ be filters of $\amo{A}'$ (elements of $\amo{A}'_*$)
          such that $p'R'q'$. In order to prove that
          $h^{-1}(p')Rh^{-1}(q')$ we have to show that
          $\Box^{-1}(h^{-1}(p')) \subseteq h^{-1}(q') \subseteq \Diamond^{-1}(h^{-1}(p'))$.
          Let $a \in \Box^{-1}(h^{-1}(p'))$.
          Then $\Box a \in h^{-1}(p')$ so $\Box'(h(a)) = h(\Box a) \in p'$.
          Therefore $h(a) \in (\Box')^{-1}(p')$, and since $p'R'q'$ this implies
          $h(a) \in q'$, so that $a \in h^{-1}(q')$.
          Next, if $a \in h^{-1}(q')$ then $h(a) \in q'$, so
          $h(\Diamond a) = \Diamond' h(a) \in p'$.
          Therefore $\Diamond a \in h^{-1}(p')$ so that $a \in \Diamond^{-1}(h^{-1}(p'))$.
    \item Suppose $h^{-1}(p')Rq$.
%          Then $\Diamond\top \in h^{-1}(p')$, so
%          $\Diamond'\top' = \Diamond' h(\top) = h(\Diamond\top) \in p'$.
          %
          Lemma~\ref{lem:modal-lat-2} then implies that $p'R'(\Box')^{-1}(p')$.
          So it suffices to show that
          $h^{-1}((\Box')^{-1}(p')) \subseteq q$.
          To this end, suppose $a \in h^{-1}((\Box')^{-1}(p'))$.
          Then $h(\Box a) = \Box' h(a) \in p'$, so $\Box a \in h^{-1}(p')$.
          Since $h^{-1}(p')Rq$ this implies $a \in q$, as desired.
    \item Suppose $h^{-1}(p')Rq$. Then ${\uparrow}h[q]$ is a filter
          (since $h$ is a lattice homomorphism it is non-empty and closed under meets).
          Let $q'$ be the filter generated by ${\uparrow}h[q]$ and
          $(\Box')^{-1}(p')$. Then $q \subseteq h^{-1}(q')$ by construction,
          so it suffices to show that $p'R'q'$.
          We have $(\Box')^{-1}(p') \subseteq q'$ by definition,
          so we only have to show that $q' \subseteq (\Diamond')^{-1}(p')$.
          Let $a' \in q'$. Then there are $b' \in (\Box')^{-1}(p')$
          and $c \in q$ such that $b' \wedge' h(c) \leq a'$.
          We find
          $$
            \Box'b' \wedge h(\Diamond c) = \Box'b' \wedge \Diamond' h(c) \leq \Diamond a'.
          $$
          By construction we have $\Box'b' \in p'$.
          Furthermore, $c \in q$ implies $\Diamond c \in h^{-1}(p')$
          and hence $h(\Diamond c) \in p'$.
          Therefore $\Diamond a' \in p'$, and consequently
          $p'R'q'$.  \qedhere
  \end{enumerate}
\end{proof}

  Using the above lemmas we now establish a duality for modal lattices.

\begin{theorem}
  The duality between $\cat{Lat}$ and $\cat{LSpace}$ lifts to a duality
  $$
    \cat{MLat} \equiv^{\op} \cat{MLSpace}.
  $$
\end{theorem}
\begin{proof}
  Let $\topo{X} = (X, \leq, \tau, R)$ be a modal L-space.
  This gives rise to the modal lattice
  $\topo{X}^* = (\cfil \topo{X}, [R], \langle R \rangle)$,
  which in turn yields a modal L-space
  $(\topo{X}^*)_* = (\fil \cfil \topo{X}, \cfil\topo{X}, \cap, \tau_{\cfil\topo{X}}, R_{\cfil\topo{X}})$.
  As a consequence of the duality for lattices (Theorem~\ref{thm:duality-lat-lspace})
  we know that $(X, \ftop, \fmeet, \tau)$ is isomorphic to $(\bfil\cfil\topo{X}, A, \cap, \tau_{\cfil\topo{X}})$
  via $x \mapsto \eta_{\topo{X}}(x) = \{ a \in \cfil\topo{X} \mid x \in a \}$.
  So we only have to show that $R$ and $R_{\cfil\topo{X}}$ coincide.
  This can be seen as follows:
  \begin{align*}
    xRy
      &\iff \forall a \in \cfil\topo{X}
            \big( x \in [R]a \text{ implies } y \in a
                  \text{ and } y \in a \text{ implies } \langle R \rangle a \big) \\
      &\iff \forall a \in \cfil\topo{X}
            \big([R]a \in \eta_{\topo{X}}(x)
                 \text{ implies } a \in \eta_{\topo{X}}(y)
                 \text{ and } a \in \eta_{\topo{X}}(y)
                 \text{ implies } \langle R \rangle a \in \eta_{\topo{X}}(x) \big) \\
%      &\iff \forall a \in \cfil\topo{X}
%            \big(a \in [R]^{-1}(\eta_{\topo{X}}(x))
%              \text{ implies } a \in \eta_{\topo{X}}(y)
%            \text{ and } a \in \eta_{\topo{X}}(y)
%              \text{ implies } a \in \langle R \rangle^{-1}(\eta_{\topo{X}}(x)) \big) \\
      &\iff \forall a \in \cfil\topo{X}
            \big( [R]^{-1}(\eta_{\topo{X}}(x)) \subseteq \eta_{\topo{X}}(y)
            \subseteq \langle R \rangle^{-1}(\eta_{\topo{X}}(x)) \big) \\
      &\iff \eta_{\topo{X}}(x) R_{\cfil\topo{X}} \eta_{\topo{X}}(y)
  \end{align*}
  
  Next, let $\amo{A} = (A, \Box, \Diamond)$ be a modal lattice,
  $\amo{A}_* = (\fil A, R_A)$ and
  $(\amo{A}_*)^* = (\cfil\fil A, [R_A], \langle R_A \rangle)$.
  Then the duality for lattices from Theorem~\ref{thm:duality-lat-lspace}
  tells us that $A$ and $\cfil\fil A$ are isomorphic via
  $a \mapsto \theta_A(a) = \{ p \in \fil A \mid a \in p \}$,
  so we just have to show that $\Box$ coincides with $[R_A]$ and
  $\Diamond$ coincides with $\langle R_A \rangle$.
  That is, we have to show that $\theta_A(\Box a) = [R_A]\theta_A(a)$ and
  $\theta_A(\Diamond a) = \langle R_A \rangle \theta_A(a)$.
  But we have already proven that in Lemma~\ref{lem:modal-lat-3}.
  
  The two paragraphs above establish the duality on objects.
  The duality for morphisms follows immediately from Theorem~\ref{thm:duality-lat-lspace}
  and the fact that all our categories are concrete.
\end{proof}

%--------------------------------------------------------------------------------
\subsection{Completions of modal lattices}\label{subsec:ext-modal-lat}

%  In the following, I will use $a \genFil b$ for the filter generated by $a \cup b$
%  and $\bigGenFil \{ a_i \mid i \in I \}$ for the filter generated
%  by $\bigcup \{ a_i \mid i \in I \}$ (where $a, a_i, b$ are filters).

%
  %We first prove an auxiliary lemma.

  We extend the $\Pi_1$-completion to modal lattices.
  Lemma~\ref{lem:Rc} suggests the following definition of completions of a modal lattice.

\begin{definition}
  Let $\ms{A} = (A, \Box, \Diamond)$ be a modal lattice.
  Let $i : A \to \FiltComp(A)$ be the filter completion of $A$,
  and $j : A \to \IdComp(A)$ the ideal completion.
%  and $j : A \to \PiComp(A)$ the $\Pi_1$-completion.
  \begin{enumerate}
    \item The \emph{filter completion} of $\ms{A}$ is the modal lattice
          $\FiltComp(\ms{A}) = (\FiltComp(A), \Box^{\FiltComp}, \Diamond^{\FiltComp})$
          where
          \begin{align*}
            \Star^{\FiltComp} c
              &= \bigwedge \{ i(\Star a) \mid a \in A \text{ and } c \leq i(a) \}
                 \qquad\text{for } \Star \in \{ \Box, \Diamond \}.
%            \Box^{\FiltComp} c
%              &= \bigwedge \{ i(\Box a) \mid a \in A \text{ and } c \leq i(a) \}, \\
%            \Diamond^{\FiltComp} c
%              &= \bigwedge \{ i(\Diamond a) \mid a \in A \text{ and } c \leq i(a) \}.
          \end{align*}
    \item The \emph{ideal completion} of $\ms{A}$ is the modal lattice
          $\IdComp(\ms{A}) = (\IdComp(A), \Box^{\IdComp}, \Diamond^{\IdComp})$
          where
          \begin{align*}
            \Star^{\IdComp} c
              &= \bigvee \{ j(\Star a) \mid a \in A \text{ and } j(a) \leq c \}
                 \qquad\text{for } \Star \in \{ \Box, \Diamond \}.
%            \Box^{\IdComp} c
%              &= \bigvee \{ j(\Box a) \mid a \in A \text{ and } j(a) \leq c \}, \\
%            \Diamond^{\IdComp} c
%              &= \bigvee \{ j(\Diamond a) \mid a \in A \text{ and } j(a) \leq c \}.
          \end{align*}
    \item The \emph{$\Pi_1$-completion} of $\ms{A}$ is $\IdComp(\FiltComp(\ms{A}))$,
          i.e.~the composition of the filter and the ideal completion.
  \end{enumerate}
\end{definition}

  Concretely, if we view $\FiltComp(A)$ as sitting inside $\PiComp(A)$,
  then the \emph{$\Pi_1$-completion} of $\ms{A}$ is the modal
  lattice $\PiComp(\ms{A}) = (\PiComp(A), \Box^{\PiComp}, \Diamond^{\PiComp})$
  where
  \begin{equation}\label{eq:Pi1-comp}
    \Star^{\PiComp} d
      = \bigvee \{ \Star^{\FiltComp} c \mid c \in \FiltComp(A)
                                       \text{ and } c \leq d \}
         \qquad\text{for } \Star \in \{ \Box, \Diamond \}.
%    \Box^{\PiComp} d
%      &= \bigvee \{ \Box^{\FiltComp} c \mid c \in \FiltComp(A)
%                                       \text{ and } c \leq d \}, \\
%    \Diamond^{\PiComp} d
%      &= \bigvee \{ \Diamond^{\FiltComp} c \mid c \in \FiltComp(A)
%                                               \text{ and } c \leq d \}.
  \end{equation}
  
  Recall that the modal L-frame underlying a modal L-space is principal,
  so both the collection of closed filters as well as the collection of
  all filters form modal lattices.
  We already proved (for the duality result) that the function
  $\theta : \amo{A} \to (\amo{A}_*)^*$ satisfies $\theta(\Box a) = [R_A]\theta(a)$
  and $\theta(\Diamond a) = \langle R_A \rangle \theta(a)$,
  so we can use that.

\begin{proposition}\label{prop:mod-comp4}
  Let $\ms{A} = (A, \Box, \Diamond)$ be a modal lattice and
  $\topo{X}$ its dual modal L-space.
  Then:
  \begin{enumerate}
    \item $\FiltComp(\ms{A})$ is isomorphic to the modal lattice of
          closed filters of $\topo{X}$, i.e.~to $(\pi\topo{X})^{\ddagger}$;
    \item $\PiComp(\ms{A})$ is isomorphic to the modal lattice of
          filters of $\topo{X}$, i.e.~to $(\kappa\topo{X})^*$.
  \end{enumerate}
\end{proposition}
\begin{proof}
  Let $\topo{X} = (X, \ftop, \fmeet, \tau, R_A)$ be the modal L-space dual
  to $\amo{A}$. By Proposition~\ref{prop:lattice-completions} we can identify
  the filter extension of $A$ with the lattice of closed filters of
  $\topo{X}$ with inclusion $\theta : A \to \kfil(\topo{X})$.
  In $\kfil(\topo{X}$, the top, bottom and meet are given by $X$, $\{ \ftop \}$
  and intersection, and the join of a collection of filters is the smallest
  filter containing their union.
  So we only have to verify that for all $c \in \kfil(\topo{X})$ we
  have $\Box^{\FiltComp}c = \rbox{R_A}c$ and $\Diamond^{\FiltComp}c = \rdiamond{R_A}c$.
  Thus compute
  \begin{align*}
    \Box^{\FiltComp}c
%      &= \bigwedge \big\{ \theta(\Box a) \mid a \in A, c \leq \theta(a) \big\}
%         &\text{(Definition of $\Box^{\FiltComp}$)} \\
      &= \bigcap \big\{ \theta(\Box a) \mid a \in A, c \subseteq \theta(a) \big\}
         &\text{(Definition of $\Box^{\FiltComp}$)} \\
      &= \bigcap \big\{ \rbox{R_A}\theta(a) \mid a \in A, c \subseteq \theta(a) \big\}
         &\text{(Lemma~\ref{lem:modal-lat-3})} \\
      &= \bigcap \big\{ \rbox{R_A}b \mid b \in \cfil(\topo{X}), c \subseteq b \big\}
         &\text{($\theta$ is an iso from $A$ to $\cfil(\topo{X})$)} \\
      &= \rbox{R_A}c.
         &\text{(Lemma~\ref{lem:Rc})}
  \end{align*}
  Similarly we find $\Diamond^{\FiltComp}c = \rdiamond{R_A}c$ for all
  $c \in \kfil(\topo{X})$.

  For the second item we adopt a similar strategy.
  Using Proposition~\ref{prop:lattice-completions},
  we identify the $\Pi_1$-completion of $A$ with $\fil(\topo{X})$,
  the filters of $\topo{X}$ with operators as in~\eqref{eq:Pi1-comp},
  and inclusion $\theta : A \to \fil(\topo{X})$.
  We view the filter completion of $A$ as sitting inside the
  $\Pi_1$-completion, just like $\kfil(\topo{X}) \subseteq \fil(\topo{X})$.
  Then we already know that $\Box^{\FiltComp}c = \rbox{R_A}c$ for all
  closed filters, wherefore
  $$
    \Box^{\PiComp}d
      = \bigGenFil \big\{ \Box^{\FiltComp}c \mid c \in \FiltComp(A), c \subseteq d \big\}
      = \bigGenFil \big\{ \rbox{R_A}c \mid c \in \kfil(\topo{X}), c \subseteq d \big\}
      = \rbox{R}d
  $$
  for all filters $d$. We use Lemma~\ref{lem:Rc} for the last equality.
  Similarly we find $\Diamond^{\PiComp}d = \rdiamond{R_A}d$.
\end{proof}

  Proposition~\ref{prop:mod-comp4} immediately implies:

\begin{corollary}
  The filter and $\Pi_1$-completion of a modal lattice are
  modal lattices themselves.
\end{corollary}

  Next, we work towards a preservation theorem.

\begin{theorem}\label{thm:mod-pres-filt}
  Let $\amo{A} = (A, \Box, \Diamond)$ be a modal lattice and
  $\phi, \psi \in \lan{L}_{\Box\!\Diamond}$.
  Then
  $$
    \amo{A} \cdash[] \phi \cp \psi
    \quad\text{iff}\quad
    \FiltComp(\amo{A}) \cdash[] \phi \cp \psi.
  $$
\end{theorem}

  The proof of the theorem relies on Lemma~\ref{lem:mod-pres}.
  We use following definition:
  if $\theta_1$ and $\theta_2$ are valuations for $\amo{A}$,
  then we define the valuation $\theta_1 \wedge \theta_2$ by
  $(\theta_1 \wedge \theta_2)(p) := \theta_1(p) \wedge \theta_2(p)$.

\begin{lemma}\label{lem:mod-pres}
  Let $\amo{A} = (A, \Box, \Diamond)$ be a modal lattice,
  $\FiltComp(\amo{A})$ its filter completion,
  and $\sigma$ a valuation of the proposition letters for $\FiltComp(\amo{A})$.
  \begin{enumerate}
    \item \label{it:mod-pres1}
          If $a \in A$ and $\phi \in \lan{L_{\Box\!\Diamond}}$ are such
          that $a \in \sigma(\phi)$, then there exists a valuation
          $\theta$ for $\amo{A}$ such that $\theta(p) \in \sigma(p)$ for
          all $p \in \Prop$ and $\theta(\phi) \leq a$.
    \item \label{it:mod-pres2}
          If $\theta$ is a valuation for $\amo{A}$ such that
          $\theta(p) \in \sigma(p)$ for all $p$ % \in \Prop$,
          then $\theta(\phi) \in \sigma(\phi)$ for all $\phi \in \lan{L_{\Box\!\Diamond}}$.
  \end{enumerate}
\end{lemma}
\begin{proof}
  We start with the first statement. We proceed by induction on the
  structure of $\phi$.
  If $\phi = p \in \Prop$ then we set $\theta(p) = a$ and $\theta(q) = \top$
  for all other $q \in \Prop$.
  If $\phi = \top$ then $\theta(\top) = \top \in \{ \top \} = \sigma(\top)$,
  and if $\phi = \bot$ then $\theta(\bot) = \bot \in A = \sigma(\bot)$.
  
  Suppose $\phi = \phi_1 \wedge \phi_2$.
  If $a \in \sigma(\phi_1 \wedge \phi_2) = \sigma(\phi_1) \genFil \sigma(\phi_2)$
  then there are $a_1 \in \sigma(\phi_1)$ and $a_2 \in \sigma(\phi_2)$
  such that $a_1 \wedge a_2 \leq a$.
  The induction hypothesis then gives valuations $\theta_1, \theta_2$ for
  $\amo{A}$ such that $\theta_1(\phi_1) \leq a_1$ and $\theta_2(\phi_2) \leq a_2$,
  and such that $\theta_1(p) \in \sigma_1(p)$ and $\theta_2(p) \in \sigma_2(p)$
  for all $p \in \Prop$.
  Let $\theta = \theta_1 \wedge \theta_2$.
  Since $\theta_i(p) \in \sigma(p)$ for all $p \in \Prop$, and $\sigma(p)$
  is a filter, we have $\theta(p) \in \sigma(p)$ for all $p \in \Prop$.
  Then $\theta(\phi_i) \leq \theta_i(\phi_i)$ for $i \in \{ 1, 2 \}$,
  and we find
  $\theta(\phi_1 \wedge \phi_2)
    = \theta(\phi_1) \wedge \theta(\phi_2)
    \leq \theta_1(\phi_1) \wedge \theta_2(\phi_2)
    \leq a_1 \wedge a_2 \leq a.
  $
  
  If $\phi = \phi_1 \vee \phi_2$, then
  $a \in \sigma(\phi_1 \vee \phi_2) = \sigma(\phi_1) \cap \sigma(\phi_2)$
  so there are valuations $\theta_1, \theta_2$ such that
  $\theta_1(\phi_1) \leq a$ and $\theta_2(\phi_2) \leq a$,
  and such that $\theta_1(p) \in \sigma_1(p)$ and $\theta_2(p) \in \sigma_2(p)$
  for all $p \in \Prop$.
  Let again $\theta = \theta_1 \wedge \theta_2$. Then just as above
  $\theta(p) \in \sigma(p)$ for all $p \in \Prop$. Moreover, $\theta(\phi_1 \vee \phi_2)
    = \theta(\phi_1) \vee \theta(\phi_2)
    \leq \theta_1(\phi_1) \vee \theta_2(\phi_2)
    \leq a \vee a = a.
  $
  
  Finally, let $\Star \in \{ \Box, \Diamond \}$
  and $\phi = \Star \phi_1$.
  Suppose $a \in \sigma(\Star \phi_1)$. Then
  $$
    a \in \sigma(\Star \phi_1)
      = \Star^{\FiltComp} \sigma(\phi_1)
      = \bigwedge \{ \uparr(\Star b) \mid b \in A, \sigma(\phi_1) \leq \uparr b \}
      = \bigwedge \{ \uparr(\Star b) \mid b \in A, b \in \sigma(\phi_1) \}.
  $$
  The last equality follows from the fact that filters are ordered by reverse
  inclusion, so that $\sigma(\phi_1) \leq \uparr b$ iff
  $\sigma(\phi_1) \supseteq \uparr b$ iff $b \in \sigma(\phi_1)$.
%  Since the filters are ordered by reverse inclusion,
%  $\sigma(\phi_1) \leq \uparr b$ means $\sigma(\phi_1) \supseteq \uparr b$, 
%  so $b \in \sigma(\phi_1)$.
  Since $\bigwedge$ in $\FiltComp(\amo{A})$ is defined as $\bigGenFil$,
  there exist $b_1, \ldots, b_n \in A$ such that $b_i \in \sigma(\phi_1)$
  for all $i \in \{ 1, \ldots, n \}$ and
  $\Star b_1 \wedge \cdots \wedge \Star b_n \leq a$.
  The induction hypothesis implies that there exist valuations $\theta_1, \ldots, \theta_n$
  for $\amo{A}$ such that $\theta_i(\phi_1) \leq b_i$
  and $\theta_i(p) \in \sigma(p)$ for all $p \in \Prop$ and $i \in \{ 1, \ldots, n \}$.
  Since $\Star$ is monotone we find
  $\theta_i(\Star\phi_1) = \Star \theta_i(\phi_1) \leq \Star b_i$ for all $i$.
  Letting $\theta = \theta_1 \wedge \cdots \wedge \theta_n$ gives
  $\theta(\Star\phi_1) \leq \theta_i(\Star\phi_1) \leq \Star b_i$ for all $i$,
  so that $\theta(\Star\phi_1) \leq \Star b_1 \wedge \cdots \wedge \Star b_n \leq a$,
  as desired.
  Since $\sigma(p)$ is a filter and $\theta_i(p) \in \sigma(p)$ for all
  $i \in \{ 1, \ldots, n \}$ we find $\theta(p) \in \sigma(p)$ for all $p \in \Prop$.
  
  We prove the second item by induction on the structure of $\phi$ as well.
  The base cases $\phi = \top, \bot, p \in \Prop$ are routine.
  Suppose $\phi = \phi_1 \wedge \phi_2$. By induction we have
  $\theta(\phi_1) \in \sigma(\phi_1)$ and $\theta(\phi_2) \in \sigma(\phi_2)$.
  So
  $\theta(\phi_1 \wedge \phi_2)
    = \theta(\phi_1) \wedge \theta(\phi_2)
    \in \sigma(\phi_1) \genFil \sigma(\phi_2)
    = \sigma(\phi_1 \wedge \phi_2).
  $
  If $\phi = \phi_1 \vee \phi_2$, then with the same induction hypothesis
  we find $\theta(\phi_i) \leq \theta(\phi_1 \vee \phi_2)$ so
  $\theta(\phi_1 \vee \phi_2) \in \sigma(\phi_i)$ for $i \in \{ 1, 2 \}$.
  Therefore
  $\theta(\phi_1 \vee \phi_2) \in \sigma(\phi_1) \cap \sigma(\phi_2) = \sigma(\phi_1 \vee \phi_2)$.
  Lastly suppose $\Star \in \{ \Box, \Diamond \}$ and $\phi = \Star \phi_1$.
  By the induction hypothesis we have $\theta(\phi_1) \in \sigma(\phi_1)$,
  so $\sigma(\phi_1) \supseteq \uparr \theta(\phi_1)$ and since
  $\FiltComp(\amo{A})$ is ordered by reverse inclusion, it follows from
  the definition of $\Star^{\FiltComp}$ that
  $\Star^{\FiltComp}\sigma(\phi_1) \supseteq \uparr \Star \theta(\phi_1)$,
  hence
  $$
    \theta(\Star\phi_1)
      = \Star\theta(\phi_1)
      \in \Star^{\FiltComp}\sigma(\phi_1)
      = \sigma(\Star\phi_1)
  $$
  as desired.
\end{proof}

  We can now prove the theorem.

\begin{proof}[Proof of Theorem~\ref{thm:mod-pres-filt}]
  We have $\amo{A} \cdash[] \phi \cp \psi$ iff $\amo{A} \cdash \phi \wedge \psi = \phi$,
  so it suffices to focus on equalities. The proposition letters play
  the role of variables, and every valuation of the proposition letters
  to elements in $\amo{A}$
  extends to valuations for $\phi$ and $\psi$ in the obvious way.
  
  So suppose $\amo{A} \cdash[] \phi = \psi$.
  Let $\sigma$ be any valuation for $\FiltComp(\amo{A})$ and $a \in A$.
  We aim to prove that $a \in \sigma(\phi)$ if and only if $a \in \sigma(\psi)$.
  So suppose $a \in \sigma(\phi)$.
  Then by Lemma~\ref{lem:mod-pres}\eqref{it:mod-pres1} we can find a
  valuation $\theta$ for $\amo{A}$ such that $\theta(\phi) \leq a$.
  By assumption $\amo{A} \cdash[] \phi = \psi$, so
  $\theta(\psi) = \theta(\phi) \leq a$,
  hence by Lemma~\ref{lem:mod-pres}\eqref{it:mod-pres2} we find
  $a \in \sigma(\psi)$.
  Similarly we can prove that $a \in \sigma(\psi)$ implies $a \in \sigma(\phi)$
  so that $\sigma(\phi) = \sigma(\psi)$.
  Since $\sigma$ is arbitrary, we conclude $\FiltComp(\amo{A}) \cdash[] \phi = \psi$.
\end{proof}

  We note that the proof of Lemma~\ref{lem:mod-pres} only relies on monotonicity
  of the modalities. Thus it yields an analogue of Theorem~\ref{thm:mod-pres-filt}
  for lattices with monotone operators.

  The arguments from Lemma~\ref{lem:mod-pres} and
  Theorem~\ref{thm:mod-pres-filt} can be dualised to similar results for
  ideal completions.
  The ideal completion of a modal lattice need not give rise
  to a modal lattice, but it does give rise to a lattice with two monotone
  operators, and the modal cases from Lemma~\ref{lem:mod-pres} and
  Theorem~\ref{thm:mod-pres-filt} only rely on monotonicity of the modalities.
  Thus we obtain a similar results for the ideal completion and hence for the
  $\Pi_1$-completion of a modal lattice.

\begin{theorem}\label{thm:preservation-modal-lat}
  Let $\amo{A} = (A, \Box, \Diamond)$ be a modal lattice and
  $\phi, \psi \in \lan{L}_{\Box\!\Diamond}$.
  Then
  $$
    \amo{A} \cdash[] \phi \cp \psi
    \quad\text{iff}\quad
    \FiltComp(\amo{A}) \cdash[] \phi \cp \psi
    \quad\text{iff}\quad
    \PiComp(\amo{A}) \cdash[] \phi \cp \psi.
  $$
\end{theorem}

  We can use this algebraic result to obtain completeness
  for weak positive modal logics.
  Theorem~\ref{thm:preservation-modal-lat} yields the following analogue
  of Lemma~\ref{lem:canonicity}

\begin{lemma}\label{lem:canonicity-modal}
  Any consequence pair $\psi \cp \chi$ of $\lan{L}_{\Box\!\Diamond}$-formulae is
  $\Pi_1$-persistent.
\end{lemma}

\begin{theorem}\label{thm:compl-modal}
  Let $\Gamma$ be a set of consequence pairs.
  Then the logic $\log{L}_{\Box\!\Diamond}(\Gamma)$ is sound and complete with
  respect to the following classes of frames:
  \begin{itemize}
    \item Modal L-spaces validating $\Gamma$;
    \item Principal modal L-frames validating $\Gamma$;
    \item Modal L-frames validating $\Gamma$.
  \end{itemize}
\end{theorem}
\begin{proof}
  Similar to the proof of Theorem~\ref{thm:compl}.
\end{proof}

%--------------------------------------------------------------------------------
\subsection{Sahlqvist Correspondence}\label{subsec:modal-corr}

  We extend the results from Section~\ref{subsec:corr} to obtain
  Sahlqvist correspondence for modal L-frames. Our definition
  of a Sahlqvist consequence pair is closely aligned to Sahlqivst formulae
  from normal modal logic (see e.g.~\cite[Definition~3.51]{BRV01}).
  To account for the additional relation in the definition of a modal L-model,
  we work with a first-order logic with an extra binary relation symbol
  (compared to Section~\ref{subsec:corr}).
  That is, we let $\FOL_2$ be the first-order logic with a unary predicate
  for each proposition letter and two binary predicates
  $S$ (corresponding to the partial order) and $R$ (corresponding to the modal relation).
  We write $\SOL_2$ for the second-order logic with the same predicates
  where we allow quantification over unary predicates.
  Every modal L-model $\mo{M}$ gives rise to a first-order structure
  $\mo{M}^{\circ}$ for $\FOL_2$ in the obvious way,
  and similarly every modal L-frame $\mo{X}$ yields a structure
  $\mo{X}^{\circ}$ where we can interpret $\SOL_2$-formulae with no free predicates.
  We extend the standard translation from Definition~\ref{def:st} to a
  translation $\st_x : \lan{L}_{\Box\!\Diamond} \to \FOL_2$ by adding the
  clauses
  \begin{equation*}
    \st_x(\Box\phi) = \forall y(xRy \to \st_y(\phi)), \qquad
    \st_x(\Diamond\phi) = \exists y(xRy \wedge \st_y(\phi)).
  \end{equation*}
  We then have the following counterpart of Proposition~\ref{prop:model-corr1}.
  %and Corollary~\ref{cor:model-corr}
  
\begin{proposition}\label{prop:st1-modal}
  Let $\mo{M}$ be a modal L-model, $w$ a state in $\mo{M}$ and
  $\phi$ an $\lan{L}_{\Box\!\Diamond}$-formula. Then
  \begin{enumerate}
    \item $\mo{M}, w \Vdash \phi$ iff $\mo{M}^{\circ} \models \st_x(\phi)[w]$;
    \item $\mo{M}, w \Vdash \phi \cp \psi$
          iff $\mo{M}^{\circ} \models \st_x(\phi \cp \psi)[w]$.
  \end{enumerate}
\end{proposition}

  Defining the second-order translation of a modal consequence pair
  $\phi \cp \psi$ as in Definition~\ref{def:fil-sot},
  we can extend Lemma~\ref{lem:sot} to the next lemma:
  
\begin{lemma}
  For all modal L-frames $\mo{X} = (X, \leq, R)$ and modal 
  consequence pairs $\psi \cp \chi$,
  $$
    \mo{X} \Vdash \psi \cp \chi
      \iff \mo{X}^{\circ} \models \so(\psi \cp \chi).
  $$
\end{lemma}
%\begin{proof}
%  Similar to the proof of Lemma~\ref{lem:sot}.
%\end{proof}

  Finally, we still have monotonicity of all connectives of
  $\lan{L}_{\Box\!\Diamond}$, so the following analogue of
  Lemma~\ref{lem:increasing-val} goes through without problems.

\begin{lemma}\label{lem:increasing-val-modal}
  Let $\mo{X}$ be a modal L-frame and let $V$ and $V'$ be valuations for $\mo{X}$
  such that $V(p) \subseteq V'(p)$ for all $p \in \Prop$.
  Then for all $\phi \in \lan{L}_{\Box\!\Diamond}$ we have $V(\phi) \subseteq V'(\phi)$.
\end{lemma}

%\blue{Where does Sahlqvist correspondence for principal frames go?}

  We are now ready to define Sahlqvist consequence pairs and prove
  a correspondence result. We make use of the following notion of
  a boxed atom.

\begin{definition}
  A \emph{boxed atom} is a formula of the form
  $$
    \Box^n p := \underbrace{\Box \cdots \Box}_{\text{$n$ times}} p,
  $$
  where $p$ is a proposition letter.
  A \emph{Sahlqvist antecedent} is a formula made from boxed atoms,
  $\top$ and $\bot$ by freely using $\wedge$, $\vee$ and $\Diamond$.
  A modal consequence pair $\phi \cp \psi$ is called \emph{Sahlqvist} if
  $\phi$ is a Sahlqvist antecedent (and $\psi$ is any formula).
\end{definition}

  If $R$ is a relation, then we write $R^n$ for the $n$-fold composition
  of $R$. That is, $xR^ny$ if there exist $x_0, \ldots, x_{n}$ such that
  $x = x_0$, $y = x_{n}$ and $x_iRx_{i+1}$ for all $i \in \{ 0, \ldots, n-1 \}$.
  With this definition, truth of $\Box^n p$ in a modal L-model
  $\mo{M} = (X, \leq, R, V)$ can be given as
  $$
    \mo{M}, x \Vdash \Box^n p
      \iff \forall y \in X, xR^ny \text{ implies } \mo{M}, y \Vdash p.
  $$
  We legislate $xR^0y$ iff $x = y$. Then the interpretation of $\Box^0p$
  simply coincides with $p$.

\begin{theorem}\label{thm:modal-corr}
%  Let $\psi \in \lan{L}_{\Box\!\Diamond}$ be a Sahlqvist antecedent,
%  and let $\chi \in \lan{L}$ be any formula.
%  Then $\psi \cp \chi$ locally corresponds to a first-order formula on
%  frames that is effectively computable from the sequent.
  Any Sahlqvist modal consequence pair $\psi \cp \chi$ locally corresponds to
  a first-order formula with one free variable.
\end{theorem}
\begin{proof}
  By Lemma ... we have $\mo{X}, w \Vdash \psi \cp \chi$ if and only if
  $\mo{X}^{\circ} \models \so(\psi \cp \chi)[w]$.
  As in Theorem~\ref{thm:corr}, our strategy for obtaining a first-order
  formula is to remove all second-order quantifiers from $\so(\psi \cp \chi)[w]$.
  We assume that no two quantifiers bind the same variable.
  The case where $\psi = \top$ or $\bot$ can be handled as in
  Theorem~\ref{thm:corr}. 
  Let $p_1, \ldots, p_n$ be the propositional variables occurring in
  $\psi$, and write $P_1, \ldots, P_n$ for their corresponding unary predicates.
  We assume that every proposition letter that occurs in $\chi$
  also occurs is $\psi$, for otherwise we may replace it by $\bot$ to
  obtain a formula which is equivalent in terms of validity on frames.
  
  \bigskip\noindent
  {\it Step 1.}
  Use equivalences of the form
  $$
    (\exists w(\alpha(w)) \wedge \beta) \leftrightarrow \exists w(\alpha(w) \wedge \beta),
    \qquad
    (\exists w(\alpha(w)) \vee \beta) \leftrightarrow \exists w(\alpha(w) \vee \beta),
  $$
  and
  $$
    (\exists w (\alpha(w)) \to \beta) \leftrightarrow\forall w(\alpha(w) \to \beta)
  $$
  to pull out all existential quantifiers that arise in $\st_x(\psi)$.
  Let $Y := \{ y_1, \ldots, y_m \}$ denote the set of (bound) variables that arise
  in the antecedent of the implication from the second-order translation.
  We end up with a formula of the form
  $$
%    \forall P_1 \cdots \forall P_n \forall x \forall y_1 \cdots \forall y_m
%    \big( (\underbrace{\isfil(P_1) \wedge \cdots \wedge \isfil(P_n)}_{\ISFIL}
%    \wedge \, \bar{\psi}) \to \st_x(\chi) \big).
    \forall P_1 \cdots \forall P_n \forall y_1 \cdots \forall y_m
      \big( (\ISFIL \wedge \AT \wedge \REL) \to \st_x(\chi)\big)
  $$
  where
  \begin{itemize}
    \item $\ISFIL$ is a conjunction of formulae of the form $\isfil(P_i)$;
    \item $\BOXAT$ is a conjunction of standard translations of boxed proposition letters,
          i.e.~formulae of the form $\forall z(z'R^nz \to P_iz)$
          (here $P_iz$ is viewed as $\forall z(z'R^0z \to P_iz)$);
    \item $\REL$ is a conjunction of formulae of the form $zRz'$ and $\abovemeet(z;, z', z'')$.
  \end{itemize}
  %Here $z, z', z'' \in Y \cup \{ x \}$.
%    \item boxed atoms: formulae of the form $\forall z(z'R^nz \to P_iz)$
%          (where $P_iz$ falls under this umbrella as $\forall z(z'R^0z \to P_iz)$);
%    \item top and bottom: formulae of the form $(x = x)$ and $(x \neq x)$;
%    \item relations of the form $zRz'$; and
%    \item formulae of the form $\abovemeet(z; z', z'')$
%  \end{itemize}
%  by using $\wedge$ and $\vee$, where $z, z', z'' \in Y \cup \{ x \}$.

%  \bigskip\noindent
%  {\it Step 1B.}
%  Use distributivity (of first-order classical logic) to pull out the
%  disjunctions from $\ISFIL \wedge \bar{\psi}$.
%  That is, we rewrite $\ISFIL \wedge \bar{\psi}$ as a (finite) disjunction
%  \begin{align*}
%    \ISFIL \wedge \, {\bar{\psi}}
%      = \bigvee \Big( \ISFIL \wedge \BOXAT \wedge \REL )
%  \end{align*}
%  where $\BOXAT$ contains atoms and boxed atoms and
%  $\REL$ contains relations of the form $zRz'$ and
%  $\abovemeet(z; z', z'')$.
%
%  \bigskip\noindent
%  {\it Step 1C.}
%  Finally, use equivalences of the form
%  $$
%    ((\alpha \vee \beta) \to \gamma) \leftrightarrow
%    ((\alpha \to \gamma) \wedge (\beta \to \gamma))
%  $$
%  and
%  $$
%    \forall \ldots (\alpha \wedge \beta) \leftrightarrow
%    ((\forall \ldots \alpha) \wedge (\forall \ldots \beta))
%  $$
%  to rewrite $\so(\phi \cp \psi)$ into a conjunction of formulae of the form
%  \begin{equation}\label{eq:post-process}
%    \forall P_1 \cdots \forall P_n \forall y_1 \cdots \forall y_m
%    \big( \ISFIL \wedge \BOXAT \wedge \REL \to \chi \big).
%  \end{equation}
  
  \noindent
  {\it Step 2.} Next we read off minimal instances of the $P_i$ making the
  antecedent true.
%  Intuitively, these correspond to the smallest valuations
%  for the $p_i$ making the antecedent true.
%
  For each proposition letter $P_i$, let $\forall y_{i_1}(z_{i_1}R^{n_1}y_{i_1} \to P_iy_{i_1}), \ldots, \forall y_{i_k}(z_{i_k}R^{n_k}y_{i_k} \to P_iy_{i_k})$ be
  the occurrences of $P_i$ in $\BOXAT$ in the antecedent of \eqref{eq:post-process}.
  Intuitively, we define the valuation of $p_i$ to be the filter generated
  by the (interpretations of) $y_{i_1}, \ldots, y_{i_k}$.
  Formally,
  \begin{align*}
    \sigma(P_i) := 
      \bigvee \Big\{ &\exists w_{i_1} \cdots \exists w_{i_k}
        \big(z_{i_1}R^{n_1}w_{i_1} \wedge \cdots \wedge z_{i_k}R^{n_k}w_{i_k}  \\
        &\wedge \abovemeet(u; w_{i_1}, \ldots, w_{i_k})\big) \mid
         \{ i_1, \ldots, i_k \} \Big\}.
  \end{align*}
  (If $k = 0$, i.e.~there are no boxed atoms involving $P_i$ in the formula,
  then we take the empty meet to be falsum, i.e.~$\sigma(P_i) = \lambda u.(u \neq u)$.)
  %This can occur because these $P_i$ may be used in a different disjoint.
%  \jg{
%  If there are no $R^{n_{i_j}}$-successors of $z_{i_j}$ then the corresponding
%  boxed atom is vacuously true, so it does not affect the ``interpretation'' of
%  $P_i$. In order to reflect this in the expression of $\sigma(P_i)$,
%  we take the join over all subsets of $\{ i_1, \ldots, i_k \}$.
%  \todo{THis cannot happen because of seriality.}}
  
  The remainder of the proof is analogous to the proof of
  Theorem~\ref{thm:corr}.
\end{proof}

  In the next example we apply the algorithm of the proof of
  Theorem~\ref{thm:modal-corr} to two simple consequence pairs,
  $p \cp \Diamond p$ and $\Box p \cp p$.
  The shows the mechanism of the proof in action.
  Moreover, it demonstrates that the duality between $\Box$ and $\Diamond$
  is weaker than in the classical case, because the formulae
  locally correspond to different frame conditions.
  
  In the correspondents, we write $R$ for the modal relation and $\fleq$
  for the poset order. Technically this should be $S$, which is then
  interpreted as $\fleq$.
  Besides, note that $\abovemeet(u; z)$ is the same as $zSu$, which
  we denote by $z \fleq u$.
  Similarly, $\abovemeet(u; z, z')$ means $z \fmeet z' \fleq u$.

\begin{example}
  The second-order translation of $p \cp \Diamond p$ is
  $$
    \forall P(\isfil(P) \wedge Px \to \exists y(xRy \wedge Py))
  $$
  This is already of the desired shape, so we proceed to Step 2.
  We find $\sigma(P) = \lambda u.x \fleq u$. Substituting this gives
  the first-order formula
  $\forall x(\isfil(P) \wedge (x \fleq x) \to \exists y(xRy \wedge (x \fleq y)))$.
  The antecedent of the formula is always true, so the
  (simplified) local correspondent of $p \cp \Diamond p$ is
  $$
    \exists y (xRy \wedge x \fleq y).
  $$
  Thus, a frame satisfies $p \cp \Diamond p$ if
  $\forall x \exists y (xRy \wedge x \fleq y)$.
\end{example}

\begin{example}
  Next consider $\Box p \cp p$.
  The second-order translation is
  $$
    \forall P \forall x(\isfil(P) \wedge \forall y(xRy \to Py) \to Px).
  $$
  Then $\sigma(P) = \lambda u. \exists y(xRy \wedge y \fleq u)$.
  Instantiating this makes the antecedent of the outer implication
  vacuously true, so that we get the local correspondent
  $
    \exists y (xRy \wedge y \fleq x).
  $
  Validity of $\Box p \cp p$ on a frame then corresponds to
  $\forall x \exists y (xRy \wedge y \fleq x)$.
\end{example}

  Next, we use Theorem~\ref{thm:modal-corr} to enforce normality for
  the diamond operator.
  It follows from Lemma~\ref{lem:some-val} that
  $\Diamond p \vee \Diamond q \cp \Diamond(p \vee q)$ is valid in
  every modal L-frame, so we focus on its converse.
%
%  We arrive at the frame condition closely related to the one identified
%  in~\cite[Definition~5.1.4]{Dmi21}.

\begin{example}\label{exm:normal-diamond}
  If we want $\Diamond$ to preserve binary joins we need to add  
  \begin{equation}\label{eq:cp-normal-diamond}
    \Diamond(p \vee q) \cp \Diamond p \vee \Diamond q
  \end{equation}
  to our logic. This is Sahlqvist, so we can use the algorithm
  from Theorem~\ref{thm:modal-corr} to the first-order frame condition
  ensuring its validity.
  The second-order translation is
  $$
    \forall P \big(
        \isfil(P) \wedge \isfil(Q) \wedge
        \exists y(xRy \wedge
        (\exists z \exists z'(\abovemeet(y; z, z') \wedge Pz \wedge Qz')))
        \to \st_x(\Diamond p \vee \Diamond q) \big).
  $$
  Processing the formula, we obtain the following second-order translation:
  \begin{equation}\label{eq:exm-normal-diamond}
    \forall P \forall Q \forall x \forall y \forall z \forall z'
      (\ISFIL \wedge xRy \wedge \abovemeet(y;, z, z') \wedge Pz \wedge Qz' \to \st_x(\chi))
  \end{equation}
  This gives
  $\sigma(P) = \lambda u. z \fleq u$, $\sigma(Q) = \lambda u. z' \fleq u$.
  The standard translation of the antecedent is
  $$
    \st_x(\chi)
      = \exists v \exists v'(\abovemeet(x; v, v')
            \wedge \exists w(vRw \wedge Pw)
            \wedge \exists w'(v'Rw' \wedge Qw')).
  $$
  Substitution $P$ and $Q$ and omitting trivial terms
  we obtain the following local correspondent:
  $$
    \forall y \forall z \forall z'\big((xRy \wedge (z\fmeet z' \fleq y)) \to \exists v \exists v'((v \fmeet v' \fleq x)
            \wedge \exists w(vRw \wedge z \fleq w)
            \wedge \exists w'(v'Rw' \wedge z' \fleq w'))\big).
  $$
%  $$
%    [\sigma(P)/P, \sigma(Q)/Q]\st_x(\chi)
%      = \exists v \exists v'(\abovemeet(x; v, v')
%            \wedge \exists w(vRw \wedge z \fleq w)
%            \wedge \exists w'(v'Rw' \wedge z' \fleq w'))
%  $$
  In a picture:
    $$
    
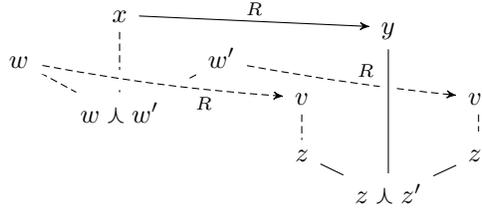
\begin{tikzcd}
        & [-1.3em]
          x \arrow[rrrd, "R"]
            \arrow[dddd, dashed, -]
        & [-1.4em]
        & [-1em]
        & [-1.5em]
        & [-1.5em] \\ [-2.5em]
        &
        &
        &
        & y
        & \\ [-2.2em]
      w     \arrow[ddr, dashed, -]
        &
        & w' \arrow[ddl, dashed, -]
            \arrow[rrrd, bend left=-2, dashed, "R"]
        &&& \\ [-1.9em]
        &
        &
        & v \arrow[dd, dashed, -]
               \arrow[lllu, bend left=3, dashed, <-, crossing over, "R" pos=.25]
        &
        & v' \arrow[dd, dashed, -] \\ [-2.6em]
        & w \fmeet w'
        &&&& \\ [-1.8em]
        &
        &
        & z \arrow[dr, -]
        &
        & z' \arrow[dl, -] \\ [-1.7em]
        &
        &
        &
        & z \fmeet z'  \arrow[uuuuu, crossing over, -]
        &
    \end{tikzcd}
  $$
  A modal L-frame satisfies normality of $\Diamond$ if this
  holds for all states $x$.
  It is a relational analogue of the second condition
  from Definition~\ref{def:L-mor} that ensures preservation of joins.
\end{example}

%================================================================================
\section{Conclusion}

  We have given a new duality for bounded (not necessarily distributive) lattices
  which resembles Stone-type dualities.
  It builds on a known duality for the category of bounded
  meet-semilattices given by Hofmann, Mislove and Stralka~\cite{HofMisStr74}.
  The relation between our duality and the duality by Hofmann, Mislove and
  Stralka is similar to the relation between Esakia duality and Priestley duality.
  It can also be seen as a Stone-type analogue of
  Jipsen and Moshier's spectral duality for lattices~\cite{MosJip09}.
  
  We also extended the duality to one for a modal extension
  of weak positive logic with $\Box$ and $\Diamond$.
  While these are interpreted using a relation in the way as in normal
  modal logic over a classical base, the non-standard interpretation of joins
  causes $\Diamond$ to no longer be join-preserving.
  This interesting phenomenon has also been observed in the context of modal
  intuitionistic logic~\cite{Koj12}.
  
  As the  dualities presented in this
  paper resemble known dualities used in modal logic, 
  they allow us to use similar tools and techniques.
  To showcase this, we proved $\Pi_1$-persistence and Sahlqvist correspondence  results
  along the lines of~\cite{BRV01}.

  There are many intriguing avenues for further research, some of which we
  list below.
  \begin{description}
%    \item[Algebraic characterisation of the double filter completion.]
%          As noted at the end of Section~\ref{subsec:completion}, it is not
%          yet known if and how we can characterise the double filter extension
%          algebraically.
%          While it follows from the definition that it is compact, it seems
%          likely that the double filter completion is only ``half dense.''
    \item[Finite model property.]
          While it is easy to derive the finite model property for weak positive
          logic, the same result for the modal extension
          presented in this paper appears to be non-trivial.
    \item[Relation to ortho(modular) lattices.]
          Ortholattices and orthomodular lattices provide other interesting classes of
          %(not necessarily distributive)
          lattices with operators. However, in  
          ortholattices the orthocomplement is turning joins into meets. Duality for these structures has been discussed by Goldblatt \cite{Gol74,Gol75} and 
          Bimbo \cite{Bimbo07}.
          In \cite[Chapter~6]{Dmi21} the duality for lattices is extended
          to account for a modal operators that turn joins into meets. 
          %These are called $\nabla$-algebras there. %A distributive lattice analogue of these operations has been studied in \cite{GNV05}. 
          Recently modal ortholattices have been studied in~\cite{HM22}.
          We leave it an an interesting open problem to see whether the preservation and correspondence  results 
          of this paper can be extended to this setting.
          It is also open whether these technique could be extended to orthomodular lattices \cite{kal83}. 
          This is especially interesting as orthomodalur lattices provide
          algebraic structures of quantum logic \cite{de2014quantum}, so
          these methods could be relevant in the study of quantum logic.
    \item[Different modalities.]
          Yet another question is what other modal extensions of weak positive
          logic we can define. In particular, it would be
          interesting to define a form of neighbourhood semantics based on the
          L-frames used in this paper and investigate the behaviour of the
          resulting modalities.
  \end{description}

\noindent
{\it Acknowledgements} 
We are very grateful to the referee for many significant comments, deep insights, and pointers to the literature that made us rethink 
and improve on a number of important components of this work.

%================================================================================
\bibliographystyle{plain} 
\bibliography{lattice-biblio.bib}

\end{document}